\documentclass[12pt, reqno]{amsart}

\usepackage{amsmath,amsthm,amsfonts,amscd,amssymb,bbm,enumerate}
\usepackage{cases}

\usepackage{times,euler,eucal,graphicx}
\usepackage{hyperref}
\usepackage[linecolor=white,backgroun dcolor=white,bordercolor=white,textsize=tiny]{todonotes}
\UseRawInputEncoding

\usepackage{graphicx}
\def\pprec{\mathrel{\scalebox{.9}[1]{$\prec$}\mkern-5mu%
  \scalebox{.4}[1]{$\prec$}\mkern-5.5mu\scalebox{.4}[1]{$\prec$}}}

\DeclareMathOperator*{\comp}{\mathbb{C}}

\usepackage{caption} 
\usepackage{subcaption}

\usepackage{IEEEtrantools}

\usepackage{color}

\renewcommand\labelenumi{(\roman{enumi})}
\renewcommand\theenumi\labelenumi

\numberwithin{equation}{section}
\newtheorem{theorem}{Theorem}[section]
\newtheorem{lemma}[theorem]{Lemma}
\newtheorem{corollary}[theorem]{Corollary}
\newtheorem{proposition}[theorem]{Proposition}
\newtheorem{assumption}{Assumption}

\theoremstyle{definition}
\newtheorem{definition}[theorem]{Definition}
\newtheorem{notation}[theorem]{Notation}
\newtheorem{example}[theorem]{Example}
\newtheorem{remark}[theorem]{Remark}

\renewcommand{\epsilon}{\varepsilon}

\newcommand{\Imm}{{\Im}m}

\newcommand{\A}{\mathcal{A}}
\newcommand{\B}{\mathcal{B}}
\newcommand{\C}{\mathbb{C}}
\newcommand{\D}{\mathcal{D}}
\newcommand{\E}{\mathbb{E}}
\newcommand{\F}{\mathcal{F}}
\newcommand{\G}{\mathcal{G}}
\newcommand{\bH}{\mathbb{H}}

\newcommand{\bN}{\mathbb{N}}

\newcommand{\R}{\mathbb{R}}

\newcommand{\uy}{\mathbf{y}}
\newcommand{\uz}{\mathbf{z}}
\newcommand{\uu}{\mathbf{u}}
\newcommand{\uv}{\mathbf{v}}

\newcommand{\ux}{\mathbf{x}}

\newcommand{\Zb}{\mathfrak{b}}

\newcommand{\h}{\textbf{h}}

\newcommand{\id}{\operatorname{id}}

\newcommand{\tr}{\operatorname{tr}}

\newcommand{\1}{\mathbf{1}}

\makeatletter
\def\moverlay{\mathpalette\mov@rlay}
\def\mov@rlay#1#2{\leavevmode\vtop{%
\baselineskip\z@skip \lineskiplimit-\maxdimen
\ialign{\hfil$#1##$\hfil\cr#2\crcr}}}
\makeatother

\makeatletter
\def\@settitle{\begin{center}%
  \baselineskip14\p@\relax
    \normalfont \Large \uppercase{ \textbf{\@title}}
  \end{center}
 }

\makeindex

\title[On Boolean and Monotone Limit Theorems]{Quantitative Estimates for Operator-Valued and Infinitesimal  Boolean and Monotone Limit Theorems }

\author[O. Arizmendi]{Octavio Arizmendi}
\address{CIMAT, Guanajuato, Mexico}
\email{octavius@cimat.mx}

\author[M. Banna]{Marwa Banna}
\address{New York University Abu Dhabi, Division of Science, Mathematics, Abu Dhabi, UAE}
\email{marwa.banna@nyu.edu}

\author[PL. Tseng]{Pei-Lun Tseng}
\address{New York University Abu Dhabi, Division of Science, Mathematics, Abu Dhabi, UAE}
\email{pt2270@nyu.edu}

\date{\today}

\thanks{...}
\keywords{ Lindeberg method, noncommutative distributions, Boolean independence, Monotone Independence, Berry-Esseen bounds, operator-valued CLT, operator-valued matrices, L\'evy distance}
\subjclass[2000]{46L54, 60B10, 60B20}

\begin{document}

\begin{abstract}
We provide Berry-Esseen bounds for sums of operator-valued Boolean and monotone independent variables, in terms of the first moments of the summands. Our bounds are on the level of Cauchy transforms as well as the L\'evy distance. As applications, we obtain quantitative bounds for the corresponding CLTs, provide a quantitative "fourth moment theorem" for monotone independent random variables including the operator-valued case, and generalize the results by Hao and Popa on matrices with Boolean entries.  Our approach relies on a Lindeberg method that we develop for sums of Boolean/monotone independent random variables.  Furthermore, we push this approach to the infinitesimal setting to obtain the first quantitative estimates for the operator-valued infinitesimal free, Boolean and monotone CLT.
\end{abstract}

\maketitle

\section{Introduction}
Noncommutative probability theory deals with operators from a probabilistic viewpoint. Like any probabilistic theory, one of its main notions is that of independence. A special feature of noncommutative probability is however the diversity of such notions, where five of them are considered to be the most fundamental following the classification of natural products of Muraki \cite{muraki2003five} from the categorical axiomatization viewpoint of Ben Ghorbal and Sch\"urmann \cite{ghorbal2002non}. Other than the well studied tensor independence and free independence, the are three other notions satisfying such axioms: Boolean independence, monotone independence  and its mirror image antimonotone independence.
Operator-valued (OV) and operator-valued infinitesimal (OVI)  extensions of these different notions of independence also exist, where independence is defined with respect to some conditional expectation and some completely positive linear $\B$-bimodule map respectively, see Sections \ref{section:prelOVPT} and \ref{section:prel-OVI}.

In this paper, we are mainly concerned with operator-valued Boolean and monotone independence.  We are interested in studying distributions of sums of independent variables $x_1+\dots+x_n$, for which we prove Berry-Esseen results that are further extended to the operator-valued infinitesimal setting. One should note that  Boolean and monotone independences have the special feature that constants are not independent from any non-trivial algebra $\A$. Thus adapting  strategies for free independence, as  for example in \cite{BannaMaiBerry}, is not direct. For instance, when working with Cauchy transforms one has the problem that (Boolean or monotone) independence of $X$  from $Y$, does not imply the independence of $X$ from $(z-Y)^{-1}$, for a complex number $z$.  Thus  in order to make precise estimates of joint moments one needs a more detailed analysis.

Our main results, stated in Section \ref{section:main results}, have various applications including quantitative bounds for the operator-valued Boolean and monotone central limit theorems. These bounds are on the level of the operator-valued Cauchy transforms as well as the L\'evy distance and are merely in terms of the moments of fourth and second order, see Sections \ref{section:BooleanCLT} and \ref{section:monotoneCLT}. As further applications, we provide in Section \ref{section:infinitelydivisiblemeasures} a quantitative "fourth moment theorem" for the class of infinitely divisible measures with respect to the monotone convolution including the operator-valued case.  Our approach also generalizes the results by Hao and Popa \cite{popahao2019} on matrices with Boolean entries to include general variance profiles, see Section \ref{section:matrices}. 
Finally, as our results extend to the infinitesimal setting, we provide in Section \ref{section:Inf-CLT} the first quantitative estimates for the operator-valued infinitesimal free, Boolean and monotone central limit theorems. Moreover, we provide in Appendix \ref{Appendix:OVI} an algebraic construction of the operator-valued infinitesimal product in the free, Boolean and monotone settings.   

Our approach relies on the Lindeberg method \cite{Lindeberg} which is a powerful method that allows comparing distributions of sums of independent variables.  The idea is to compare the distribution of $x_1+\dots+x_n$ with that of $y_1+\dots+y_n$, by keeping track of the effect of changing $x_i$ by $y_i$, one at a time. This method has been used in noncommutative probability previously, in the context of random matrices, free and exchangeable random variables; see the papers \cite{bannacebron2018,BannaMaiBerry,chatterjee2006}.

\subsection{Statement of results}\label{section:main results}

Let $(\A, E, \B)$ be an operator-valued $C^\ast$-probability space  and let  $x=\{x_1, \dots, x_N\}$ and $y=\{y_1, \dots, y_N\}$ be two self-adjoint families in $\A$ whose elements are infinitesimally free, Boolean, or monotone independent over $\B$. Our aim is to compare the analytic distribution, in terms of Cauchy transforms, of 
\[
\ux_N=\sum_{i=1}^N x_i
\quad \text{and} \quad
\uy_N=\sum_{i=1}^N y_i,
\]
whenever the variables are centered and have matching moments of second order. More precisely, we assume the following:
\begin{assumption}\label{A:general}
The two families $\{x_1, \dots, x_N\}$ and $\{y_1, \dots, y_N\}$ consist of self-adjoint elements in $\A$  such that for any $j=1, \dots , N$, 
\begin{itemize}
\item $ E [ x_j]=E[ y_j]=0$,
\item $E \big[ x_j  b \,  x_j  \big]= E \big[ y_j\,  b \, y_j \big],$ for all $b \in \B$.
\end{itemize}
\end{assumption}
To state our main estimate we need to introduce first some notation on operator-valued moments
\[
\alpha_2(x) = \max_{1\leq i\leq N} \|E[x_i^2 ] \|, \quad  \alpha_4(x):=\max_{1\leq i\leq N}\sup \|E[x_i  b^* x_i^2 b x_i ] \|
\quad \text{and}\quad  
\widetilde{\alpha}_4(x) = \max_{1\leq i\leq N} \|E[x_i^4 ] \|,
\]
 where the supremum is taken over $b \in \B$ such that $ \|b\|=1$.
\begin{theorem}[Boolean Case]\label{theo:boolean-monotone}
Let $(\A, E, \B)$ be an operator-valued $C^\ast$-probability space. Let $N \in \bN$ and consider two families  $x=\{x_1, \dots, x_N\}$ and $y=\{y_1, \dots, y_N\}$ of self-adjoint elements in $\A$ that are Boolean independent and satisfy Assumption \ref{A:general}. Then for any $\Zb \in \bH^+(\B)$,
\begin{align}\label{eq1:theo-BM}
\big\| E[G_{\ux_N}(\Zb)] - E[G_{\uy_N}(\Zb)] \big\|
\leq \|\Imm(\Zb)^{-1}\|^4 \sqrt{\alpha_2(x)}\left(\sqrt{\alpha_4(x)}+\sqrt{\alpha_4(y)}\right) \ N .\end{align}
If, in addition, $(\A,\varphi)$ is a $W^\ast$-probability space with $\varphi=\varphi \circ E$, we have that, for any $\epsilon>0$,
\begin{equation}\label{eq2:theo-BM}
\int_\R \big|\Imm \big(\varphi[G_{\ux_N}(t+i\epsilon)]\big)- \Imm \big(\varphi[G_{\uy_N}(t+i\epsilon)]\big)\big| \text{d}t \leq \frac{\pi}{\epsilon^3} \sqrt{\alpha_2(x)} \left(\sqrt{\widetilde{\alpha}_4(x)} + \sqrt{\widetilde{\alpha}_4(y)}\right) N. \end{equation}
\end{theorem}
Let $(\A,E,\B)$ be an operator-valued $C^*$-probability space and $a$ be a element in $\A$. We denote by $\A_a$  the $\B$-bimodule algebra generated by $a$. For given $x_1,\dots,x_k\in \A$, we write $\A_{x_1} \prec \A_{x_2} \prec \dots \prec \A_{x_k}$ over $\B$, whenever $i<j$ and $\A_{x_i}$ is monotone independent from $\A_{x_j}$ over $\B$. 

\begin{theorem}[Monotone Case]\label{theo:monotone-monotone}
Let $(\A, E, \B)$ be an operator-valued $C^\ast$-probability space. Let $N \in \bN$  and consider two families  $x=\{x_1, \dots, x_N\}$ and $y=\{y_1, \dots, y_N\}$ of self-adjoint elements in $\A$  satisfying Assumption \ref{A:general} and that are such that $\A_{x_1} \prec \A_{x_2} \prec \dots \prec \A_{x_N}$ over $\B$ and  $\A_{y_1} \prec \A_{y_2} \prec \dots \prec \A_{y_N}$  over $\B$. Then for any $\Zb \in \bH^+(\B)$,
\begin{align}\label{eq1:theo-monotone}
\big\| E[G_{\ux_N}(\Zb)] - E[G_{\uy_N}(\Zb)] \big\|
\leq  \|\Imm(\Zb)^{-1}\|^4 \sqrt{\alpha_2(x)}\left(\sqrt{\alpha_4(x)}+\sqrt{\alpha_4(y) }\right) \ N .
\end{align}

In the case of a $W^\ast$-probability space $(\A, \varphi)$, where $\B=\mathbb{C}$ and $E=\varphi$, we have in addition for any $\epsilon>0$,
\begin{equation}\label{eq2:theo-monotone}
\int_\R \big|\Imm \big(\varphi[G_{\ux_N}(t+i \epsilon)]\big)- \Imm \big(\varphi[G_{\uy_N}(t+i\epsilon)]\big)\big| \text{d}t \leq \frac{\pi}{\epsilon^3} \sqrt{\alpha_2(x)}\left(\sqrt{ \widetilde{\alpha}_4(x)} + \sqrt{\widetilde{\alpha}_4(y)}\right) N. 
\end{equation}
\end{theorem}

The above estimates allow us to quantify the CLT for sums of Boolean or monotone independent random variables with amalgamation in terms of the second and fourth moment; see Section \ref{Section 4} for the precise statements. More precisely, the bounds \eqref{eq1:theo-BM} and \eqref{eq1:theo-monotone} allow us to prove convergence in distribution over $\B$, while the bounds \eqref{eq2:theo-BM} and \eqref{eq2:theo-monotone} yield estimates on the L\'evy distance. The method of proof is also adapted to provide estimates on the distribution of matrices with Boolean entries; see Section \ref{section:matrices}.

Suppose that $(\A, E,E', \B)$ is an OV $C^*$- infinitesimal probability space, and let $x=\{x_1,\dots,x_N\}$ and $y=\{y_1,\dots,y_N\}$ be two self-adjoint families in $\A$. 
\begin{assumption}\label{A:Infgeneral}
The two families $x$ and $y$ consist of self-adjoint elements in $\A$ such that for any $j=1,\dots,N$, 
\begin{itemize}
\item $ E [ x_j]=E[ y_j]=E'[x_j]=E'[y_j]=0$,
\item $E \big[ x_j  b \,  x_j  \big]= E \big[ y_j\,  b \, y_j \big]$ and $E' \big[ x_j  b \,  x_j  \big]= E' \big[ y_j\,  b \, y_j \big],$ for all $b \in \B$.
\end{itemize}
\end{assumption}

\begin{theorem}[Infinitesimal Case]\label{eq:theo-Inf}
Let $(\A,E,E',\B)$ be an OV $C^*$-infinitesimal probability space. Let $N\in \mathbb{N}$ and consider two infinitesimally independent families $x=\{x_1,\dots,x_N\}$ and $y=\{y_1,\dots,y_N\}$ of self-adjoint elements in $\A$ that are infinitesimally freely/Boolean/monotone independent satisfy Assumption \ref{A:Infgeneral}. Then for each $b\in \mathbb{H}^+(\B)$,  
\begin{equation}\label{eq:theo-Inf Eqn}
\|E'\big[G_{\ux_N}  (\Zb)\big]  - E' \big[G_{\uy_N}(\Zb)\big]\|\leq 2N \|\Imm(\Zb)^{-1}\|^4  \left( \max_{1\leq i\leq N}\|x_i\|^3  + \max_{1\leq i\leq N}\|y_i\|^3 \right). 
\end{equation}
\end{theorem}
Here $x$ and $y$ are infinitesimally monotone independent in the sense that $\A_{x_1}\pprec \A_{x_2}\pprec \dots \pprec \A_{x_N}$ and $\A_{y_1}\pprec \A_{y_2}\pprec \dots \pprec \A_{y_N}$ over $\B$, see Section \ref{section:prel-OVI} for more details on the order.
\section{Preliminaries}

\subsection{OV Probability Theory}\label{section:prelOVPT}
Operator-valued probability theory allows considering conditional expectations and gives a wider range of applicability. To define such spaces, we start by introducing some notions and recalling some definitions. Let $\A$ be a unital algebra over $\comp$ and $\B$ a unital subalgebra of $\A$.  We say that $\mathcal{C}$ is a subalgebra of $\A$ over $\B$ if $\mathcal{C}$ is a subalgebra of $\A$ and $bc ,cb \in \mathcal{C}$, for all $b \in \B$ and $c \in \mathcal{C}$. A subalgebra of $\A$ over $\B$ may not contain the unit of $\A$. 

For $x_1,\dots,x_r \in \A$, let $\B\langle x_1,\dots,x_r \rangle_0$ denote the subalgebra of $\A$ over $\B$ consisting of finite sums of elements of $\{b_1x_{i_1}b_2\dots x_{i_n}b_{n+1}: b_i \in \B, n \geq 1, i_1,\dots,i_n \in \{1,\dots, r \}\}$.  Note that, in general, $\B$ is not contained in $\B\langle x_1,\dots,x_r \rangle_0$. 

Let $\mathcal{D}$ be another unital algebra containing $\B$ as a subalgebra. 
A map $f$ from $\A$ to $\mathcal{D}$ is called $\B$-linear if $f(b_1 x b_2 + y) = b_1 f(x) b_2 + f(y),$ for all  $b_1,b_2 \in \B$ and $x,y \in \A$. A $\B$-linear map $E$ with values in $\B$ is called a \textit{conditional expectation} if $E(b) = b,$ for all $b \in \B$. From now on we assume that $E$ is a conditional expectation in the above sense.  A \textit{random variable} is an element of $\A$ and a \textit{random vector} is an element of $\A^r$ for any $r \geq 1$.

A $W^\ast$-probability space is a pair $(\A, \varphi)$ where $\A$ is a von Neumann algebra and $\varphi$ is a faithful normal state on $\A$.
$(\A,E,\B)$ is called an \emph{operator-valued $C^*$-probability space} if $\A$ is a unital $C^*$-algebra, $\B$ is a unital subalgebra, and $E:\A\to\B$ is a conditional expectation that is completely positive. Working in such spaces allows conditioning with respect to $E$ which carries many similar properties as the scalar-valued expectation $\varphi$. In addition, if $(\A,\varphi)$ is a tracial $W^\ast$-probability space and $\B$ is a von Neumann subalgebra of $\A$, then there exists a unique conditional expectation $E:\A\to \B$ that preserves the trace, i.e. $\varphi\circ E=\varphi$. 
Note that in the tracial setting, the existence of a trace-preserving conditional expectation $E$ was proved in \cite{umegaki1954conditional}. As in the Boolean and monotone cases, the state $\varphi$ is non-tracial, we will only consider the settings where there exists a conditional expectation $E$ with $\varphi\circ E=\varphi$.

A main difference, however, is that notions of positivity for operator-valued distributions need to hold for all matrix amplifications. 
It is also needed to extend the different notions of independence to the operator-valued setting, where independence is defined with respect to the conditional expectation $E$. This was done over time, where the notion of OV free independence was first introduced in \cite{voiculescu1995operations}. On the other hand, Boolean independence, introduced by Speicher and Woroudi \cite{speicher1997boolean}, and monotone independence, introduced by Muraki \cite{Muraki97,muraki2000monotonic,muraki2001monotonic}, were extended later on to the operator-valued setting (see \cite{popa2009new,popa2008combinatorial,hasebe2014operator,skeide2004independence}). 
We recall these notions of independence in the following definition. 

\begin{definition}\label{defn:indep}
Given an OV $C^*$-probability space $(\A,E,\B)$. 
\begin{itemize}
    \item Sub-algebras $(\A_i)_{i\in I}$ of $\A$ that contain $\B$ are called \emph{
freely independent (or just free for short) over $\B$} if for $i_1,i_2,\dots,i_n\in I$, $i_1\neq i_2 \neq i_3\dots\neq i_n$, and $x_j\in \A_{i_j}$ with $E[x_{j}]=0$ for all $j=1,2,\dots,n$, then
\begin{eqnarray*}
E[x_1\cdots x_n] =0.
\end{eqnarray*}
    \item (The possibly non-unital) $\B$-bimodule subalgebras $(\A_i)_{i\in I}$ of $\A$ are called \emph{Boolean independent over $\B$} if for all $n\in\mathbb{N}$ and $x_1,\dots, x_n\in\mathcal{A}$ such that $x_j\in\mathcal{A}_{i_j}$ where $i_1\neq i_2\neq \dots \neq i_n\in I$, we have
\begin{eqnarray*}
E[x_1\cdots x_n] = E[x_1]\cdots E[x_n].
\end{eqnarray*}
    \item 
Assume that $\Lambda$ is equipped with a linear order $<$. (The possibly non-unital) $\B$-bimodule subalgebras $(\A_i)_{i\in \Lambda}$ of $\A$ are called \emph{monotone independent over $\B$} if 
\begin{eqnarray*}
E[x_1\cdots x_j \cdots x_n] = E[x_1\cdots x_{j-1}E[x_j]x_{j+1}\cdots x_n]
\end{eqnarray*}
whenever $x_j\in\A_{i_j}, i_j\in \Lambda$ for all $j$ and $ i_{j-1} < i_j > i_{j+1}$, where one of the inequalities is eliminated if $j=1$ or $j=n$.
\end{itemize}
\end{definition}
A crucial difference between monotone independence from the free and Boolean independence is the lack of symmetry, so when talking about monotone tuples of variables we need to specify an order.  Therefore, we say subalgebras $\A_{\alpha}$ and $\A_{\beta}$ are monotone independent in the sense that they are monotone independent with the order $\alpha<\beta$ and we  denote it by $A_{\alpha}\prec A_{\beta}.$ 

Given an OV probability space $(\A,E,\B)$, elements $(x_i)_{i \in I}$ of $\mathcal{A}$ are said to be freely independent over $\mathcal{B}$ if the unital algebras generated by elements $x_i$ over $\mathcal{B}$ are freely independent. Similarly, $(x_i)_{i\in I}$ are said to be Boolean (respectively monotone) independent if the non-unital $\B$-bimodule subalgebras generated by $x_i, (i\in I)$ over $\mathcal{B}$ form a Boolean (respectively monotone) independent family. Since the notion of monotone independence depends on the order, we  write $x\prec y$ or $\A_x \prec \A_y$ whenever $x,y$ are monotone independent.

As in the scalar, noncommutative joint distributions are defined in the operator-valued realm as the collection of all possible $\B$-valued mixed moments. 
\begin{definition}[$\B$-valued joint distributions]
Let $X=(x_1,\dots, x_r)$ be a random vector in an OV $C^*$-probability space $(\A,E,\B)$  and let $i_1,\dots, i_{k} \in \{1,\dots, r\}$ for $n \geq 1$. The multilinear map $\mu^X_{i_1,\dots,i_{k}}$ defined by 
\[
\mu^X _{i_1,\dots,i_{k}}(b_1, \dots, b_{k-1}) = E(x_{i_1}b_1\cdots b_{k-1}x_{i_{k}}) 
\]
is called the $(i_1,\dots, i_{k})$-moment of $X$. Then the \emph{ $\B$-distribution} of $X$ is defined as the collection of all possible $(i_1,\dots,i_k)$-moments of $X$:
\[
\mu^X=\big\{ \mu^X _{i_1,\dots,i_{k}}(b_1, \dots, b_{k-1}) \ | \ i_1,\dots, i_{k} \in \{1,\dots, r\}, b_1, \dots, b_{k-1} \in \B
\big\}.
\]
\end{definition}  
Moreover in the OV setting, convergence in distribution means convergence in norm of all joint $\B$-valued moments. More precisely, for each $n\in \mathbb{N}$, $X_n=(x_n^{(i)})_{i \in I} \in \A_n$ and $X=(x_i)_{i \in I} \in \A$, we say $(X_n)_n$ converges to $X$ in distribution over $\B$ if for any $k\geq 0$, $i_1,\dots, i_{k} \in \{1,\dots, r\}$ and $b_0,b_1,\dots,b_k\in \B$, 
$$
\|\mu^{X_n} _{i_1,\dots,i_{k}}(b_1, \dots, b_{k-1})-\mu^X _{i_1,\dots,i_{k}}(b_1, \dots, b_{k-1})\|\longrightarrow 0 \quad \text{ as } \quad n \to\infty.
$$
Let us now introduce the notion of Cauchy transforms, which is a crucial tool to study distributions. Suppose that $(\A,E,\B)$ is an OV $C^*$-probability space. The resolvent of an element $x\in \A$, is given by $G_x(b)=(b-x)^{-1}$ for any $b\in \B$ such that $b-x$ is invertible. If $x=x^*\in \A$, then the \emph{OV Cauchy transform of $x$} is given by
$$
\G_x^{\B}(b):= E[G_x(b)] \quad \text{ for all }b\in \bH^+(\B),
$$
where $\bH^+(\B)$ denotes the OV upper half plane of $\B$ defined by
$$
\bH^+(\B):=\big\{ b\in \B \mid \Imm(\Zb)=\frac{1}{2i}(b-b^*)>0 \big\}.
$$
Note that for each $x=x^*\in \A$, $b-x$ is invertible for all $b\in \bH^+(\B)$ over which $\G_x^{\B}$ is hence well-defined. If $\B=\C$, then we have the (scalar-version) \emph{Cauchy transform of $x$} 
$$
\G_x(z):= \varphi[G_x(z)] \quad \text{ for all }z\in \C^+,
$$
where $\C^+$ denotes the upper half complex plane. Note that 
${\G_x}$ encodes all the moments of $x$; indeed, for any $z\in \C^+$ such that $|z|<\|x\|$, we have 
$$\G_x(z)=\sum\limits_{n=1}^{\infty}\frac{\varphi(x^n)}{z^{n+1}}.$$
For later use, we note that for a given $x=x^*\in \A$, we obtain by the resolvent identity 
$$
\|G_x(z)\|_{L^2(\A,\varphi)}^2 =\varphi\left[(\bar{z}-x)^{-1}(z-x)^{-1}\right]=-\frac{\Imm(\G_x(z))}{\Imm(z)} \qquad \text{for all }z\in\mathbb{C}^+,
$$
which together with the fact that $-\frac{1}{\pi} \Im(\G_x(t+i\epsilon))\, \mathrm{d} t$ is a probability measure for every $\epsilon>0$ yields that
$$
\int_{\mathbb{R}}\|G_x(t+i\varepsilon)\|_{L^2(\A,\varphi)}^2 dt =\frac{\pi}{\varepsilon} \qquad \text{for all }\varepsilon>0.
$$
However in the OV setting, the OV Cauchy transform $\G_x^{\B}$ alone does not encode all possible moments of $x$ and one would need for this purpose all the matrix amplifications $\{\G_{1_k \otimes x}^{M_k(\B)}\}_{k\in \mathbb{N}}$ of $\G_x^{\B}$ defined for any $k \in \mathbb{N}$ by
$$\G_{1_k \otimes x}^{M_k(\B)}(b):= id_k \otimes E[G_{1_k\otimes x}(b)] \quad  \text{ for all }b\in M_k(\B) \text{ such that }b-1_k \otimes x \text{ is invertible. } 
$$ 
For more details on this matter, see \cite{voiculescu1995operations}.

The combinatorial approach for OV free probability theory was first introduced by Speicher \cite{speicher1998combinatorial}. The OV Boolean case was later developed by Popa \cite{popa2009new} while the OV monotone case was investigated in \cite{hasebe2014operator,popa2008combinatorial}.   We review now  notions and properties of OV cumulants, which play a key role in the combinatorial counterpart of OV non-commutative probability theory. 

We denote by $NC(n)$ the set of non-crossing partitions of $[n]:=\{1,\dots,n\}$, namely, partitions $\pi=\{V_1,\dots,V_r\}$ such that $V_i$ and $V_j$ do not cross for any $1\leq i,j \leq r$. The blocks of $\pi$ are $V_1,\dots,V_r$ and the size (in this case $r$) of the partition is denoted by $|\pi|$. We denote furthermore by $NC_2(n)$ the set of non-crossing pair partitions of $[n]$. In order to describe cumulants for various notions of independence, we give further notation:
\begin{notation}\label{notation-f_pi}
Suppose that $\A$ is a unital algebra, and $\B$ is a unital sub-algebra of $\A$. Let $\{f_n^{\B}: \A^n \to \B\}_n$ be a family of multilinear maps. Given a partition $\pi \in NC(n)$, we define the map $f_{\pi}^{\B}:\A^n\to \B$ recursively as follows: 
\begin{itemize}
    \item For $\pi=1_n:= \{1,\dots,n\}$, define $f_{\pi}^{\B}:=f_n^{\B}$. \\
    \item For $\pi\in NC(n)\setminus \{1_n\}$, pick up an interval block $V=\{l+1,\dots,l+k\}\in\pi$ and define 
    \begin{eqnarray*}
    f_{\pi}^{\B}(x_1,\dots,x_n) &=& f_{\pi'}^{\B}(x_1,\dots,x_lf_k^{\B}(x_{l+1},\dots,x_{l+k}),x_{l+k+1},\dots,x_n) \\
    &=& 
    f_{\pi'}^{\B}(x_1,\dots,x_l,f_k^{\B}(x_{l+1},\dots,x_{l+k})x_{l+k+1},\dots,x_n)
    \end{eqnarray*}
    where $\pi'=\pi\setminus V \in NC(n-k)$. 
\end{itemize}
\end{notation}

\begin{definition}
Let $(\A,E,\B)$ be an OV $C^*$-probability space. The \emph{OV free cumulants} $\{r_n^{\B}:\A^n\to \B\}_n$, the \emph{OV Boolean cumulants} $\{\beta_n^{\B}:\A^n\to \B\}_n$, and the \emph{OV monotone cumulants} $\{h_n^{\B}:\A^n\to \B\}_n$ are recursively defined for all $n\in \mathbb{N}$ and $x_1,\dots,x_n\in \A$ via the following moment-cumulant formulas: 
    \begin{eqnarray*}
    E[x_1\dots x_n] &=& \sum\limits_{\pi\in NC(n)} r_{\pi}^{\B} (x_1,\dots,x_n);  \\
    E[x_1\dots x_n] &=& \sum\limits_{\pi\in I(n)} \beta_{\pi}^{\B} (x_1,\dots,x_n); \\
    E[x_1\dots x_n] &=& \sum\limits_{\pi\in NC(n)} \frac{1}{\tau(\pi)!} h_{\pi}^{\B} (x_1,\dots,x_n); 
    \end{eqnarray*}
    where $I(n)$ denotes the subset of $NC(n)$ consisting of all interval partitions of $[n]$ and
    $$
    \tau(\pi)!:= \prod_{V\in \pi} \Big|\big\{W\in \pi \mid W\subseteq \{min(V),\dots,max(V)\}\big\}\Big|.
    $$
\end{definition}
A crucial property of OV free and Boolean cumulants is the fact that they capture independence over $\B$. Indeed, the notions of free and Boolean independence can be characterized by the property of vanishing mixed cumulants as described below:
\begin{proposition}\label{Vanishing Prop}
Suppose that $(\A,E,\B)$ is an OV $C^*$-probability space. 
\begin{itemize}
    \item Unital subalgebras $\A_1,\dots,\A_n$ that contain $\B$ are free over $\B$ if and only if for $n\geq 2$ and $i_1,\dots,i_s\in [n]$ which are not all equal and for $x_1\in \A_{i_1},\dots,x_s\in \A_{i_s}$, we have $r_s^{\B}(x_1,\dots,x_s)=0.$ \\
    \item $\B$-bimodule subalgebras $\A_1,\dots,\A_n$ are Boolean independent over $\B$ if and only if for $n\geq 2$ and $i_1,\dots,i_s\in [n]$ which are not all equal and for $x_1\in \A_{i_1},\dots,x_s\in \A_{i_s}$, we have $\beta_s^{\B}(x_1,\dots,x_s)=0.$ \\
\end{itemize}
\end{proposition}
In the free and Boolean case, it is easy to see - using the above property of vanishing of mixed cumulants - that independence is preserved when lifted to matrices.  In the following proposition, we prove directly and without relying on cumulants that the statement remains valid in the case of monotone independence. 
\begin{proposition}\label{prop:lift-monotoneIND-to-matrices}
Let $(\A,E,\B)$ be an OV $C^*$-probability space and $\A_1,\dots,\A_n$ be subalgebras over $\B$. If $\A_1\prec A_2\prec \cdots \prec \A_n$ over $\B$, then for each $N\in \mathbb{N}$, $M_N(\A_1)\prec M_N(\A_2) \prec \cdots \prec  M_N(\A_n)$ over $M_N(\B)$ in the OV $C^*$-probability space $(M_N(\A),id_N\otimes E,M_N(\B))$. 
\end{proposition}
\begin{proof}
Let $A_m=[a^{(m)}_{r,s}]_{r,s=1}^N\in M_N(\A_{i_m})$ for each $m$ where $i_1,\cdots,i_k\in [n]$ with $i_1\neq i_2\neq \cdots \neq i_n$. Suppose that $i_{k_0-1}<i_{k_o}>i_{k_0+1}$, then we observe that for any $i,j\in [N]$,
\begin{eqnarray*}
&&\Big[(id_N\otimes E)[A_1A_2\cdots A_{k_{0}-1}A_{k_0}A_{k_0+1}\cdots A_n] \Big]_{i,j} \\
&=& \sum_{i_1,\cdots,i_{n-1}\in [N]}E\big[a_{i,i_1}^{(1)}a_{i_1,i_2}^{(2)}\cdots a^{(k_0-1)}_{i_{k_0-2},i_{k_0-1}}a^{(k_0)}_{i_{k_0-1},i_{k_0}}a^{(k_0+1)}_{i_{k_0},i_{k_0+1}}\cdots a^{(n)}_{i_{n-1},j} \big] \\
&=& \sum_{i_1,\cdots,i_{n-1}\in [N]}E\big[a_{i,i_1}^{(1)}a_{i_1,i_2}^{(2)}\cdots a^{(k_0-1)}_{i_{k_0-2},i_{k_0-1}}E\big[a^{(k_0)}_{i_{k_0-1},i_{k_0}}\big]a^{(k_0+1)}_{i_{k_0},i_{k_0+1}}\cdots a^{(n)}_{i_{n-1},j} \big] \\ 
&=& E\Big[ \sum_{i_1,\cdots,i_{n-1}\in [N]}a_{i,i_1}^{(1)}a_{i_1,i_2}^{(2)}\cdots a^{(k_0-1)}_{i_{k_0-2},i_{k_0-1}}E\big[a^{(k_0)}_{i_{k_0-1},i_{k_0}}\big]a^{(k_0+1)}_{i_{k_0},i_{k_0+1}}\cdots a^{(n)}_{i_{n-1},j} \Big] \\
&=& \Big[(id_N\otimes E)\big[A_1A_2\cdots A_{k_{0}-1}(Id_N\otimes E)\left[A_{k_0}\right]A_{k_0+1}\cdots A_n\big] \Big]_{i,j}, 
\end{eqnarray*}
which implies that $$(id_N\otimes E)[A_1\cdots A_{k_{0}-1}A_{k_0}A_{k_0+1}\cdots A_n] = (id_N\otimes E)[A_1\cdots A_{k_{0}-1}(Id_N\otimes E)[A_{k_0}]A_{k_0+1}\cdots A_n], $$
 from which we conclude that monotone independence over $\B$ lifts to the matrix level. 
\end{proof}

The development of OV probability spaces and different notions of independence led naturally to proving OV counterparts of the central limit theorem and studying their limiting laws. As in the scalar setting, the limiting distributions are universal in the sense that they only depend on the initial distribution of the variables through their variance. Note that for a $\B$-valued element $x$ with $E[x]=0$, the variance is given by the completely positive map $\eta_x:\B\to\B$, $\eta_x(b):=E[xbx]$ rather than a positive scalar $\sigma^2_x$. The free case and its OV semicircular limiting distribution were first proven in \cite{voiculescu1995operations} and \cite{speicher1998combinatorial} followed by the Boolean and monotone cases with their OV Bernoulli and arcsine limiting distributions in \cite{BVP-OV-SAB} and \cite{hasebe2014operator}. Before recalling their definitions precisely, we refer the reader to \cite[Chapter 6]{Jekel-notes} for more discussion around the OV central limit theorems. 

\begin{definition}
Let $(\A,E,\B)$ be an OV $C^*$-probability space and  $x\in \A$ with $E[x]=0$. We say that $x$ is
\begin{itemize}
    \item  a \emph{$\B$-valued semi-circular element} with variance $\eta$ if 
    \begin{align*}
        E[b_0xb_1\dots xb_k] = \begin{cases} \sum\limits_{\pi\in NC_2(k)}b_0 \eta_{\pi}(b_1,\dots,b_{k-1})b_k & \text{ if }k \text{ is even }  \\ 0 & \text{ if }k \text{ is odd.}\end{cases} 
    \end{align*}
    \item a \emph{$\B$-valued Bernoulli element} with variance $\eta$ if 
     \begin{align*}
        E[b_0xb_1\dots xb_k] = \begin{cases} b_0 \eta(b_1)b_2\dots\eta(b_{k-1})b_k & \text{ if }k \text{ is even }  \\ 0 & \text{ if }k \text{ is odd.}\end{cases} 
     \end{align*}   
     \item   a \emph{$\B$-valued arcsine element} with variance $\eta$ if
     \begin{align*}
        E[b_0xb_1\dots xb_k] = \begin{cases} \sum\limits_{\pi\in NC_2(k)}\frac{1}{\tau(\pi)!}b_0\eta_{\pi}(b_1,\dots,b_{k-1})b_k & \text{ if }k \text{ is even }  \\ 0 & \text{ if }k \text{ is odd.}\end{cases} 
    \end{align*}
\end{itemize}
where 
$\eta_\pi: \B^{k-1} \rightarrow \B$, $\eta_\pi(b_1,\dots, b_{k-1}) = E_\pi [xb_1,xb_2, \dots, xb_{k-1},x]$.
\end{definition}
Note that by Notation \ref{notation-f_pi}, $E_\pi$ breaks for a pairing $\pi \in NC_2(k)$ the product  into pairs of the form $E[xbx]=\eta_x(b)$, where $\pi$ just indicates the way in which way the pairs are nested.
\begin{remark}
Note that $x$ is a $\B$-valued semicircular (respectively Bernoulli, arcsine) element with variance $\eta$ if and only if the OV free (respectively Boolean, monotone) cumulants $\kappa_n^{\B}$ are given by 
\begin{equation}\label{CLT:cumulants}
\kappa_n^{\B}(xb_1,xb_2,\dots,xb_{n-1},x)= \begin{cases}\eta(b_1) &\text{ if } n=2 , \\ 0 &\text{ if }n\neq 2. \end{cases} 
\end{equation}
We refer to \cite{BVP-OV-SAB} and \cite[Chapter 6]{Jekel-notes} for more details.
\end{remark}

\begin{theorem} 
Suppose that $(\A,E,\B)$ is an OV $C^*$-probability space, and $\{x_i\}_{i\in \mathbb{N}}$ is a sequence of identically distributed freely (respectively Boolean, monotone) independent elements that are centered, $E[x_1]=0$, and whose variance is given by a completely positive map $\eta$. For each given $N$, we let 
$$
S_N=\frac{1}{\sqrt{N}}(a_1+\dots+a_N). 
$$
Then $(S_N)_N$ converges to $s$ in distribution over $\B$,  
where $s$ is a $\B$-valued semi-circular (respectively Bernoulli, arcsine) element with variance $\eta$.
\end{theorem}

\subsection{OV Infinitesimal Probability}\label{section:prel-OVI}

One of the generalizations of OV free probability is OV infinitesimal (OVI) free probability that was first studied in \cite{curran2011asymptotic}. Notions of OVI Boolean independence and OVI monotone independence were later developed in \cite{perales2021operator}. In this subsection, we will review the framework of various notions of infinitesimal independence.    

For a given OV $C^*$-probability space $(\A,E,\B)$, if we have an additional self-adjoint linear $\B$-bimodule map $E':\A\to \B$ that is completely bounded with $E'(1)=0,$ then $(\A,E,E',\B)$ is called an \emph{OV $C^*$-infinitesimal probability space}. 

\begin{definition}\label{defn:indep2}
Given an OV $C^*$-infinitesimal probability space $(\A,E,E',\B)$. 
\begin{itemize}
\item 
Sub-algebras $(\A_i)_{i\in I}$ of $\A$ that contain $\B$ are called \emph{
infinitesimally freely independent (or just infinitesimally free for short) over $\B$} if 
for $i_1,i_2,\dots,i_n\in I$, $i_1\neq i_2 \neq i_3\dots\neq i_n$, and $x_j\in \A_{i_j}$ with $E[x_{j}]=0$ for all $j=1,2,\dots,n$, then
\begin{eqnarray*}
E[x_1\dots x_n] &=&0; \\
E'[x_1\dots x_n] &=& \sum\limits_{j=1}^n E[x_1\dots x_{j-1}E'[x_j]x_{j+1}\dots x_n]. 
\end{eqnarray*}
    \item 
(The possibly non-unital) $\B$-bimodule subalgebras $(\A_i)_{i\in I}$ of $\A$ are called \emph{infinitesimally Boolean independent over $\B$} if for all $n\in\mathbb{N}$ and $x_1,\dots, x_n\in\mathcal{A}$ such that $x_j\in\mathcal{A}_{i_j}$ where $i_1\neq i_2\neq \dots \neq i_n\in I$, we have
\begin{eqnarray*}
E[x_1\dots x_n] &=& E[x_1]\dots E[x_n]; \\
E'[x_1\dots x_n] &=& \sum_{j=1}^n E[x_1]\dots E[x_{j-1}]E'[x_j]E[x_{j+1}]\dots E[x_{n}]. 
\end{eqnarray*}
    \item 
Assume that $\Lambda$ is equipped with a linear order $<$. (The possibly non-unital) $\B$-bimodule subalgebras $(\A_i)_{i\in \Lambda}$ of $\A$ are called \emph{infinitesimally monotone independent over $\B$} if 
\begin{eqnarray*}
E[x_1\dots x_j \dots x_n] &=& E[x_1\dots x_{j-1}E[x_j]x_{j+1}\dots x_n];\\
E'[x_1\dots x_j \dots x_n] &=& E'[x_1\dots x_{j-1}E[x_j]x_{j+1}\dots x_n]+E[x_1\dots x_{j-1}E'[x_j]x_{j+1}\dots x_n].
\end{eqnarray*}
whenever $x_j\in\A_{i_j}, i_j\in \Lambda$ for all $j$ and $ i_{j-1} < i_j > i_{j+1}$, where one of the inequalities is eliminated if $j=1$ or $j=n$.
\end{itemize}
\end{definition}
 Note that the notion of infinitesimal monotone independence is similar to the case of monotone independence, which is lack of symmetry. In this context, $\A_{\alpha}$ and $\A_{\beta}$ are infinitesimally monotone independent if they are infinitesimally monotone independent with the order $\alpha<\beta$. We write in this case $\A_{\alpha}\pprec \A_{\beta}$. 

The elements $(x_i)_{i\in I}$ of $\A$ are infinitesimally free if the unital algebras generated by $x_i, (i\in I)$ are infinitesimally free over $\B$. In addition, $(x_i)_{i\in I}$ are infinitesimally Boolean (respectively monotone) independent if the non-unital $\B$-bimodule subalgebras generated by $x_i, (i\in I)$ over $\mathcal{B}$ form a infinitesimally Boolean (respectively monotone) independent family.
\begin{definition}
Let $X=(x_1,\dots, x_r)$ be a random vector and $i_1,\dots, i_{k} \in \{1,\dots, r\}$ for $n \geq 1$. Then the $(i_1,\dots, i_{k})$-infinitesimal moment of $X$ is the multilinear map $\partial \mu^X_{i_1,\dots,i_k}$ that is defined by
$$
\partial \mu^X_{i_1,\dots,i_k}(b_1,\dots,b_{k-1})=E'(x_{i_1}b_1\dots b_{k-1}x_{i_{k}}). 
$$
The \emph{infinitesimal distribution of X} is the collection of all possible $(i_1,\dots, i_{n})$-moments and infinitesimal moments of $X$. 
\end{definition}   
Moreover in the OV-infinitesimal setting, convergence in infinitesimal distribution means convergence in norm of all joint $\B$-valued moments and  $\B$-valued infinitesimal moments. More precisely, for each $n\in \mathbb{N}$, $X_n=(x_n^{(i)})_{i \in I} \in \A_n$ and $X=(x_i)_{i \in I} \in \A$, we say $(X_n)_n$ converges to $X$ in infinitesimal distribution over $\B$ if for any $k\geq 0$, $i_1,\dots, i_{k} \in \{1,\dots, r\}$ and $b_0,b_1,\dots,b_k\in \B$,
\[
\begin{aligned}
&\|\mu^{X_n} _{i_1,\dots,i_{k}}(b_1, \dots, b_{k-1})-\mu^X _{i_1,\dots,i_{k}}(b_1, \dots, b_{k-1})\|\longrightarrow 0, \quad \text{ and } \\
&\|\partial\mu^{X_n} _{i_1,\dots,i_{k}}(b_1, \dots, b_{k-1})-\partial\mu^X _{i_1,\dots,i_{k}}(b_1, \dots, b_{k-1})\|\longrightarrow 0 \quad \text{ as } \ n\to\infty.
\end{aligned}\]

 Suppose that $(\A,E,E',\B)$ is an OV $C^*$-infinitesimal probability space, and let $x=x^*\in \A$. The \emph{OV  infinitesimal Cauchy transform of $x$} is defined by 
$$
\partial \G_x^{\B}(b):=E'[G_x(b)] \quad  \text{ for all }b\in \B \text{ such that }b-x \text{ is invertible. } 
$$
Note that the OV infinitesimal moments of $x$ are obtained by considering all matrix amplifications $\{\partial \G_{1_k \otimes x}^{M_k(\B)}\}_{k\in\mathbb{N}}$ of $\partial \G_x^{\B}$, defined for any $k \in \mathbb{N}$ by
$$\G_{1_k \otimes x}^{M_k(\B)}(b):= id_k\otimes E[G_{1_k\otimes x}(b)] \quad  \text{ for all }b\in M_m(\B) \text{ such that }b- 1_k\otimes x \text{ is invertible. } 
$$
For more details, we refer the reader to  \cite{tseng2021thesis}.

Note that for a given OV $C^*$-infinitesimal probability space $(\A,E,E',\B)$, the corresponding upper triangular probability space $(\widetilde{\A},\widetilde{E},\widetilde{\B})$ is an OV probability space where 
$$
\widetilde{\A}=\left\{\begin{bmatrix} x & x' \\ 0 & x
\end{bmatrix} \Big| x,x'\in \A\right\},\ \widetilde{B}=\left\{\begin{bmatrix} b & b' \\ 0 & b \end{bmatrix}  \Big| b,b'\in \B\right\}, 
$$
and $\widetilde{E}$ is a map from $\widetilde{A}$ to $\widetilde{B}$ that is defined by
$$
\widetilde{E}\begin{bmatrix}x & x' \\ 0 & x\end{bmatrix}= \begin{bmatrix} E[x] & E'[x]+E[x'] \\ 0 & E[x]
\end{bmatrix}.
$$
The following proposition provides the connection between $(\A,E,E',\B)$ and $(\widetilde{\A},\widetilde{E},\widetilde{\B})$ (see \cite{perales2021operator,tseng2019unified}). 

\begin{proposition}\label{Inf.Prop}
Let $(\A,E,E',\B)$ be an OV $C^*$-infinitesimal probability space and  $(\widetilde{\A},\widetilde{E},\widetilde{\B})$ be its upper triangular probability space. Then subalgebras $(\A_i)_{i\in I}$ that contain $\B$ are infinitesimally free/Boolean/monotone over $\B$ if and only if $(\widetilde{A}_i)_{i\in I}$ are free/Boolean/monotone over $\widetilde{B}$ respectively, 
where for each $i$,
$
\widetilde{A}_i$ 
and $
\widetilde{B}_i$ are as defined above.
\end{proposition}

Note that $(\widetilde{\A},\widetilde{E},\widetilde{B})$ is an OV Banach non-commutative probability space (that is, $\widetilde{\A}$ is a unital Banach algebra, $\widetilde{\B}$ be a subalgebra with $1\in \widetilde{\B}$ and $\widetilde{E}:\widetilde{\A}\to \widetilde{\B}$ is a linear, bounded, $\widetilde{B}$-bimodule projection) with the norm $\|\cdot\|$ on $\widetilde{\A}$ which is defined by 
$$
\|A\| =\|x\|_{\A}+\|x'\|_{\A} \quad \text{where} \quad A=\begin{bmatrix} x & x' \\ 0 & x \end{bmatrix}.
$$
Before recalling the OVI cumulants, we introduce some further notation.
\begin{notation}
Suppose that $\A$ is a unital algebra, and $\B$ is a unital sub-algebra of $\A$. Let $\{f_n^{\B}:\A^n\to \B\}_n$ and $\{\partial f_n^{\B}:\A^n\to \B\}_n$ be two families of multilinear maps. Given a partition $\pi\in NC(n)$, and a block $V\in \pi$, we define $\partial f_{\pi,V}^{\B}:\A^n\to \B$ to be the map that is equal to $f_{\pi}^{\B}$, as given in Notation \ref{notation-f_pi}, except that for the block $V$, we replace $f_{|V|}^{\B}$ by $\partial f_{|V|}^{\B}$. Then, the map $\partial f_{\pi}^{\B}:\A^n\to \B$ is defined by 
$$
\partial f_{\pi}^{\B}(x_1,x_2,\dots,x_n):= \sum\limits_{V\in \pi} \partial f_{\pi,V}^{\B}(x_1,\dots,x_n). 
$$
\end{notation}

\begin{definition}
Suppose that $(\A,E,E',\B)$ is an OV $C^*$-infinitesimal probability space. We define the
\begin{itemize}
    \item  \emph{OV infinitesimally free cumulants} $\{\partial r_n^{\B}:\A^n\to \B\}_n$ to be families of multilinear maps such that for all $n\in \mathbb{N}$ and $x_1,\dots,x_n\in \A$, 
    $$
     E'[x_1\dots x_n] = \sum\limits_{\pi\in NC(n)} \partial r_{\pi}^{\B} (x_1,\dots,x_n),
    $$
    where $\{r_n^{\B}:\A^n\to \B\}_n$ denotes the OV free cumulants.
    \item  \emph{OV infinitesimally Boolean cumulants} $\{\partial\beta_n^{\B}:\A^n\to \B\}_n$ to be families of multilinear maps such that for all $n\in \mathbb{N}$ and $x_1,\dots,x_n\in \A$, 
    $$
    E'[x_1\dots x_n] = \sum\limits_{\pi\in I(n)} \partial \beta_{\pi}^{\B} (x_1,\dots,x_n),
    $$
    where $\{\beta_n^{\B}:\A^n\to \B\}_n$ denotes the OV Boolean cumulants. 
    \item  \emph{OV infinitesimally monotone cumulants} $\{\partial h_n^{\B}:\A^n\to \B\}_n$ to be families of multilinear maps such that for all $n\in \mathbb{N}$ and $x_1,\dots,x_n\in \A$, 
    $$
    E'[x_1\dots x_n] = \sum\limits_{\pi\in NC(n)} \frac{1}{\tau(\pi)!}\partial h_{\pi}^{\B} (x_1,\dots,x_n),
    $$
    where $\{h_n^{\B}:\A^n\to\B\}_n$ denotes the OV monotone cumulants.
\end{itemize}
\end{definition}
Note that the notions of OV infinitesimally free and Boolean independence still share the property of vanishing  mixed cumulants, see \cite{tseng2019unified} and \cite{perales2021operator}.
\begin{proposition}\label{InfVanishing-prop}
Suppose that $(\A,E,E',\B)$ is an OV $C^*$-infinitesimal probability space. 
\begin{itemize}
    \item Unital subalgebras $\A_1,\dots,\A_n$ that contain $\B$ are infinitesimally free if and only if for $n\geq 2$ and $i_1,\dots,i_s\in [n]$ which are not all equal, and for $x_1\in \A_{i_1},\dots,x_s\in \A_{i_s}$, we have $r_s^{\B}(x_1,\dots,x_s)=\partial r_s^{\B}(x_1,\dots,x_s)=0.$ \\
    \item $\B$-bimodule subalgebras $\A_1,\dots,\A_n$ are infinitesimally Boolean independent if and only if for $n\geq 2$ and $i_1,\dots,i_s\in [n]$ which are not all equal, and for $x_1\in \A_{i_1},\dots,x_s\in \A_{i_s}$, we have $\beta_s^{\B}(x_1,\dots,x_s)=\partial\beta_s^{\B}(x_1,\dots,x_s)=0.$ 
\end{itemize}
\end{proposition}
Finally to recall and address the OV infinitesimal central limit theorems that were also proved in \cite{perales2021operator,tseng2021thesis}, we recall the definition of the associated limiting distributions. 

Let $(\A,E,E',\B)$ be an OV $C^*$-infinitesimal probability space, and $x$ be an element in $\A$ with $E[x]=E'[x]=0$.     

\begin{definition}
Suppose that $(\A,E,E',\B)$ is an OV $C^*$-infinitesimal probability space. Let $x\in \A$ with $E[x]=E'[x]=0$. We denote by $\eta_x':\B\to \B$ the $\B$-valued map given by $\eta_x'(b):=E'[xbx]$ and call the pair $(\eta_x,\eta_x')$, or simply $(\eta,\eta')$, the \emph{infinitesimal variance of x}. We say that $x$ is
\begin{itemize}
    \item  a \emph{$\B$-valued infinitesimal semi-circular element} with infinitesimal variance $(\eta,\eta')$ if $x$ is a semi-circular element with variance $\eta$ and 
    \begin{align*}
        E'[b_0xb_1\dots xb_k] = \begin{cases} \sum\limits_{\pi\in NC_2(k)}b_0\partial \eta_{\pi}(b_1,\dots,b_{k-1})b_k & \text{ if }k \text{ is even, }  \\ 0 & \text{ if }k \text{ is odd.}\end{cases} 
    \end{align*}
    \item  a \emph{$\B$-valued infinitesimal Bernoulli element} with infinitesimal variance $(\eta,\eta')$ if $x$ is a Bernoulli element with variance $\eta$ and 
     \begin{align*}
        E'[b_0xb_1\dots xb_k] = \begin{cases} \sum\limits_{j=0}^{k/2-1}b_0 \eta(b_1)\dots \eta'(b_{2j+1}) \dots\eta(b_{k-1})b_k & \text{ if }k \text{ is even, }  \\ 0 & \text{ if }k \text{ is odd.}\end{cases} 
     \end{align*}
    \item  a \emph{$\B$-valued infinitesimal arcsine element} with infinitesimal variance $(\eta,\eta')$ if $x$ is an arcsine element with variance $\eta$ and 
     \begin{align*}
        E'[b_0xb_1\dots xb_k] = \begin{cases} \sum\limits_{\pi\in NC_2(k)}\frac{1}{\tau(\pi)!}b_0\partial \eta_{\pi}(b_1,\dots,b_{k-1})b_k & \text{ if }k \text{ is even, }  \\ 0 & \text{ if }k \text{ is odd.}\end{cases} 
    \end{align*}
\end{itemize}
\end{definition}

\begin{theorem}\label{OVI-CLT} 
Suppose that $(\A,E,E',\B)$ is an OV $C^*$-infinitesimal probability space, and $\{x_i\}_{i\in \mathbb{N}}$ is a sequence of identically distributed infinitesimally freely (respectively Boolean, monotone) independent elements with infinitesimal variance $(\eta,\eta')$ and are such that $E[x_1]=E'[x_1]=0$. For each given $N \in \mathbb{N}$, we let 
$$
S_N=\frac{1}{\sqrt{N}}(a_1+\dots+a_N). 
$$
Then $(S_N)_N$ converges to $s$ in infinitesimal distribution over $\B$ 
where $s$ is a $\B$-valued infinitesimal semi-circular (respectively Bernoulli, arcsine) element with infinitesimal variance $(\eta,\eta')$.
\end{theorem}
\begin{remark}
By applying Proposition \ref{Inf.Prop} and \cite[Lemma 2.1]{perales2021operator}, we conclude that $x$ is a $\B$-valued infinitesimal semi-circular (respectively Bernoulli, arcsine) element with infinitesimal variance $(\eta,\eta')$ if and only if the OV free (respectively Boolean, monotone) cumulants $(\kappa_n^{\B})_n$ are as in \eqref{CLT:cumulants} and the OVI free (respectively Boolean, monotone) cumulants $(\partial\kappa_n^{\B})_n$ are as follows:
\begin{equation*}\label{ICLT:cumulants}
\partial\kappa_n^{\B}[xb_1,xb_2,\dots,xb_{n-1},x]= \begin{cases} \eta'(b_1) &\text{ if } n=2, \\ 0 &\text{ if }n\neq 2. \end{cases}
\end{equation*}
\end{remark}

Let $x_1$ and $x_2$ be elements in $\A$ such that $E[x_i]=E'[x_i]=0$ for $i=1,2$. Assume that $x_1$ and $x_2$ are $\B$-valued semicircular (respectively Bernoulli) elements with variances $\eta_{x_1}$ and $\eta_{x_2}$ respectively. Then $x_1+x_2$ is a semicircular (respectively Bernoulli) element with variance $\eta_{x_1+x_2}=\eta_{x_1}+\eta_{x_2}$. This easily follows from independence and the property of vanishing mixed moments, (see also \cite{Jekel-notes}). We prove now the infinitesimal analogue as follows. 
\begin{lemma}\label{Lemma.Is.elt}
For each $i=1,2$, assume $x_i$ is an infinitesimal semicircular (respectively Bernoulli) element with infinitesimal variance $(\eta_{x_i},\eta'_{x_i})$. If $x_1$ and $x_2$ are infinitesimally free (respectively Boolean), then $x_1+x_2$ is an infinitesimal semicircular (respectively Bernoulli) element with infinitesimal variance $(\eta_{x_1+x_2},\eta'_{x_1+x_2})=(\eta_{x_1}+\eta_{x_2},\eta'_{x_!}+\eta'_{x_2})$.
\end{lemma}
\begin{proof}
The statement follows from independence and the property of vanishing mixed infinitesimal moments in Proposition \ref{InfVanishing-prop}, which yield directly for any $n\neq 2$ and $b_1, \dots , b_{n-1} \in \B$ that
$$
\partial\kappa_n^{\B}((x_1+x_2)b,\dots,(x_1+x_2)b,(x_1+x_2)) = 0, 
$$
while for $n=2$ and any $b\in \B$, 
\begin{align*}
    \eta'_{x_1+x_2}(b)= \partial\kappa_2^{\B}\big((x_1+x_2)b,(x_1+x_2)\big) &= \partial\kappa_2^{\B}(x_1b,x_1)+\partial\kappa_2^{\B}(x_2b,x_2) \\                                                &= E'[x_1bx_1]+E'[x_2bx_2]=\eta'_{x_1}(b)+\eta'_{x_2}(b). 
\end{align*}
\end{proof}
\begin{remark}
Note that there are no analogue lemmas for the infinitesimal monotone case. To be precise, suppose that $x_1,\dots,x_n$ are infinitesimal monotone independent arcsine elements in $\A$ such that $E[x_i]=E'[x_i]=0$ for $i\in {n}$. It is not guaranteed that $x_1+\cdots+x_n$ is still an infinitesimal arcsine element. We call, in this case, $x_1+\dots+x_n$ an \emph{infinitesimal generalized arcsine element}. 
\end{remark}

\subsection{Positivity of conditional expectations}
 Since the conditional expectation $E$ is positive, it induces a $\B$-valued pre-inner product
$$\langle \cdot,\cdot\rangle:\ \A\times\A \rightarrow \B,\ (x,y) \mapsto E[x^\ast y]$$
with respect to which $\A$ becomes a right pre-Hilbert $\B$-module. In particular, we have the following analogue of the Cauchy-Schwarz inequality:
\begin{equation}\label{Cauchy-Schwarz}
\|E[x^\ast y]\|^2 \leq \|E[x^\ast x]\| \|E[y^\ast y]\|.
\end{equation}
The positivity of $E$ also implies the following important inequality
\begin{equation}\label{positivity}
\| E[x^\ast w x] \| \leq \|w\| \|E[x^\ast x]\|,
\end{equation}
which holds for all $x\in\A^n$ and $w\in M_n(\A)$ satisfying $w\geq 0$.

If $\mu$ and $\nu$ are two Borel probability measures on $\R$, then
\begin{itemize}
 \item the \emph{L\'evy distance} is defined by
$$L(\mu,\nu) := \inf\{\epsilon>0 \mid \forall t\in\R:\ \F_\mu(t-\epsilon) - \epsilon \leq \F_\nu(t) \leq \F_\mu(t+\epsilon) + \epsilon\};$$
 \item the \emph{Kolmogorov distance} is defined by
$$\Delta(\mu,\nu) := \sup_{t\in\R} |\F_\mu(t) - \F_\nu(t)|.$$
\end{itemize}
The L\'evy distance provides a metrization of convergence in distribution and may be bounded in terms of the associated Cauchy transforms of $\mu$ and $\nu$ as follows:
\begin{equation}\label{eq:Levy_bound}
L(\mu,\nu) \leq 2\sqrt{\frac{\epsilon}{\pi}} + \frac{1}{\pi} \int_\R | \Imm(\G_\mu(t+i\epsilon)) - \Imm(\G_\nu(t+i\epsilon)) |\, \mathrm{d} t.
\end{equation}
For a proof of \eqref{eq:Levy_bound}, we refer to Appendix B in \cite{BannaMaiBerry}.

\section{Proof of the main results}
In order to prove our main results, Theorems \ref{theo:boolean-monotone} and \ref{theo:monotone-monotone}, we rely on an operator-valued Lindeberg method, which allows replacing the $x_i$'s by the $y_i$'s, one at a time, and controlling the error of such a replacement. We follow the first lines of the proof of Theorem 3.1 in \cite{BannaMaiBerry} for which we recall the following basic algebraic identity:

\begin{lemma}\label{lem:Taylor_resolvents}
 Let $x$ and $y$ be invertible in some unital complex algebra $\A$, then for each $m\in\mathbb{N}$, the following identity holds:
\begin{align*}
x^{-1} - y^{-1} = \sum^m_{k=1} y^{-1} \big[ \big(y - x\big) y^{-1} \big]^k + x^{-1} \big[ \big(y - x\big) y^{-1} \big]^{m+1}.
\end{align*}
\end{lemma}

\begin{proof}
The proof follows by induction on $m$ and the algebraic identity \[x^{-1} \big[ \big(y - x\big) y^{-1} \big]^{m+1}-y^{-1} \big[ \big(y - x\big) y^{-1} \big]^{m+1}=x^{-1} \big[ \big(y - x\big) y^{-1} \big]^{m+2}.\]
\end{proof}

\begin{proposition}\label{prop:Lindeberg}
Let $(\A, E, \B)$ be an operator-valued $C^\ast$-probability space. Let $N \in \bN$  and consider the families  $x=\{x_1, \dots, x_N\}$ and $y=\{y_1, \dots, y_N\}$ of self-adjoint elements in $\A$. For any $i=1,\dots,N$, set \[
 \uz_i =\sum_{j=1}^i x_j + \sum_{j=i+1}^N y_j
\qquad \text{and} \qquad 
  \uz^0_i =\sum_{j=1}^{i-1} x_j + \sum_{j=i+1}^N y_j.
 \]
Then for any $\Zb \in \bH^+(\B)$,
\begin{align}\label{telescoping-sum}
G_{\uz_N}  (\Zb)  - G_{\uz_0}(\Zb) 
 &= \sum^N_{i=1} \big(A_i(\Zb) + B_i(\Zb) +C_i(\Zb)\big),
\end{align}
where for each $i=1,\dots, N$,
\begin{align*}
A_i (\Zb)&= G_{\uz^0_i}(\Zb) \ x_i \ G_{\uz^0_i}(\Zb) - G_{\uz^0_i}(\Zb) \ y_i \ G_{\uz^0_i}(\Zb),\\
B_i (\Zb) & = G_{\uz^0_i}(\Zb) \big( x_i \  G_{\uz^0_i}(\Zb)\big)^2 - G_{\uz^0_i}(\Zb) \big( y_i \  G_{\uz^0_i}(\Zb)\big)^2, \\
C_i (\Zb) & =G_{\uz_i}(\Zb) \big( x_i \  G_{\uz^0_i}(\Zb)\big)^3 -G_{\uz_{i-1}}(\Zb)\big(  y_i \  G_{\uz^0_i}(\Zb)\big)^3.
\end{align*}

\end{proposition}
\begin{proof}
We start by writing, for any $\Zb \in \bH^+(\B)$, the difference as a telescoping sum:  
\begin{align*}
G_{\uz_N}  (\Zb)  - G_{\uz_0}(\Zb)
 &= \sum^N_{i=1} \big(G_{\uz_i}(\Zb) - G_{\uz_{i-1}}(\Zb)\big) 
= \sum^N_{i=1} \Big(\! \big( G_{\uz_i}(\Zb) -G_{\uz^0_i}(\Zb) \big)
 -  \big( G_{\uz_{i-1}}(\Zb)   -G_{\uz^0_i}(\Zb) \big) \!\Big).
\end{align*}
Noting that $\uz_i-  \uz^0_i = x_i$ and $\uz_{i-1}- \uz^0_i = y_i$, then by applying the algebraic identity in Lemma \ref{lem:Taylor_resolvents} up to order $3$, we get 
\begin{align*}
G_{\uz_i}(\Zb) - G_{\uz^0_i}(\Zb)  
&= G_{\uz^0_i}(\Zb) \ x_i \  G_{\uz^0_i}(\Zb)
 +G_{\uz^0_i}(\Zb) \big( x_i \  G_{\uz^0_i}(\Zb)\big)^2
+G_{\uz_i}(\Zb) \big( x_i \  G_{\uz^0_i}(\Zb)\big)^3 \qquad\text{and}\\
G_{\uz_{i-1}}(\Zb) - G_{\uz^0_i}(\Zb)  
&= G_{\uz^0_i}(\Zb) \ y_i \ G_{\uz^0_i}(\Zb)
+G_{\uz^0_i}(\Zb) \big( y_i \  G_{\uz^0_i}(\Zb)\big)^2
+G_{\uz_{i-1}}(\Zb)\big(  y_i \  G_{\uz^0_i}(\Zb)\big)^3. 
\end{align*}
Putting the above terms together and summing over $i=1,\dots, N$ ends the proof.
\end{proof}
 
\subsection{Boolean Case}\label{section:Boolean} We start with the proof of Theorem \ref{theo:boolean-monotone} relative to Boolean independence. While Boolean independence is simpler than the free case, we still need to treat it closely since Boolean independence is not well-behaved with respect to scalars.

\begin{proof}[Proof of Theorem \ref{theo:boolean-monotone}]
The starting point of the proof is the operator-valued Lindeberg method is Proposition \ref{prop:Lindeberg} where we note that $\ux_N=\uz_N$ and $\uy_N=\uz_0$.

\vspace{0.25cm}
\noindent
{\bf Proof of \eqref{eq1:theo-BM}.} We will prove that $E\big[ A_i (\Zb) \big]=0$, $E\big[ B_i (\Zb) \big]=0$ and that 
\begin{align*}
\|E[C_i(\Zb)]\| \leq \|\Imm(\Zb)^{-1}\|^4 \sqrt{\alpha_2(x)}\left(\sqrt{\alpha_4(x)}+\sqrt{\alpha_4(y)}\right)\ . 
\end{align*}
Considering the telescopic sum in \eqref{telescoping-sum}, we fix $i \in \{ 1,\dots,N\}$ and start by controlling the first order term 
\[
E\big[ A_i (\Zb) \big]= E\big[ G_{\uz^0_i}(\Zb) \ x_i \ G_{\uz^0_i}(\Zb)\big] - E\big[G_{\uz^0_i}(\Zb) \ y_i \ G_{\uz^0_i}(\Zb)\big].
\]
First, noting that for any $a=a^\ast \in \A$ and $b\in \bH^\pm(\B)$ such that $\|b^{-1}\|<1/\|a\|$, we can write $G_a(b)$ as a convergent power series as follows:
\begin{equation}\label{Resolvent}
G_a(b)= \sum_{n\geq 0} b^{-1} (ab^{-1})^n\,.
\end{equation}
Hence, for $\Zb \in \bH^+(\B)$ such that $\|\Zb^{-1}\|\leq 1/\| \uz_i^0\|$, we have 
\begin{align}\label{1storder-B}
E\big[ G_{\uz^0_i}(\Zb) \ x_i \ G_{\uz^0_i}(\Zb)\big]
&= \sum_{n,m\geq 0} E\big[\Zb^{-1}(\uz_0^i \Zb^{-1})^n \ x_i \ \Zb^{-1}(\uz_0^i \Zb^{-1})^m  \big]
\\& = \sum_{n,m\geq 0} E\big[\Zb^{-1}(\uz_0^i \Zb^{-1})^n \big] \ E[x_i ] \ E\big[\Zb^{-1}(\uz_0^i \Zb^{-1})^m  \big] , \nonumber
\\& =E\big[ G_{\uz^0_i}(\Zb)\big]\ E[x_i ]\ E \big[ G_{\uz^0_i}(\Zb)\big]=0, \nonumber
\end{align}
where the last equality follows from the fact that $E[x_i]=0$. In the same way, we prove that $E\big[G_{\uz^0_i}(\Zb) \ y_i \ G_{\uz^0_i}(\Zb)\big]=0$ and hence $E\big[ A_i (\Zb) \big]=0$ for $\Zb \in \bH^+(\B)$ such that $\|\Zb^{-1}\|\leq 1/\| \uz_i^0\|$. This identity extends by analyticity overall $\bH(\B)^+$. We turn now to the second order term in \eqref{telescoping-sum}:
\[
E\big[ B_i (\Zb) \big] = E\big[ G_{\uz^0_i}(\Zb) \big( x_i \  G_{\uz^0_i}(\Zb)\big)^2\big] - E\big[ G_{\uz^0_i}(\Zb) \big( y_i \  G_{\uz^0_i}(\Zb)\big)^2\big].
\]
We develop the first term on the right-hand side using \eqref{Resolvent} again. For $\Zb \in \bH^+(\B)$ such that $\|\Zb^{-1}\|\leq 1/\| \uz_i^0\|$, we write 
\begin{align}\label{2ndorder-B}
 E\big[  G_{\uz^0_i}(\Zb) &\ x_i \  G_{\uz^0_i}(\Zb) \ x_i \  G_{\uz^0_i}(\Zb)\big] \nonumber \\
& = \sum_{n,m,k\geq 0} E\big[\Zb^{-1}(\uz_0^i \Zb^{-1})^n \ x_i \ \Zb^{-1}(\uz_0^i \Zb^{-1})^m  \ x_i \ \Zb^{-1}(\uz_0^i \Zb^{-1})^k\big]\nonumber
 \\& = \sum_{n,k\geq 0} E\big[\Zb^{-1}(\uz_0^i \Zb^{-1})^n\big] \ E\big[ x_i \Zb^{-1} x_i\big] \ E\big[\Zb^{-1}(\uz_0^i \Zb^{-1})^k\big]\nonumber
 \\& \qquad + \sum_{n,k\geq 0} \sum_{m\geq 1} E\big[\Zb^{-1}(\uz_0^i \Zb^{-1})^n\big] \ E[ x_i] \ E\big[ \Zb^{-1}(\uz_0^i \Zb^{-1})^m\big]\ E[ x_i] \ E\big[\Zb^{-1}(\uz_0^i \Zb^{-1})^k\big] \nonumber
  \\& = \sum_{n,k\geq 0} E\big[\Zb^{-1}(\uz_0^i \Zb^{-1})^n\big] \ E\big[ x_i \Zb^{-1} x_i\big] \ E\big[\Zb^{-1}(\uz_0^i \Zb^{-1})^k\big]\nonumber\\
 &= E\big[ G_{\uz^0_i}(\Zb)\big] \ E\big[ x_i \Zb^{-1} x_i\big] \  E\big[ G_{\uz^0_i}(\Zb)\big].
\end{align}
The second equality follows from the fact that $x_i$ and $\uz_i^0$ are Boolean independent over $\B$ while the third one follows from the fact that $E[x_i]=0$. Similarly, we prove that 
\[E\big[ G_{\uz^0_i}(\Zb) \big( y_i \  G_{\uz^0_i}(\Zb)\big)^2\big]= E\big[ G_{\uz^0_i}(\Zb)\big] \ E\big[ y_i \Zb^{-1} y_i\big] \  E\big[ G_{\uz^0_i}(\Zb)\big] .\]
As the second order moments match by Assumption \ref{A:general}, we get that $E\big[ B_i (\Zb) \big]=0$ for $\Zb \in \bH^+(\B)$ such that $\|\Zb^{-1}\|\leq 1/\| \uz_i^0\|$ and hence by analyticity for  all $\Zb \in \bH^+(\B)$. We are left with the third order term $E\big[C_i (\Zb) \big]$, namely 
\begin{align*}
E\big[G_{\ux_N}  (\Zb)\big]  - E \big[G_{\uy_N}(\Zb)\big]
 = \sum^N_{i=1} \Big( E \big[G_{\uz_i}(\Zb) \big( x_i \  G_{\uz^0_i}(\Zb)\big)^3\big] -E \big[G_{\uz_{i-1}}(\Zb)\big(  y_i \  G_{\uz^0_i}(\Zb)\big)^3\big]\Big).
 \end{align*}
 Noting that $G_{\uz_i}(\Zb)  x_i G_{\uz^0_i}(\Zb) = G_{\uz^0_i}(\Zb)  x_i G_{\uz_i}(\Zb)$, then \eqref{Cauchy-Schwarz} and \eqref{positivity} yield 
\begin{align}\label{CC-3rdorder}
\big\|E & \big[  G_{\uz_i}(\Zb)  x_i  G_{\uz^0_i}(\Zb)  x_i  G_{\uz^0_i}(\Zb)  x_i  G_{\uz^0_i}(\Zb) \big]\big\|^2 \nonumber
\\& = \big\|E \big[  G_{\uz^0_i}(\Zb) x_i  G_{\uz_i}(\Zb) x_i G_{\uz^0_i}(\Zb) x_i \ G_{\uz^0_i}(\Zb) \big]\big\|^2 \nonumber
\\ & \leq  \big\| E\big[ G_{\uz^0_i}(\Zb) x_i G_{\uz_i}(\Zb)  G_{\uz_i}(\Zb)^\ast x_i \ G_{\uz^0_i}(\Zb)^\ast \big]\big\| \cdot
\big\| E\big[ G_{\uz^0_i}(\Zb)^\ast x_i \ G_{\uz^0_i}(\Zb)^\ast x_i^2 G_{\uz^0_i}(\Zb) x_i \ G_{\uz^0_i}(\Zb) \big]\big\| \nonumber
\\&  \leq  \|G_{\uz_i}(\Zb)\|^2  
\cdot \big\| E\big[ G_{\uz^0_i}(\Zb) x_i^2 \ G_{\uz^0_i}(\Zb)^\ast \big]\big\|
\cdot
\big\| E\big[ G_{\uz^0_i}(\Zb)^\ast x_i \ G_{\uz^0_i}(\Zb)^\ast x_i^2 G_{\uz^0_i}(\Zb) x_i \ G_{\uz^0_i}(\Zb) \big]\big\| .
\end{align}
 First, we note that $\|G_{\uz_i}(\Zb)\| \leq \| \Imm (\Zb)^{-1}\|$. We also note that by similar computations as in \eqref{1storder-B} and the positivity of $E$, we have 
\begin{align*}
  \big\| E\big[ G_{\uz^0_i}(\Zb) \ x_i^2 \ G_{\uz^0_i}(\Zb)^\ast \big]\big\| 
  & = \big\| E\big[ G_{\uz^0_i}(\Zb) \ E[x_i^2 ] \ G_{\uz^0_i}(\Zb)^\ast \big]\big\|
  \\& \leq \big\|E[ x_i^2 ]\big\|  \cdot \big\| E\big[ G_{\uz^0_i}(\Zb)\big]\big\|^2 
   \leq \big\|\Imm (\Zb)^{-1}\big\|^2 \alpha_2(x) ,
\end{align*}
where $\alpha_2(x):= \max_{1\leq i \leq N}\|E[ x_i^2 ]\|$. Again, by similar computations as in \eqref{2ndorder-B} and the positivity of $E$, we get 
\begin{align*}
    \big\| E\big[ G_{\uz^0_i}(\Zb)^\ast x_i \ G_{\uz^0_i}(\Zb)^\ast x_i^2 G_{\uz^0_i}(\Zb) x_i \ G_{\uz^0_i}(\Zb) \big]\big\| 
    &= \big\| E\big[ G_{\uz^0_i}(\Zb)^\ast \cdot E \big[x_i (\Zb^{-1})^\ast x_i^2 \Zb^{-1} x_i\big] \cdot G_{\uz^0_i}(\Zb) \big]\big\| 
    \\& \leq  \big\|E\big[x_i (\Zb^{-1})^\ast x_i^2 \Zb^{-1} x_i\big]\big\|  \cdot \big\| E\big[ G_{\uz^0_i}(\Zb)\big]\big\|^2 
  \\& \leq  \|\Zb^{-1}\|^2 \big\|\Imm (\Zb)^{-1}\big\|^2 \alpha_4(x),
\end{align*}
where $\alpha_4(x):=\max_{1\leq i\leq N}\sup_{} \|E[x_i  b^* x_i^2 b x_i ] \| $ with the supremum  taken over $b \in \B$ such that $ \|b\|=1 $. Putting the above bounds together, we get 
\begin{align*}
\big\|E & \big[  G_{\uz_i}(\Zb)  x_i  G_{\uz^0_i}(\Zb)  x_i  G_{\uz^0_i}(\Zb)  x_i  G_{\uz^0_i}(\Zb) \big]\big\|
\leq \|\Zb^{-1}\| \|\Imm(\Zb)^{-1}\|^3 \sqrt{\alpha_2(x)}\sqrt{\alpha_4(x) }\ .
\end{align*}
Similarly, we prove that 
\begin{align*}
\big\|E & \big[  G_{\uz_{i-1}}(\Zb)  y_i  G_{\uz^0_i}(\Zb)  y_i  G_{\uz^0_i}(\Zb)  y_i  G_{\uz^0_i}(\Zb) \big]\big\|
\leq \|\Zb^{-1}\| \|\Imm(\Zb)^{-1}\|^3 \sqrt{\alpha_2(y)}\sqrt{\alpha_4(y) }\ .
\end{align*}
As the second moments match, we have  $\alpha_2(x)=\alpha_2(y)$, and hence we infer that
\begin{align*}
\|E[C_i(\Zb)]\| &\leq \|\Zb^{-1}\| \|\Imm(\Zb)^{-1}\|^3 \sqrt{\alpha_2(x)}\Big(\sqrt{\alpha_4(x)}+\sqrt{\alpha_4(y) }\Big) \\ 
&\leq \|\Imm(\Zb)^{-1}\|^4 \sqrt{\alpha_2(x)}\Big(\sqrt{\alpha_4(x)}+\sqrt{\alpha_4(y) }\Big). 
\end{align*}
Finally, summing over $i=1,\ldots,N$ we get the bound in \eqref{eq1:theo-BM}. 
\begin{remark}
Note that the last inequality we use the fact that $\|\Zb^{-1}\|\leq \|\Imm(\Zb)^{-1}\|$ for all $\Zb\in \bH^+(\B)$. The proof of this inequality can be found in \cite{haagerup2005new}, or it also can be easily verified as follows: we write 
$$
\Zb=x+iy \text{ where }x=x^* \text{ and }y=y^*>0\in \B.
$$
Then we observe that 
$$
\Zb^{-1} = y^{-1/2}[w+i]^{-1}y^{-1/2} \text{ where }w=y^{-1/2}xy^{-1/2}, 
$$
which implies that
$
\|\Zb^{-1}\|\leq \|y^{-1/2}\|^2 \|(w+i)^{-1}\| =\|y^{-1}\| \|(w+i)^{-1}\|.
$
Finally, note that $w$ is self-adjoint, and hence 
$$
\|(w+i)^{-1}\|=\frac{1}{d(i,\sigma(w))}\leq 1,
$$
where $\sigma(w)$ is the spectrum of $w$ and $d(i,\sigma(w))=\inf\ \{\ |i-s| \mid s\in\sigma(w)\}$.
\end{remark}

\vspace{0.25cm}
\noindent
{\bf Proof of \eqref{eq2:theo-BM}.} As the conditional expectation $E$ preserves $\varphi$, $\varphi= \varphi \circ E$, the first steps of the proof follow in the same way as for \eqref{eq1:theo-BM}. For each $z \in \mathbb{C}^+$, we have that 
\begin{align*}
 \varphi [G_{\ux_N}(z)] - \varphi [G_{\uy_N}(z)] &= \sum_{i=1}^N \varphi   \big[C_i (z)\big]= \sum_{i=1}^N \varphi  \big[G_{\uz_i}(z) \big( x_i \  G_{\uz^0_i}(z)\big)^3 -G_{\uz_{i-1}}(z)\big(  y_i \  G_{\uz^0_i}(z)\big)^3\big]. 
 \end{align*}
We start by controlling the first term of $C_i (z)$. By similar computations as in \eqref{CC-3rdorder}, we get  
\begin{align*}
 \big|\varphi \big[  & G_{\uz_i}(z)  x_i  G_{\uz^0_i}(z)  x_i  G_{\uz^0_i}(z)  x_i  G_{\uz^0_i}(z) \big]\big|^2
\\ &  \leq  \|G_{\uz_i}(z)\|^2  
\cdot  \varphi\big[ G_{\uz^0_i}(z) x_i^2 \ G_{\uz^0_i}(z)^\ast \big]
\cdot
\varphi\big[ G_{\uz^0_i}(z)^\ast x_i \ G_{\uz^0_i}(z)^\ast x_i^2 G_{\uz^0_i}(z) x_i  G_{\uz^0_i}(z) \big] .
\end{align*}
 First, we note that $\|G_{\uz_i}(z)\| \leq \| \Imm (z)^{-1}\|$. As  $\varphi= \varphi \circ E$, then by similar computations as in \eqref{1storder-B} and the positivity of $E$, we have 
\begin{align*}
 \varphi \big[ G_{\uz^0_i}(z) \ x_i^2 \ G_{\uz^0_i}(z)^\ast \big]
  & = \varphi \big[ G_{\uz^0_i}(z) \ E[x_i^2 ] \ G_{\uz^0_i}(z)^\ast \big]
  \\& \leq   \big\|E[ x_i^2 ]\big\| \cdot \varphi \big[ G_{\uz^0_i}(z) G_{\uz^0_i}(z)^\ast \big]
 \leq  \alpha_2(x) \big\| G_{\uz^0_i}(z)\big\|_{L^2(\A,\varphi)}^2 . 
\end{align*}
Again, by similar computations as in \eqref{2ndorder-B} and the positivity of $E$, we get 
\begin{align*}
\varphi\big[ G_{\uz^0_i}(z)^\ast x_i \ G_{\uz^0_i}(z)^\ast x_i^2 G_{\uz^0_i}(z) x_i \ G_{\uz^0_i}(z) \big] 
&= \varphi \big[ G_{\uz^0_i}(z)^\ast \cdot E \big[x_i \bar{z}^{-1} x_i^2 z^{-1} x_i\big] \cdot G_{\uz^0_i}(z) \big]
\\& \leq  \frac{1}{|z|^2}\big\|E[x_i^4]\big\|  \cdot  \varphi \big[ G_{\uz^0_i}(z)^\ast G_{\uz^0_i}(z) \big] 
\\& \leq \frac{1}{\Imm(z)^2} \widetilde{\alpha}_4(x)  \big\| G_{\uz^0_i}(z)\big\|_{L^2(\A,\varphi)}^2 ,
\end{align*}
where $\widetilde{\alpha}_4(x):=\max_{1\leq i \leq N}\|E[x_i^4]\|$. Putting the above terms together we get for any $z \in \mathbb{C}^+$, 
\[
 \big|\varphi \big[ G_{\uz_i}(z)  x_i  G_{\uz^0_i}(z)  x_i  G_{\uz^0_i}(z)  x_i  G_{\uz^0_i}(z) \big]\big|
 \leq \frac{1}{\Imm(z)^2} \sqrt{ \alpha_2(x) \widetilde{\alpha}_4(x)} \big\| G_{\uz^0_i}(z)\big\|_{L^2(\A,\varphi)}^2.
\]
Similarly, we control the second term in $\varphi[C_i(z)]$ and get
\[
 \big|\varphi \big[ G_{\uz_{i-1}}(z)  y_i  G_{\uz^0_i}(z)  y_i  G_{\uz^0_i}(z)  y_i  G_{\uz^0_i}(z) \big]\big|
 \leq \frac{1}{\Imm(z)^2} \sqrt{\alpha_2(y) \widetilde{\alpha}_4(y)} \big\| G_{\uz^0_i}(z)\big\|_{L^2(\A,\varphi)}^2 .
\]
As the second order moments match, $\alpha_2(x)=\alpha_2(y)$ and hence 
\[
\big| \varphi [G_{\ux_N}(z)] - \varphi [G_{\uy_N}(z)]\big|
\leq \frac{1}{\Imm(z)^2} \sqrt{\alpha_2(x)} \left(\sqrt{ \widetilde{\alpha}_4(x)} + \sqrt{\widetilde{\alpha}_4(y)}\right) \sum_{i=1}^N\big\| G_{\uz^0_i}(z)\big\|_{L^2(\A,\varphi)}^2 .
\]
Now taking $z=t+i\epsilon$ and using the fact that for each $x=x^\ast\in\A$,
\begin{equation}\label{eq:Cauchy-integral}
\int_\R \|G_x(t+i\epsilon)\|^2_{L^2}\, \mathrm{d} t = \frac{\pi}{\epsilon},
\end{equation}
we prove that 
\begin{align*}
\int_\R \big|\Imm \big(\varphi[G_{\ux_N}(t+i\epsilon)]\big)- \Imm \big(\varphi[G_{\uy_N}(t+i\epsilon)]\big)\big| \text{d}t \leq \frac{\pi}{\epsilon^3} \sqrt{\alpha_2(x)} \left(\sqrt{ \widetilde{\alpha}_4(x)} + \sqrt{\widetilde{\alpha}_4(y)}\right) N. 
\end{align*}
\end{proof}

\subsection{Monotone Case}\label{section:Monotone} We prove now Theorem \ref{theo:monotone-monotone} relative to monotone independence. The first lines of the proof are the same as in the Boolean case and rely on the operator-valued Lindeberg method in Proposition \ref{prop:Lindeberg}. However, as the order of the index set in the monotone case matters, this will require that we further expand the terms $\uz_i^0$ in the power series expansions of the relevant resolvents before factorizing with respect to monotone independence. For this aim, we would need the following lemma:

\begin{lemma}\label{Lemma:fact-monotone}
Let $x_1,x_2,y_1,y_2$ and $ w$ be self-adjoint elements in $\A$ that are such that $\{x_1,x_2\}\prec w \prec \{y_1,y_2\}$ over $\B$. Then, for any $n,m\geq 0$,
 \begin{align*}
 E\big[(x_1+y_1)^n  w(x_2+y_2)^m)\big]
= \sum^n_{k=0}& \sum^m_{\ell=0}
\sum_{\substack{q_0,...,q_k\geq 0\\q_0+\dots+q_k=n-k}} \sum_{\substack{p_0,...,p_\ell\geq 0\\p_0+\dots+p_\ell=m-\ell}}\\& \qquad\qquad
E\Big[ E[ y_1^{q_0}] x_1E[ y_1^{q_1}]\dots E[ y_1^{q_{k-1}}] x_1 E[ y_1^{q_k}]
 \\& \qquad\qquad\cdot E[w] \cdot E[ y_2^{p_0}] x_2 E[y_2^{p_1}] \dots E[ y_2^{p_{\ell-1}}] x_2 E[ y_2^{p_\ell}]\Big] .
\end{align*}
 \end{lemma}

\begin{proof}
For any $x,y \in \A$ and $r\geq 0$, we have by the noncommutative binomial expansion 
\begin{align*}
 (x+y)^r &= \sum^r_{k=0} \sum_{\substack{q_0,...,q_k\geq 0\\q_0+\dots+q_k=r-k}} 
y^{q_0} x  y^{q_1} \dots  y^{q_{k-1}}  x y^{q_k} \, .
\end{align*}
Hence, we write
\begin{align*}
 E\big[(x_1+y_1)^n  w(x_2+y_2)^m)\big]
= &\sum^n_{k=0} \sum^m_{\ell=0}
\sum_{\substack{q_0,...,q_k\geq 0\\q_0+\dots+q_k=n-k}} \sum_{\substack{p_0,...,p_\ell\geq 0\\p_0+\dots+p_\ell=m-\ell}}\\& \quad
E\Big[ y_1^{q_0} x_1  y_1^{q_1} \dots   y_1^{q_{k-1}}  x_1 y_1^{q_k} \cdot w \cdot
y_2^{p_0}  x_2 y_2^{p_1} \dots  y_2^{p_{\ell-1}} x_2  y_2^{p_\ell} \Big].
\end{align*}
Now as $\{x_1,x_2\}\prec w \prec \{y_1,y_2\}$ over $\B$, from the definition of monotone independence with amalgamation over $\B$, we get 
\begin{align*}
E\Big[& y_1^{q_0} x_1  y_1^{q_1} \dots   y_1^{q_{k-1}}  x_1 y_1^{q_k} \cdot w \cdot
y_2^{p_0}  x_2 y_2^{p_1} \dots  y_2^{p_{\ell-1}} x_2  y_2^{p_\ell} \Big]
\\&=
E\Big[ E[ y_1^{q_0}] x_1E[ y_1^{q_1}]\dots E[ y_1^{q_{k-1}}] x_1 E[ y_1^{q_k}]
\cdot  w \cdot E[ y_2^{p_0}] x_2 E[y_2^{p_1}] \dots E[ y_2^{p_{\ell-1}}] x_2 E[ y_2^{p_\ell}]\Big] 
\\&=
E\Big[ E[ y_1^{q_0}] x_1E[ y_1^{q_1}]\dots E[ y_1^{q_{k-1}}] x_1 E[ y_1^{q_k}]
\cdot E[w] \cdot E[ y_2^{p_0}] x_2 E[y_2^{p_1}] \dots E[ y_2^{p_{\ell-1}}] x_2 E[ y_2^{p_\ell}]\Big],
\end{align*}
which ends the proof. 
\end{proof}
Having the above factorization lemma in hand, we prove the following bounds that will be used at several occasions in the proof later on.
\begin{lemma}\label{Lemma:Monotone-Cauchy}
Let $x,y$ and $w$ be self-adjoint elements in $\A$ such that $x\prec w \prec y$ over $\B$, then for any $ b_1, b_2 \in \bH^\pm(\B)$,
\begin{enumerate}
\item  $E\big[G_{x+y}(b_1)\ w \ G_{x+y}(b_2)\big]
=E \big[G_{x}\big(E[G_{y}(b_1)]^{-1}\big)\cdot E[w] \cdot G_{x}\big(E[G_{y}(b_2)]^{-1}\big)\big].$
\end{enumerate}
Moreover, if $W\geq 0$, then for any $ b_1, b_2 \in \bH^\pm(\B)$,
\begin{enumerate}
\item[(ii)] $\big\| E[G_{x+y}(b_1)\ w \ G_{x+y}(b_2)] \big\| \leq \big\|E[w]\big\| \cdot \big\| \Imm(E[G_y(b_1)]^{-1})^{-1} \big\| \cdot \big\| \Imm (E[G_y(b_2)]^{-1})^{-1}\big\|$,
\end{enumerate}
and for any $z \in \mathbb{C}^+$,
\begin{enumerate}
    \item[(iii)] $ \varphi[G_{x+y}(z)\ w \ G_{x+y}(z)^\ast ] \leq \big\| E[w]\big\| \cdot \| G_{x}\big(E[G_{y}(z)]^{-1}\big)\|_{L_2(\A,\varphi)}^2.$
\end{enumerate}\end{lemma}

\begin{proof}
{\it (i)} 
Writing the resolvent as a power series as in \eqref{Resolvent}, we have for $b_1$ and $b_2$ that are such that $\max\{ \|b_1^{-1}\|, \|b_2^{-1}\| \} < 1/ \|x+y\|$,
\[
E\big[ G_{x+y}(b_1) \ w \ G_{x+y}(b_2) \big]
=\sum_{n,m\geq 0} b_1^{-1} E\big[ (xb_1^{-1} +yb_1^{-1})^n w b_2^{-1} (xb_2^{-1} +yb_2^{-1})^m \big].
\]
Now, we get by Lemma \ref{Lemma:fact-monotone} after summing over $n,m \geq0$,
\begin{align*}
&E \big[ G_{x+y}(b_1) \ w \ G_{x+y}(b_2) \big] 
\\&=  \sum_{k,\ell \geq 0} \sum_{q_0,...,q_k\geq 0} \sum_{p_0,...,p_\ell\geq 0}
\E\Big[b_1^{-1} E\big[(yb_1^{-1})^{q_0}\big] xb_1^{-1}  E\big[(yb_1^{-1})^{q_1}\big]  \cdots       xb_1^{-1}  E\big[(yb_1^{-1})^{q_k}\big]
\\& \qquad \qquad \cdot E[w] \cdot
b_2^{-1} E\big[(yb_2^{-1})^{p_0}\big]  xb_2^{-1} E\big[(yb_2^{-1})^{p_1}\big] \cdots  xb_2^{-1}  E\big[(yb_2^{-1})^{p_\ell}\big] \Big]
\\& =\sum_{k,\ell \geq 0} E \Bigg[ E\Big[ \sum_{q_0\geq 0} b_1^{-1}(yb_1^{-1})^{q_0} \Big] \cdot x \cdot E\Big[ \sum_{q_1\geq 0} b_1^{-1}(yb_1^{-1})^{q_1} \Big] \cdots x \cdot  E\Big[ \sum_{q_k\geq 0} b_1^{-1}(yb_1^{-1})^{q_k} \Big]
\\ &\qquad  \qquad \cdot E[w] \cdot E\Big[ \sum_{p_0\geq 0} b_2^{-1}(yb_2^{-1})^{p_0}\Big]\cdot  x \cdot E\Big[ \sum_{p_1\geq 0} b_2^{-1}(yb_2^{-1})^{p_1}\Big] \cdots x \cdot E\Big[ \sum_{p_\ell\geq 0} b_2^{-1}(yb_2^{-1})^{p_\ell}\Big]\Bigg]
\\ &=\sum_{k,\ell \geq 0} E \Big[ E\big[G_y(b_1)\big] \big( x \cdot E\big[G_y(b_1)\big]\big)^k \cdot E[w] \cdot E\big[G_y(b_2)\big] \big( x \cdot E\big[G_y(b_2)\big]\big)^\ell \Big].
\end{align*}
For $i\in \{1,2\}$, we set $\widetilde{b}_i= E[G_y(b_i)]^{-1}$ and write
\begin{align*}
E \big[ G_{x+y}(b_1) \ w \ G_{x+y}(b_2) \big] 
&= E\Big[ \sum_{k \geq 0} \widetilde{b}_1^{-1} (x \widetilde{b}_1^{-1})^k \cdot E[w] \cdot \sum_{\ell \geq 0} \widetilde{b}_2^{-1} (x \widetilde{b}_2^{-1})^\ell \Big]
\\& =  E\Big[G_x \big(E[G_y(b_1)]^{-1}\big) \cdot E[w] \cdot G_x \big(E[G_y(b_2)]^{-1}\big) \Big],
\end{align*}
which proves the desired equality for $b_1,b_2 \in \bH^\pm(\B)$ such that $\max(\|b_1^{-1}\|,\|b_2^{-1}\|)< 1/||x+y||$.  The equality can be extended to $\bH^\pm(\B)$ since both of its sides are analytical.

\vspace{0.25cm}
\noindent
{\it (ii)} The second part of the lemma follows directly from the first: 
\begin{align*}
\big\| E[G_{x+y}(b)\ w \ G_{x+y}(b)^\ast ] \big\| 
&= \big\| E[G_{x+y}(b)\ w \ G_{x+y}(b^\ast) ] \big\| 
\\&=\big\| E \big[G_{x}\big(E[G_{y}(b)]^{-1}\big)\cdot E[w] \cdot G_{x}\big(E[G_{y}(b^\ast)]^{-1}\big)\big] \big\|
\\&\leq \big\| G_{x}\big(E[G_{y}(b)]^{-1}\big)\big\| \cdot \big\|G_{x}\big(E[G_{y}(b^\ast)]^{-1}\big)\big] \big\| \cdot \big\| E[w]\big\|
\\&\leq  \big\| \Imm\big(E[G_{y}(b)]^{-1}\big)^{-1}\big\| \cdot \big\|\Imm\big(E[G_{y}(b^\ast)]^{-1}\big)^{-1} \big\| \cdot \big\| E[w]\big\|
\\& \leq \big\| \Imm(b)^{-1} \big\|^2 \cdot \big\| E[w]\big\|\,,
\end{align*}
where the last inequality follows from the fact that $\Imm\big(E[G_{y}(b)]^{-1}\big) \geq \Imm (b)$ for any $b \in \bH^+(\B)$, see Remark 2.5 in \cite{BPV-12}.

\vspace{0.25cm}
\noindent
{\it (iii)} The last estimate also follows from $(i)$: for any $z \in \mathbb{C}^+$,
\begin{align*}
 \varphi[G_{x+y}(z)\ w \ G_{x+y}(z)^\ast ]
&= \varphi \big[G_{x}\big(E[G_{y}(z)]^{-1}\big)\cdot E[w] \cdot G_{x}\big(E[G_{y}(z)]^{-1}\big)^\ast\big] 
\\&\leq  \big\| E[w]\big\| \cdot  \varphi \big[G_{x}\big(E[G_{y}(z)]^{-1}\big) G_{x}\big(E[G_{y}(z)]^{-1}\big)^\ast\big] 
\\& =\big\| E[w]\big\| \cdot \| G_{x}\big(E[G_{y}(z)]^{-1}\big)\|_{L_2(\A,\varphi)}^2,
\end{align*}
where the inequality follows from the positivity of $\varphi$.
\end{proof}

Having the above lemmas in hand, we are now ready to prove Theorem \ref{theo:monotone-monotone} relative to the monotone case.
\begin{proof}[Proof of Theorem \ref{theo:monotone-monotone}]
Without loss of generality, we assume that $\A_{x_1,y_1} \prec \A_{x_2,y_2} \prec \dots \prec \A_{x_N,y_N}$ over $\B$. Indeed, one can simply take copies of the families $\{x_1, \dots x_N \}$ and $\{y_1, \dots y_N \}$ where this holds, for instance, we could choose $\A_{x_1} \prec \A_{x_2} \prec \cdots \prec  \A_{x_N}\prec\A_{y_1} \prec \A_{y_2} \cdots \prec \A_{y_N}$. The starting point is the operator-valued Lindeberg method in Proposition \ref{prop:Lindeberg} where we note that $\ux_N=\uz_N$ and $\uy_N=\uz_0$. 

\vspace{0.25cm}
\noindent
{\bf Proof of \eqref{eq1:theo-monotone}}. Again, we will prove that $E\big[ A_i (\Zb) \big]=0$, $E\big[ B_i (\Zb) \big]=0$ and that 
\begin{align*}
\|E[C_i(\Zb)]\| \leq \|\Imm(\Zb)^{-1}\|^4 \sqrt{\alpha_2(x)}\left(\sqrt{\alpha_4(x)}+\sqrt{\alpha_4(y) }\right)\ . 
\end{align*}

As above, all the identities here will be proven for $\Zb \in \bH^+(\B)$ such that $\|\Zb^{-1}\| $ is small enough and then get extended analytically to $\bH^+(\B)$. Considering  the telescoping sum in \eqref{telescoping-sum}, we fix $i \in \{ 1,\dots,N\}$ and start by controlling the first order term 
\[
E\big[ A_i (\Zb) \big]= E\big[ G_{\uz^0_i}(\Zb) \ x_i \ G_{\uz^0_i}(\Zb)\big] - E\big[G_{\uz^0_i}(\Zb) \ y_i \ G_{\uz^0_i}(\Zb)\big].
\]
We set $\uu_i= \sum_{j=1}^{i-1} x_j$ and $\uv_i = ~\sum_{j=i+1}^{N} y_j$, and note that  $\uu_i \prec x_i \prec \uv_i$ over $\B$. Then taking into account that $E[x_i]=0$, we get by Lemma \ref{Lemma:Monotone-Cauchy}, 
\begin{align*}
 E \big[ G_{\uz^0_i}(\Zb) \ x_i \ G_{\uz^0_i}(\Zb) \big] &=  E \big[ G_{\uu_i+\uv_i}(\Zb) \ x_i \ G_{\uu_i+\uv_i}(\Zb) \big]
\\&=  E \big[G_{\uu_i}\big(E[G_{\uv_i}(\Zb)]^{-1}\big)\cdot E[x_i] \cdot G_{\uu_i}\big(E[G_{\uv_i}(\Zb)]^{-1}\big)\big]=0.
\end{align*}
Similarly, we prove that $E \big[ G_{\uz^0_i}(\Zb) \ y_i \ G_{\uz^0_i}(\Zb) \big]=0$ and infer that $E[A_i(\Zb)]=0$. For  the second order term in \eqref{telescoping-sum}:
\[
E\big[ B_i (\Zb) \big] = E\big[ G_{\uz^0_i}(\Zb) \big( x_i \  G_{\uz^0_i}(\Zb)\big)^2\big] - E\big[ G_{\uz^0_i}(\Zb) \big( y_i \  G_{\uz^0_i}(\Zb)\big)^2\big],
\]
we develop the first term using \eqref{Resolvent}:
\begin{align*}
 E \Big[G_{\uz^0_i}(\Zb) \big( x_i   G_{\uz^0_i}(\Zb)\big)^2\Big]\!
&= E \Big[G_{\uu_i+\uv_i}(\Zb)  x_i   G_{\uu_i+\uv_i}(\Zb) x_i   G_{\uu_i+\uv_i}(\Zb)\Big]
\\&=\!\!\sum_{n,m,r\geq 0}\! \!E\Big[ \Zb^{-1} \big( (\uu_i+\uv_i)\Zb^{-1}\big)^n  x_i  \Zb^{-1} \big( (\uu_i+\uv_i)\Zb^{-1}\big)^m  x_i  \Zb^{-1} \big( (\uu_i+\uv_i)\Zb^{-1}\big)^r\Big]
\\&= \!\!\sum_{n,m,r\geq 0}\! \!\Zb^{-1} E\Big[ \big( \widetilde{\uu}_i+\widetilde{\uv}_i\big)^n  \widetilde{x}_i \big( \widetilde{\uu}_i+\widetilde{\uv}_i\big)^m  \widetilde{x}_i \big( \widetilde{\uu}_i+\widetilde{\uv}_i\big)^r\Big],
\end{align*}
where $\widetilde{\uu}_i= \uu_i \Zb^{-1}$, $\widetilde{\uv}_i= \uv_i \Zb^{-1}$ and $\widetilde{x}_i=x_i \Zb^{-1}$. Now using the noncommutative binomial expansion, we write 
\begin{multline*}
E\Big[ \big( \widetilde{\uu}_i  +\widetilde{\uv}_i\big)^n  \widetilde{x}_i \big( \widetilde{\uu}_i+\widetilde{\uv}_i\big)^m  \widetilde{x}_i \big( \widetilde{\uu}_i+\widetilde{\uv}_i\big)^r\Big]
 \\= \sum^n_{k=0} \sum^m_{\ell=0} \sum^r_{s=0}
\sum_{\substack{q_0,...,q_k\geq 0\\q_0+\dots+q_k=n-k}} \sum_{\substack{p_0,...,p_\ell\geq 0\\p_0+\dots+p_\ell=m-\ell}}
\sum_{\substack{t_0,...,t_s\geq 0\\t_0+\dots+t_s=r-s}}
\\ \qquad \qquad E\Big[ \widetilde{\uv}_i^{q_0}\widetilde{\uu}_i\widetilde{\uv}_i^{q_1} \dots \widetilde{\uu}_i\widetilde{\uv}_i^{q_k} \cdot \widetilde{x}_i \cdot \widetilde{\uv}_i^{p_0}\widetilde{\uu}_i\widetilde{\uv}_i^{p_1} \dots \widetilde{\uu}_i\widetilde{\uv}_i^{p_\ell} \cdot \widetilde{x}_i \cdot \widetilde{\uv}_i^{t_0}\widetilde{\uu}_i\widetilde{\uv}_i^{t_1} \dots \widetilde{\uu}_i\widetilde{\uv}_i^{t_s}  
\Big].
\end{multline*}
Noting that $\widetilde{\uu}_i \prec \widetilde{x}_i \prec \widetilde{\uv}_i$ over $\B$ and taking into account that $E[\widetilde{x}_i]=E[x_i] \Zb^{-1}=0$, we get whenever $m\neq 0$
\begin{align}\label{monotone-fact-2}
 E\Big[ \widetilde{\uv}_i^{q_0}\widetilde{\uu}_i \widetilde{\uv}_i^{q_1} & \dots \widetilde{\uu}_i\widetilde{\uv}_i^{q_k} \cdot \widetilde{x}_i \cdot \widetilde{\uv}_i^{p_0}\widetilde{\uu}_i\widetilde{\uv}_i^{p_1} \dots \widetilde{\uu}_i\widetilde{\uv}_i^{p_\ell} \cdot \widetilde{x}_i \cdot \widetilde{\uv}_i^{t_0}\widetilde{\uu}_i\widetilde{\uv}_i^{t_1} \dots \widetilde{\uu}_i\widetilde{\uv}_i^{t_s}  
\Big]\nonumber
\\&
= E\Big[ E[\widetilde{\uv}_i^{q_0}] \ \widetilde{\uu}_i \ E[\widetilde{\uv}_i^{q_1}] \dots \widetilde{\uu}_i \ E[\widetilde{\uv}_i^{q_k}] \cdot E[\widetilde{x}_i] \cdot E[\widetilde{\uv}_i^{p_0}] \ \widetilde{\uu}_i \ E[\widetilde{\uv}_i^{p_1}] \nonumber
\\& \qquad \qquad \qquad \dots \widetilde{\uu}_i \ E[\widetilde{\uv}_i^{p_\ell}] \cdot E[\widetilde{x}_i] \cdot E[\widetilde{\uv}_i^{t_0}]\ \widetilde{\uu}_i E[\widetilde{\uv}_i^{t_1}] \dots \widetilde{\uu}_i \ E[\widetilde{\uv}_i^{t_s}] 
\Big]
=0.
\end{align}
This yields that 
\begin{align*}
 E \Big[G_{\uz^0_i}(\Zb)&  x_i   G_{\uz^0_i}(\Zb) x_i   G_{\uz^0_i}(\Zb)\Big]
 \\&= \sum_{n,r\geq 0} \sum^n_{k=0}  \sum^r_{s=0}
\sum_{\substack{q_0,...,q_k\geq 0\\q_0+\dots+q_k=n-k}}
\sum_{\substack{t_0,...,t_s\geq 0\\t_0+\dots+t_s=r-s}}
\\& \qquad\Zb^{-1} E\Big[ E[\widetilde{\uv}_i^{q_0}] \widetilde{\uu}_i E[\widetilde{\uv}_i^{q_1}] \dots \widetilde{\uu}_i E[\widetilde{\uv}_i^{q_k}] \cdot E\big[x_i \Zb^{-1} x_i\big] \cdot \Zb^{-1}E[\widetilde{\uv}_i^{t_0}]\widetilde{\uu}_i E[\widetilde{\uv}_i^{t_1}] \dots \widetilde{\uu}_i E[\widetilde{\uv}_i^{t_s}] 
\Big].
\end{align*}
Similarly, we prove that 
\begin{align*}
 E \Big[G_{\uz^0_i}(\Zb)&  y_i   G_{\uz^0_i}(\Zb) y_i   G_{\uz^0_i}(\Zb)\Big]
 \\&= \sum_{n,r\geq 0} \sum^n_{k=0}  \sum^r_{s=0}
\sum_{\substack{q_0,...,q_k\geq 0\\q_0+\dots+q_k=n-k}}
\sum_{\substack{t_0,...,t_s\geq 0\\t_0+\dots+t_s=r-s}}
\\& \qquad\Zb^{-1} E\Big[ E[\widetilde{\uv}_i^{q_0}] \widetilde{\uu}_i E[\widetilde{\uv}_i^{q_1}] \dots \widetilde{\uu}_i E[\widetilde{\uv}_i^{q_k}] \cdot E\big[y_i \Zb^{-1} y_i\big] \cdot \Zb^{-1}E[\widetilde{\uv}_i^{t_0}]\widetilde{\uu}_i E[\widetilde{\uv}_i^{t_1}] \dots \widetilde{\uu}_i E[\widetilde{\uv}_i^{t_s}] 
\Big].
\end{align*}
Subtracting the two terms and recalling that $E[x_i b x_i]=E[y_i b y_i]$ for any $b \in \B$, we infer that $E[B_i(\Zb)]=0$. It remains to control the third order term $E[C_i(\Zb)]$ in \eqref{telescoping-sum}; namely, 
\begin{equation}\label{est:CE-telescopingsum}
E\big[G_{\ux_N}  (\Zb)\big]  - E \big[G_{\uy_N}(\Zb)\big]
 = \sum^N_{i=1} \Big( E \big[G_{\uz_i}(\Zb) \big( x_i \  G_{\uz^0_i}(\Zb)\big)^3\big] -E \big[G_{\uz_{i-1}}(\Zb)\big(  y_i \  G_{\uz^0_i}(\Zb)\big)^3\big]\Big).
\end{equation}
 Following the lines in \eqref{CC-3rdorder}, we get by \eqref{Cauchy-Schwarz} and \eqref{positivity}
\begin{align*}
\big\|E  \big[  G_{\uz_i}(\Zb)  & x_i  G_{\uz^0_i}(\Zb)   x_i  G_{\uz^0_i}(\Zb)  x_i  G_{\uz^0_i}(\Zb) \big]\big\|^2 \\
&\leq  \|G_{\uz_i}(\Zb)\|^2  
\cdot \big\| E\big[ G_{\uz^0_i}(\Zb) x_i^2 \ G_{\uz^0_i}(\Zb)^\ast \big]\big\|
\cdot
\big\| E\big[ G_{\uz^0_i}(\Zb)^\ast x_i \ G_{\uz^0_i}(\Zb)^\ast x_i^2 G_{\uz^0_i}(\Zb) x_i \ G_{\uz^0_i}(\Zb) \big]\big\| .
\end{align*}
First, we note that $\|G_{\uz_i}(\Zb)\| \leq \| \Imm (\Zb)^{-1}\|$. Then applying Lemma \ref{Lemma:Monotone-Cauchy} (ii) with $x=\uu_i, y=\uv_i$ and $W=x_i^2$, and for $b_1=b_2^\ast=b$,
we get
\begin{align*}
\big\| E\big[ G_{\uz^0_i}(\Zb) x_i^2 \ G_{\uz^0_i}(\Zb)^\ast \big]\big\| 
&= \big\| E\big[ G_{\uu_i +\uv_i}(\Zb) x_i^2 \ G_{\uu_i +\uv_i}(\Zb^\ast) \big]\big\| 
\\& \leq \|E[x_i^2]\| \cdot \| \Imm (\Zb)^{-1}\|^2 \leq   \| \Imm (\Zb)^{-1}\|^2 \alpha_2(x),
\end{align*}
where $\alpha_2(x)=\max\limits_{1\leq i\leq N}\|E[x_i^2]\|$.
As for the last inequality, we expand it using \eqref{Resolvent} and write
\begin{align*}
E\big[  G_{\uz^0_i}(\Zb)^\ast &x_i  G_{\uz^0_i}(\Zb)^\ast x_i^2 G_{\uz^0_i}(\Zb) x_i  G_{\uz^0_i}(\Zb) \big]
\\& =E\big[ G_{\uu_i +\uv_i}(\Zb^\ast) x_i  G_{\uu_i +\uv_i}(\Zb^\ast) x_i^2 G_{\uu_i +\uv_i}(\Zb) x_i \ G_{\uu_i +\uv_i}(\Zb) \big]
\\&  =\sum_{n,m\geq 0}\sum_{r,u\geq 0}E\Big[ (\Zb^\ast)^{-1} \big( (\uu_i+\uv_i)(\Zb^\ast)^{-1}\big)^n  x_i   (\Zb^\ast)^{-1} \big( (\uu_i+\uv_i)(\Zb^\ast)^{-1}\big)^m 
\\& \qquad  \qquad \qquad \qquad  \qquad \qquad x_i^2  \Zb^{-1} \big( (\uu_i+\uv_i)\Zb^{-1}\big)^r x_i  \Zb^{-1} \big( (\uu_i+\uv_i)\Zb^{-1}\big)^u\Big]
\\&=\sum_{n,m\geq 0}\sum_{r,u\geq 0}(\Zb^\ast)^{-1} E\Big[ \big( \widehat{\uu}_i+\widehat{\uv}_i\big)^n  \widehat{x}_i   \big( \widehat{\uu}_i+\widehat{\uv}_i\big)^m  x_i^2  \Zb^{-1}\big( \widetilde{\uu}_i+\widetilde{\uv}_i\big)^r \widetilde{x}_i \big( \widetilde{\uu}_i+\widetilde{\uv}_i\big)^u\Big],
\end{align*} 
where we have adopted the same notation as above in addition to $\widehat{\uu}_i= \uu_i (\Zb^\ast)^{-1}$, $\widehat{\uv}_i= \uv_i (\Zb^\ast)^{-1}$ and $\widehat{x}_i=x_i (\Zb^\ast)^{-1}$ for $i=1,2$. 
Similarly, we develop using the noncommutative binomial expansion and write
\begin{align*}
 &E\Big[ \big( \widehat{\uu}_i+\widehat{\uv}_i\big)^n  \widehat{x}_i   \big( \widehat{\uu}_i+\widehat{\uv}_i\big)^m  x_i^2 \Zb^{-1} \big( \widetilde{\uu}_i+\widetilde{\uv}_i\big)^r \widetilde{x}_i \big( \widetilde{\uu}_i+\widetilde{\uv}_i\big)^u\Big]  
 \\= &\sum^n_{k=0} \sum^m_{\ell=0} \sum^r_{s=0} \sum^u_{v=0}
\sum_{\substack{q_0,...,q_k\geq 0\\q_0+\dots+q_k=n-k}} \sum_{\substack{p_0,...,p_\ell\geq 0\\p_0+\dots+p_\ell=m-\ell}}
\sum_{\substack{t_0,...,t_s\geq 0\\t_0+\dots+t_s=r-s}}
\sum_{\substack{h_0,...,h_v\geq 0\\h_0+\dots+h_v=u-v}}
\\&   E\Big[ \widehat{\uv}_i^{ q_0}\widehat{\uu}_i\widehat{\uv}_i^{ q_1} \! \dots \widehat{\uu}_i\widehat{\uv}_i^{ q_k} \! \cdot \widehat{x}_i \cdot \widehat{\uv}_i^{ p_0}\widehat{\uu}_i \widehat{\uv}_i^{ p_1} \! \dots \widehat{\uu}_i \widehat{\uv}_i^{ p_\ell} \cdot x_i^2 \Zb^{-1} \!\cdot \widetilde{\uv}_i^{t_0}\widetilde{\uu}_i\widetilde{\uv}_i^{t_1} \! \dots \widetilde{\uu}_i\widetilde{\uv}_i^{t_s} \cdot \widetilde{x}_i \cdot \widetilde{\uv}_i^{h_0}\widetilde{\uu}_i\widetilde{\uv}_i^{h_1} \! \dots \widetilde{\uu}_i\widetilde{\uv}_i^{h_v}
\Big].
\end{align*}
Noting that $\uu_i \prec x_i \prec \uv_i$ over $\B$ and taking into account that $E[\widetilde{x}_i]=E[\widehat{x}_i] =0$, then factorizing as in \eqref{monotone-fact-2}, we prove that the above term is zero whenever $m\neq 0$ or $r\neq 0$. Hence, we get
\begin{align}\label{est:OC-4thmoment}
&E\big[  G_{\uz^0_i}(\Zb)^\ast  x_i  G_{\uz^0_i}(\Zb)^\ast x_i^2 G_{\uz^0_i}(\Zb) x_i  G_{\uz^0_i}(\Zb) \big]\nonumber
\\&  =\sum_{n\geq 0}\sum_{u\geq 0} \sum^n_{k=0}  \sum^u_{v=0}
\sum_{\substack{q_0,...,q_k\geq 0\\q_0+\dots+q_k=n-k}} 
\sum_{\substack{h_0,...,h_v\geq 0\\h_0+\dots+h_v=u-v}} \nonumber
\\& \quad (\Zb^\ast)^{-1} E\Big[ E[\widehat{\uv}_i^{ q_0}] \ \widehat{\uu}_i  E[\widehat{\uv}_i^{ q_1}] \dots \widehat{\uu}_i  E[\widehat{\uv}_i^{ q_k}]  \cdot E[x_i (\Zb^\ast)^{-1}  x_i^2 \Zb^{-1}   x_i ] \cdot \Zb^{-1} E[\widetilde{\uv}_i^{h_0}]  \widetilde{\uu}_i E[\widetilde{\uv}_i^{h_1} ] \dots \widetilde{\uu}_i  E[\widetilde{\uv}_i^{h_v}]
\Big] \nonumber
\\& =  \sum_{k\geq0}  \sum_{v\geq 0}
\sum_{q_0,...,q_k\geq 0} 
\sum_{h_0,...,h_v\geq 0} \nonumber
\\& \qquad\qquad (\Zb^\ast)^{-1} E\Big[ \widehat{\uv}_i^{ q_0} \ \widehat{\uu}_i \ \widehat{\uv}_i^{ q_1} \dots \widehat{\uu}_i \ \widehat{\uv}_i^{ q_k}  \cdot x_i (\Zb^\ast)^{-1}  x_i^2 \Zb^{-1}   x_i  \cdot \Zb^{-1} \widetilde{\uv}_i^{h_0} \ \widetilde{\uu}_i \widetilde{\uv}_i^{h_1}  \dots \widetilde{\uu}_i \ \widetilde{\uv}_i^{h_v}
\Big] \nonumber
\\& =E\big[ G_{\uu_i +\uv_i}(\Zb^\ast) \cdot x_i  (\Zb^\ast)^{-1} x_i^2 \Zb^{-1} x_i \cdot  G_{\uu_i +\uv_i}(\Zb) \big].
\end{align} 
Finally, applying Lemma \ref{Lemma:Monotone-Cauchy} with $x=\uu_i$, $y=\uv_i$ and $W=x_i  (\Zb^\ast)^{-1} x_i^2 \Zb^{-1} x_i$ for $b_1=\Zb^\ast=b_2$, yields
\begin{align*}
  \big\| E\big[  G_{\uz^0_i}(\Zb^\ast ) x_i  G_{\uz^0_i}(\Zb^\ast) x_i^2 G_{\uz^0_i}(\Zb) x_i  G_{\uz^0_i}(\Zb) \big] \big\|
  & = \big\|  E\big[ G_{\uu_i +\uv_i}(\Zb^\ast) \cdot x_i  (\Zb^\ast)^{-1} x_i^2 \Zb^{-1} x_i \cdot  G_{\uu_i +\uv_i}(\Zb) \big] \big\|
  \\& \leq \|\Imm(\Zb)^{-1}\|^2 \cdot \|E[x_i  (\Zb^\ast)^{-1} x_i^2 \Zb^{-1} x_i ] \| 
 \\& \leq \|\Zb^{-1}\|^2 \|\Imm(\Zb)^{-1}\|^2 \ \alpha_4(x) ,
\end{align*}
where $\alpha_4(x)=\max\limits_{1\leq i\leq N}\sup_{} \|E[x_i  b x_i^2 b x_i ] \| $ where the supremum is taken over $b \in \B$ such that $ \|b\|=1 $. 

Putting the above bounds together, we get 
\begin{align*}
\big\|E & \big[  G_{\uz_i}(\Zb)  x_i  G_{\uz^0_i}(\Zb)  x_i  G_{\uz^0_i}(\Zb)  x_i  G_{\uz^0_i}(\Zb) \big]\big\|
\leq \|\Zb^{-1}\| \|\Imm(\Zb)^{-1}\|^3 \sqrt{\alpha_2(x)}\sqrt{\alpha_4(x) }\ .
\end{align*}
Similarly, we prove that 
\begin{align*}
\big\|E & \big[  G_{\uz_{i-1}}(\Zb)  y_i  G_{\uz^0_i}(\Zb)  y_i  G_{\uz^0_i}(\Zb)  y_i  G_{\uz^0_i}(\Zb) \big]\big\|
\leq \|\Zb^{-1}\| \|\Imm(\Zb)^{-1}\|^3 \sqrt{\alpha_2(y)}\sqrt{\alpha_4(y) }\ .
\end{align*}
As the second moments match, we have  $\alpha_2(x)=\alpha_2(y)$, and hence we infer that
\begin{align*}
\|E[C_i(\Zb)]\| &\leq \|\Zb^{-1}\| \|\Imm(\Zb)^{-1}\|^3 \sqrt{\alpha_2(x)}\left(\sqrt{\alpha_4(x)}+\sqrt{\alpha_4(y) }\right) \\
&\leq \|\Imm(\Zb)^{-1}\|^4 \sqrt{\alpha_2(x)}\left(\sqrt{\alpha_4(x)}+\sqrt{\alpha_4(y) }\right). 
\end{align*}
\bigskip
Finally, summing over $i=1,\dots,N$ we get the bound in \eqref{eq1:theo-monotone}.

\vspace{0.25cm}
\noindent
{\bf Proof of \eqref{eq2:theo-monotone}}. 
Starting from \eqref{est:CE-telescopingsum}, we have for any $z \in \mathbb{C}^+$,
\[
\varphi\big[G_{\ux_N}  (z)\big]  - \varphi \big[G_{\uy_N}(z)\big] 
 = \sum^N_{i=1} \Big(\varphi \big[G_{\uz_i}(z) \big( x_i \  G_{\uz^0_i}(z)\big)^3\big] - \varphi \big[G_{\uz_{i-1}}(z)\big(  y_i \  G_{\uz^0_i}(z)\big)^3\big]\Big).
\]
As before, noting that $G_{\uz_i}(z)  x_i G_{\uz^0_i}(z) = G_{\uz^0_i}(z)  x_i G_{\uz_i}(z)$, then \eqref{Cauchy-Schwarz} and \eqref{positivity} and the positivity of $\varphi$ yield 
\begin{align*}
\big|\varphi  \big[ & G_{\uz_i}(z)  x_i  G_{\uz^0_i}(z)  x_i  G_{\uz^0_i}(z)  x_i  G_{\uz^0_i}(z) \big]\big|^2
\\& \leq  \varphi\big[ G_{\uz^0_i}(z) x_i G_{\uz_i}(z)  G_{\uz_i}(z)^\ast x_i \ G_{\uz^0_i}(z)^\ast \big] \cdot
\varphi\big[ G_{\uz^0_i}(z)^\ast x_i \ G_{\uz^0_i}(z)^\ast x_i^2 G_{\uz^0_i}(z) x_i \ G_{\uz^0_i}(z) \big] 
\\ & \leq  \|G_{\uz_i}(z)\|^2  
\cdot  \varphi \big[ G_{\uz^0_i}(z) x_i^2 \ G_{\uz^0_i}(z)^\ast \big]
\cdot
 \varphi \big[ G_{\uz^0_i}(z)^\ast x_i \ G_{\uz^0_i}(z)^\ast x_i^2 G_{\uz^0_i}(z) x_i \ G_{\uz^0_i}(z) \big] .
\end{align*}
Applying Lemma \ref{Lemma:Monotone-Cauchy} (iii) with $x=\uu_i, y=\uv_i$ and $W=x_i^2$, 
\begin{align*}
    \varphi \big[ G_{\uz^0_i}(z) x_i^2 \ G_{\uz^0_i}(z)^\ast \big]
   \leq  \| G_{\uu_i}\big(\varphi[G_{\uv_i}(z)]^{-1}\big)\|_{L_2(\A,\varphi)}^2 \cdot \big| \varphi[x_i^2]\big|.
\end{align*}
 By the same arguments as in \eqref{est:OC-4thmoment}, we get 
 \begin{align*}
\varphi \big[  G_{\uz^0_i}(z)^\ast x_i \ G_{\uz^0_i}(z)^\ast x_i^2 G_{\uz^0_i}(z) x_i \ G_{\uz^0_i}(z) \big]
&= \frac{1}{|z|^2}\varphi \big[ G_{\uu_i +\uv_i}(z)^\ast \cdot  \varphi[x_i^4] \cdot  G_{\uu_i +\uv_i}(z) \big]
\\& \leq \frac{1}{|z|^2} \| G_{\uu_i}\big(\varphi[G_{\uv_i}(z)]^{-1}\big)\|_{L_2(\A,\varphi)}^2 \cdot \big| \varphi[x_i^4]\big|,
 \end{align*}
 where the last inequality follows again by Lemma \ref{Lemma:Monotone-Cauchy} (iii) with $x=\uu_i, y=\uv_i$ and $W=x_i^4$. Putting the above bounds together, we get
 \begin{align*}
\big|\varphi  \big[  G_{\uz_i}(z)  x_i  G_{\uz^0_i}(z)  x_i  G_{\uz^0_i}(z)  x_i  G_{\uz^0_i}(z) \big]\big|
&\leq \frac{\|G_{\uz_i}(z)\|}{|z|}  \|G_{\uu_i}\big(\varphi[G_{\uv_i}(z)]^{-1}\big)\|_{L_2(\A,\varphi)}^2 \sqrt{|\varphi[x_i^4 ] | | \varphi[x_i^2]|}
\\ &\leq \frac{1}{ \Imm(z)^2}  \|G_{\uu_i}\big(\varphi[G_{\uv_i}(z)]^{-1}\big)\|_{L_2(\A,\varphi)}^2 \sqrt{\widetilde{\alpha}_4(x) \alpha_2(x)} ,
\end{align*}
where $\widetilde{\alpha}_4(x):=
\max_{1\leq i\leq N}\sup_{} |\varphi[x_i^4 ] | $ and $\alpha_2(x):=
\max_{1\leq i\leq N}\sup_{} |\varphi[x_i^2 ] | $. Similarly, we prove that 
 \begin{align*}
\big|\varphi  \big[  G_{\uz_{i-1}}(z)  y_i  G_{\uz^0_i}(z)  y_i  G_{\uz^0_i}(z) y_i  G_{\uz^0_i}(z) \big]\big|
&\leq \frac{1}{ \Imm(z)^2} \|G_{\uu_i}\big(\varphi[G_{\uv_i}(z)]^{-1}\big)\|_{L_2(\A,\varphi)}^2 \sqrt{\widetilde{\alpha}_4(y) \alpha_2(y)},
\end{align*}
As the second moments match, we have  $\alpha_2(x)=\alpha_2(y)$, and hence we infer that
\begin{multline*}
   \big|\varphi\big[G_{\ux_N}  (z)\big]  - \varphi \big[G_{\uy_N}(z)\big] \big|
 \\ \leq \frac{1}{ \Imm(z)^2} \sum_{i=1}^N \|G_{\uu_i}\big(\varphi[G_{\uv_i}(z)]^{-1}\big)\|_{L_2(\A,\varphi)}^2 \sqrt{\alpha_2(x)}\left(\sqrt{\widetilde{\alpha}_4(x)} +\sqrt{\widetilde{\alpha}_4(y)}\right). 
\end{multline*}
Note that $\uu_i\prec \uv_i$ implies that $F_{\uu_i}(F_{\uv_i}(z))=F_{\uu_i+\uv_i}(z)$  for all $z\in \mathbb{C}^+$ where $F_x$ is the reciprocal of the Cauchy transform $\G_x$ for any $x=x^*\in\A$. Thus, we have $\G_{\uu_i}(F_{\uv_i}(z))=\G_{\uu_i+\uv_i}(z).$ Therefore, for all $z\in \mathbb{C}^+,$
\begin{align}\label{Ineq: monotone-L^2}
    \|G_{\uu_i}\big(\varphi[G_{\uv_i}(z)]^{-1}\big)\|_{L_2(\A,\varphi)}^2&=-\frac{\Imm \G_{\uu_i}(F_{\uv_i}(z))}{\Imm F_{\uv_i}(z)}= -\frac{\Imm \G_{\uu_i+\uv_i}(z)}{\Imm F_{\uv_i}(z)} \\ \nonumber
    &\leq -\frac{\Imm\G_{\uu_i+\uv_i}(z)}{\Imm(z)}=\|G_{\uu_i+
    \uv_i}(z)\|_{L_2(\A,\varphi)}^2. 
\end{align}
Hence by \eqref{eq:Cauchy-integral} and the above inequality, we infer that
\begin{align*}
\int_\R \big|\Imm \big(\varphi[G_{\ux_N}(t+i\epsilon)]\big)- \Imm \big(\varphi[G_{\uy_N}(t+i\epsilon)]\big)\big| \text{d}t \leq \frac{\pi}{\epsilon^3} \sqrt{\alpha_2(x)}\left(\sqrt{ \widetilde{\alpha}_4(x)} + \sqrt{\widetilde{\alpha}_4(y)}\right) N. 
\end{align*}

\subsection{Infinitesimal case} The main idea of the proof is to bring the problem from the infinitesimal setting to the operator-valued framework where we can use already existing results. More precisely, Proposition \ref{Inf.Prop} allows us to obtain the desired estimates in the infinitesimal free, Boolean, and monotone settings by passing to the operator-valued setting and applying respectively the results in \cite{BannaMaiBerry}[Theorem 3.1], Section \ref{section:Boolean}, and Section \ref{section:Monotone}.

Let $(\A,E,E',\B)$ be an OV $C^*$-infinitesimal probability space and recall the notation Section \ref{section:prel-OVI}. The following Lemma demonstrates precisely how we can, with help of Proposition \ref{Inf.Prop}, pass to the operator-valued setting and still use the Lindeberg method to control the operator-valued infinitesimal Cauchy transform.
\begin{proposition}\label{Inf-Lemma}
Let $x=\{x_1,\dots,x_N\}$ and $y=\{y_1,\dots,y_N\}$ be two infinitesimally independent families of selfadjoint infinitesimally freely/Boolean/monotone independent elements satisfying Assumption \ref{A:Infgeneral}. Then for any given $\Zb\in \bH^+(\B)$, we have
\begin{equation}\label{Inf-telescopingsum}
E'\big[G_{\ux_N}  (\Zb)\big]  - E' \big[G_{\uy_N}(\Zb)\big]
 = \sum^N_{i=1} \Big( E' \big[G_{\uz_i}(\Zb) \big( x_i \  G_{\uz^0_i}(\Zb)\big)^3\big] -E' \big[G_{\uz_{i-1}}(\Zb)\big(  y_i \  G_{\uz^0_i}(\Zb)\big)^3\big]\Big).
\end{equation}
Here  $x$ and $y$ are infinitesimally monotone independent in the sense that
$\A_{x_1}\pprec \A_{x_2}\pprec \cdots \pprec \A_{x_N}\pprec \A_{y_1}\pprec \A_{y_2}\pprec \cdots \pprec \A_{y_N}.$
\end{proposition}
\begin{proof}
Consider $x$ and $y$ to be two self-adjoint families in $\A$. Let $(\widetilde{\A},\widetilde{E},\widetilde{\B})$ be the corresponding upper triangular space of $(\A,E,E',\B)$ and set
$\widetilde{x}=\{\widetilde{x}_1,\dots,\widetilde{x}_N\}$ and $\widetilde{y}=\{\widetilde{y}_1,\dots,\widetilde{y}_N\}$ where
$
\widetilde{x_j}=\begin{bmatrix}x_j & 0 \\ 0 & x_j\end{bmatrix} \text{ and }\widetilde{y_j}=\begin{bmatrix}y_j & 0 \\ 0 & y_j \end{bmatrix} 
$
for each $j=1,\dots,N$. Finally, for each $i=1,\dots,N$,  let
$$
\widetilde{\uz}_i=\sum\limits_{j=1}^i \widetilde{x}_j+\sum\limits_{j=i+1}^N \widetilde{y}_j
\qquad \text{and} \qquad 
\widetilde{\uz}^0_i=\sum\limits_{j=1}^{i-1}\widetilde{x}_j+ \sum\limits_{j=i+1}^N \widetilde{y}_j;
$$  
and note that 
$$
\widetilde{\uz}_N=\begin{bmatrix}\ux_N & 0 \\ 0 & \ux_N \end{bmatrix}:=\widetilde{\ux}_N  
\qquad \text{and} \qquad 
\widetilde{\uz}_0=\begin{bmatrix}\uy_N & 0 \\ 0 & \uy_N \end{bmatrix}:=\widetilde{\uy}_N.
$$
Also note that if $x$ and $y$ are infinitesimally free, then by Proposition \ref{Inf.Prop}, it follows that $\A_{\widetilde{x}_1},\A_{\widetilde{x}_2},$ $\dots,\A_{\widetilde{x}_N},\A_{\widetilde{y}_1},\A_{\widetilde{y}_2},\dots, \A_{\widetilde{y}_N}$ are free over $\widetilde{\B}$. In addition, we observe that under Assumption \ref{A:Infgeneral}, for each $j=1,\dots,N$,
$$
\widetilde{E}[\widetilde{x}_j]=\widetilde{E}\begin{bmatrix}x_j & 0 \\ 0 & x_j\end{bmatrix} = \begin{bmatrix}E[x_j] & E'[x_j] \\ 0 & E[x_j]\end{bmatrix}
=\begin{bmatrix}0 & 0 \\ 0 & 0\end{bmatrix}.$$ Moreover, for any $B=\begin{bmatrix}b & b' \\ 0 & b\end{bmatrix}\in \widetilde{\B}$ and for each $j=1,\dots,N$, we have
$$
\widetilde{E}[\widetilde{x}_jB\widetilde{x}_j]=\begin{bmatrix}E[x_jbx_j] & E'[x_jbx_j]+E[x_jb'x_j] \\ 0 & E[x_jbx_j]\end{bmatrix}\! = \!\begin{bmatrix}E[y_jby_j] & E'[y_jby_j]+E[y_jb'y_j] \\ 0 & E[y_jby_j]\end{bmatrix} = \widetilde{E}[\widetilde{y}_jB\widetilde{y}_j]. 
$$
Therefore, we conclude that $\widetilde{x}$ and $\widetilde{y}$ satisfy the assumption in \cite[Theorem 3.1]{BannaMaiBerry}, which implies that for all $B=\begin{bmatrix}
\Zb & b' \\ 0 & \Zb
\end{bmatrix}\in \widetilde{\B}$ with $\Zb\in \bH^+(\B)$ and $b'\in \B$, 
\begin{equation}\label{wide-E-eqn}
\widetilde{E}\big[G_{\widetilde{\ux}_N}(B)\big]  - \widetilde{E} \big[G_{\widetilde{\uy}_N}(B)\big]
 = \sum^N_{i=1} \Big( \widetilde{E} \big[G_{\widetilde{\uz}_i}(B) \big( \widetilde{x_i} \  G_{\widetilde{\uz}^0_i}(B)\big)^3\big] -\widetilde{E} \big[G_{\widetilde{\uz}_{i-1}}(B)\big(  \widetilde{y_i} \  G_{\widetilde{\uz}^0_i}(B)\big)^3\big]\Big).
\end{equation}
Here we note that all the resolvents above are well-defined; indeed,
$$
G_{\widetilde{\ux}_N}(B) = \begin{bmatrix}
(\Zb-\ux_N)^{-1} & -(\Zb-\ux_N)^{-1}b'(\Zb-\ux_N)^{-1} \\
0 & (\Zb-\ux_N)\end{bmatrix}.
$$
In particular, for a given $\Zb\in \bH^+(\B)$, if we let $B=\begin{bmatrix}\Zb & 0 \\ 0 & \Zb\end{bmatrix}$, then
the left hand side of \eqref{wide-E-eqn} is
$$
\begin{bmatrix}
E[G_{\ux_N}  (\Zb)]  - E [G_{\uy_N}(\Zb)] &
E'[G_{\ux_N}  (\Zb)]  - E' [G_{\uy_N}(\Zb)]  \\
0 & E[G_{\ux_N}  (\Zb)]  - E [G_{\uy_N}(\Zb)]
\end{bmatrix}.
$$
On the other hand, we observe that
$$
\widetilde{E} \big[G_{\widetilde{\uz}_i}(B) ( \widetilde{x_i} \  G_{\widetilde{\uz}^0_i}(B))^3\big] 
= \begin{bmatrix}
E\big[G_{z_i}(\Zb)\big(x_iG_{z_i^0}(\Zb)\big)^3\big] & E'\big[G_{z_i}(\Zb)\big(x_iG_{z_i^0}(\Zb)\big)^3\big] \\
0 & E\big[G_{z_i}(\Zb)\big(x_iG_{z_i^0}(\Zb)\big)^3\big]
\end{bmatrix},
$$
and
$$\widetilde{E} \big[G_{\widetilde{\uz}_{i-1}}(B)\big(  \widetilde{y_i} \  G_{\widetilde{\uz}^0_i}(B)\big)^3\big] 
= 
\begin{bmatrix}
E\big[G_{{z}_{i-1}}(\Zb)\big( {y_i} \  G_{{z}^0_i}(\Zb)\big)^3\big] & E'\big[G_{{z}_{i-1}}(\Zb)\big( {y_i} \  G_{{z}^0_i}(\Zb)\big)^3\big] \\
0 & E\big[G_{{z}_{i-1}}(\Zb)\big( {y_i} \  G_{{z}^0_i}(\Zb)\big)^3\big] 
\end{bmatrix}
$$
for each $i=1,\dots,N$. Therefore, the $(1,2)$-entry of the right hand side of \eqref{wide-E-eqn} is nothing but
$$
\sum^N_{i=1} \Big( E' \big[G_{\uz_i}(\Zb) \big( x_i \  G_{\uz^0_i}(\Zb)\big)^3\big] -E' \big[G_{\uz_{i-1}}(\Zb)\big(  y_i \  G_{\uz^0_i}(\Zb)\big)^3\big]\Big).
$$
By comparing both sides of \eqref{wide-E-eqn}, we conclude that \eqref{Inf-telescopingsum} holds. 

Similarly, if $x$ and $y$ are infinitesimally Boolean (respectively monotone) independent that satisfy Assumption \ref{A:Infgeneral}, then $\widetilde{x}$ and $\widetilde{y}$ are Boolean (respectively monotone) independent that satisfy Assumption \ref{A:general} with respect to $\widetilde{E}$. Therefore, combining the estimates in Section 3.1 and Section 3.2, \eqref{Inf-telescopingsum} also holds whenever $x$ and $y$ are either infinitesimally Boolean or monotone independent.  
\end{proof}

\begin{remark}\label{Inf-Rem}
Note that the fact that $E'$ is completely bounded and self-adjoint implies that $E'=E_1-E_2$ for some completely positive maps $E_1$ and $E_2$. Therefore, for all $a\in \A$, we have
$$
\|E'[a]\|=\|E_1[a]-E_2[a]\|\leq \|E_1[a]\|+\|E_2[a]\|\leq 2\|a\|. 
$$
Suppose $(\widetilde{A},\widetilde{E},\widetilde{B})$ is an upper triangular probability space that is induced by $(\A,E,E',\B)$, then for a given $A=\begin{bmatrix}a & a' \\ 0 & a\end{bmatrix}\in\widetilde{A}$, 
$$\big \| \widetilde{E}\left[A\right]\big\| = \|E[a]\|+\|E'[a]+E[a']\| \leq 3 (\|a\|+\|a'\|) = 3\|A\|. $$
\end{remark}
{\bf Proof of \eqref{eq:theo-Inf Eqn}}
Following Proposition \ref{Inf-Lemma} and Remark \ref{Inf-Rem}, 
\begin{eqnarray*}
\|E'\big[G_{\ux_N}  (\Zb)\big]  - E' \big[G_{\uy_N}(\Zb)\big]\| 
 &\leq&  \sum^N_{i=1}  \Big( \|E' \big[G_{\uz_i}(\Zb) \big( x_i \  G_{\uz^0_i}(\Zb)\big)^3\big] \|+\|E' \big[G_{\uz_{i-1}}(\Zb)\big(  y_i \  G_{\uz^0_i}(\Zb)\big)^3\big]\|\Big) \\
 &\leq& 2\sum^N_{i=1} \left(\|G_{\uz_i}(\Zb) \big( x_i \  G_{\uz^0_i}(\Zb)\big)^3\|+\|G_{\uz_{i-1}}(\Zb) \big( y_i \  G_{\uz^0_i}(\Zb)\big)^3\|\right). 
\end{eqnarray*}
Note that for each $1\leq i\leq N$, 
$$
\|G_{\uz_i}(\Zb) \big( x_i \  G_{\uz^0_i}(\Zb)\big)^3\| \leq \|G_{\uz_i}(\Zb)\|\cdot\|x_i\|^3\cdot  \|G_{\uz^0_i}(\Zb)\|^3 \leq \|\Imm(\Zb)^{-1}\|^4 \|x_i\|^3. 
$$
Similarly, we prove that
$$
\|G_{\uz_{i-1}}(\Zb) \big( y_i \  G_{\uz^0_i}(\Zb)\big)^3\|\leq \|\Imm(\Zb)^{-1}\|^4 \|y_i\|^3.
$$
Thus, we obtain the desired result. 
\end{proof}

\section{Operator-valued Central Limit Theorems} \label{Section 4}
The aim of this section is to provide an application of our main results in Theorems \ref{theo:boolean-monotone} and \ref{theo:monotone-monotone} to the operator-valued central theorems for Boolean and monotone independence respectively. We will show how with our quantitative bounds in terms of the fourth and second operator-valued moments, we can furthermore obtain results on the fourth moment theorem for infinitely divisible measures in Section \ref{section:infinitelydivisiblemeasures}. Finally, we provide in Section \ref{section:Inf-CLT} the first quantitative estimates in the infinitesimal setting. 

Consider an  operator-valued $C^\ast$-probability space $(\A, E, \B)$. All along this section, we let $x:=\{x_1,\dots , x_n\}$ be a family of self-adjoint elements in $\A$ that are centered with respect to $E$ and set 
\[
X_n:=\frac{1}{\sqrt{n}}\sum_{j=1}^n x_j.
\]

\subsection{Boolean CLT}\label{section:BooleanCLT}

The Boolean Central Limit Theorem was first proved in the scalar-valued setting in the original paper by Speicher and Woroudi \cite{speicher1997boolean}. Quantitative versions were provided in \cite{ArSa-Boolean} and \cite{salazar2020berry} in terms of L\'evy distance. In the operator-valued setting, the Boolean CLT was then proved in \cite{BVP-OV-SAB}, where its quantitative extension may be found in \cite{Je-Li-19} or in the notes \cite{Jekel-notes}.

We improve on the above-mentioned quantitative results by providing quantitative estimates in terms of the moments instead of the operator norm that can be pushed to obtaining quantitative bounds on the L\'evy distance.  We start by letting $B_n$ be a centered $\B$-valued Bernoulli element whose variance is given by the completely positive map 
\[
\eta_n : \B \rightarrow \B, \qquad b \mapsto \eta_n(b) = \frac{1}{n} \sum_{j=1}^n E[x_j b x_j].
\]
We provide in the following theorem quantitative results on the level of the operator-valued Cauchy transforms as well as the L\'evy distance.
\begin{theorem}
Let $x:=\{x_1,\dots , x_n\}$ be a family of self-adjoint elements in $\A$ that are Boolean independent with amalgamation over $\B$ and that are such that $E[x_j]=0$. Then, for any $\Zb \in \bH^+(\B)$,
\begin{align*}
\big\| E[G_{X_n}(\Zb)] - E[G_{B_n}(\Zb)] \big\|
\leq \frac{1}{\sqrt{n}}  \|\Imm(\Zb)^{-1}\|^4 \sqrt{\alpha_2(x)}\left(\sqrt{\alpha_4(x)}+\sqrt{\alpha_2(x)^2 }\right) .
\end{align*}
Furthermore, there is a universal positive constant $c
$ such that
\[
L(\mu_{X_n}, \mu_{B_n}) \leq c \big( \alpha_2(x)(\widetilde{\alpha}_4(x)+\alpha_2(x)^2)\big)^{1/14} n^{-1/14}.
\]
\end{theorem}

\begin{proof}
The proof follows by a direct application of Theorem \ref{theo:boolean-monotone} where we choose  $y=\{y_1 , \ldots , y_n\}$ to be a family of $\B$-valued Bernoulli elements that are Boolean independent over $\B$ and are such that $E[y_j] = 0$ and $E[y_j b y_j] = E[x_j b x_j]$ for any $j \in [n]$ and $b \in \B$. Moreover, we note that in this case, we have that
\[
E[y_j b^* y_j^2 b y_j] = E[y_j b^* y_j] E[y_j b y_j].
\] 
As $E$ is a completely positive map, we note that $\sup_{b \in \B, \|b\|=1} \|E[y_j b^* y_j]\|=\|E[y_j^2]\|  $ and hence we get that $\alpha_4(y)=\alpha_2(x)^2$. To end the proof of the first part of the theorem, it remains to notice that $\frac{1}{\sqrt{N}} \sum_{j=1}^n y_j$ is a centered $\B$-valued Bernoulli element with variance $\eta_n$, (See \cite[Lemma 6.2.5]{Jekel-notes}). 

To prove the bound on the L\'evy distance, we apply similarly \eqref{eq2:theo-BM} which yields for any $\epsilon>0$
\[
\int_\R \big|\Imm \big(\varphi[G_{X_n}(t+i\epsilon)]\big)- \Imm \big(\varphi[G_{B_n}(t+i\epsilon)]\big)\big| \text{d}t \leq \frac{1}{\sqrt{n}}\frac{\pi}{\epsilon^3} \sqrt{\alpha_2(x)} \left(\sqrt{ \widetilde{\alpha}_4(x)} + \sqrt{\alpha_2(x)^2}\right).
\]
Then using the bound in \eqref{eq:Levy_bound}, we get for any $\epsilon>0$
\[
L(\mu_{X_n}, \mu_{B_n}) \leq 2\sqrt{\frac{\epsilon}{\pi}} +  \frac{1}{\sqrt{n}}\frac{1}{\epsilon^3} \sqrt{\alpha_2(x)} \left(\sqrt{ \widetilde{\alpha}_4(x)} + \sqrt{\alpha_2(x)^2}\right).
\]
Finally, optimizing over $\epsilon>0$, we get the desired bounds on the L\'evy distance. 
\end{proof}

\begin{theorem}\label{theo:OV_Bernoulli_comparison}
 Let $(\A,E,\B)$ be an operator-valued $C^\ast$-probability space. Consider two operator-valued Bernoulli elements $B_0,B_1$ with respective covariance maps $\eta_0,\eta_1: \B \to \B$. Then, for every $k\in\bN$ and each $\Zb\in \bH^+(M_k(\B))$, we have
\begin{equation}\label{eq:OVB_comparison_Cauchy}
\|\G^{M_k(\B)}_{\1_k \otimes B_1}(\Zb) - \G^{M_k(\B)}_{\1_k \otimes B_0}(\Zb)\| \leq k \|\Imm(\Zb)^{-1}\|^3 \|\eta_1-\eta_0\|.
\end{equation}
Moreover, if $(\A,\varphi)$ is a $W^*$-probability space with $\varphi=\varphi\circ E$, then the scalar-valued Cauchy transforms of $B_1$ and $B_0$ satisfy
\begin{equation}\label{eq:OV-Bernoulli_integral}
\frac{1}{\pi} \int_\R |\G_{B_1}(t+i\epsilon) - \G_{B_0}(t+i\epsilon)|\, \mathrm{d} t \leq \frac{1}{\epsilon^2} \|\eta_1 - \eta_0\|
\end{equation}
for each $\epsilon>0$ and, with the universal positive constant
$c = 5(\frac{1}{4\pi})^{2/5} <  1.817$, we have that
\begin{equation}\label{eq:OVB_comparison_Levy}
L(\mu_{B_1},\mu_{B_0}) \leq c \|\eta_1-\eta_0\|^{1/5}.
\end{equation}
\end{theorem}
While the proof is similar to that of the free case in \cite[Theorem 3.5]{BannaMaiBerry}, we write it again for the convenience of the reader as its arguments will be used repeatedly in the coming sections.

\begin{proof}
 Fix $m\in \mathbb{N}$ and let $x=\{x_1,\dots,x_m\}$ and $y=\{y_1,\dots,y_m\}$ be two Boolean independent families consisting of Boolean independent copies of the given Bernoulli elements $\frac{1}{\sqrt{m}}B_0$ and $\frac{1}{\sqrt{m}}B_1$ respectively. Adopting the notation in Proposition \ref{prop:Lindeberg}, we note that $\uz_m$ and $\uz_0$ have the same distributions as $B_0$ and $B_1$ respectively. Hence following  \eqref{telescoping-sum}, we obtain
\begin{eqnarray*}
E[G_{B_0}(\Zb)]-E[G_{B_1}(\Zb)] &=& \frac{1}{m} \sum\limits_{i=1}^m \left( E\big[G_{\uz^0_i}(\Zb) \big( x_i \  G_{\uz^0_i}(\Zb)\big)^2\big] - E\big[G_{\uz^0_i}(\Zb) \big( y_i \  G_{\uz^0_i}(\Zb)\big)^2\big]\right) \\
&+& \frac{1}{m\sqrt{m}}\sum\limits_{i=1}^m \left( E\big[G_{\uz_i}(\Zb) \big( x_i \  G_{\uz^0_i}(\Zb)\big)^3\big] - E\big[G_{\uz_{i-1}}(\Zb) \big( y_i \  G_{\uz^0_i}(\Zb)\big)^3\big]\right).
\end{eqnarray*}
Then for each $i\in [m]$, we follow the steps of \eqref{2ndorder-B} to obtain
\begin{eqnarray*}
E\big[G_{\uz^0_i}(\Zb) x_i  G_{\uz^0_i}(\Zb)  x_i   G_{\uz^0_i}(\Zb)\big] =E[G_{\uz^0_i}(\Zb) ]E[x_ib^{-1}x_i]E[G_{\uz^0_i}(\Zb) ] 
\end{eqnarray*}
Hence, 
\begin{align*}
\big\|E\big[G_{\uz^0_i}(\Zb) \big( x_i \  G_{\uz^0_i}(\Zb)\big)^2\big] - E\big[G_{\uz^0_i}(\Zb) \big( y_i \  G_{\uz^0_i}(\Zb)\big)^2\big]\big\| 
&\leq \big\|E[G_{\uz^0_i}(\Zb) ]\big\| \big\| (\eta_0-\eta_1)(b^{-1}) \big\| E[G_{\uz^0_i}(\Zb) ] \big\| \\
&\leq \|b^{-1}\|\|\Imm(b)^{-1}\|^2  \|\eta_0-\eta_1\|.
\end{align*}
As for the third order term, we get for each $i\in [m]$,
\begin{equation*}
\big\|E\big[G_{\uz_i}(\Zb) \big( x_i \  G_{\uz^0_i}(\Zb)\big)^3\big]\big\|+\big\|E\big[G_{\uz_{i-1}}(\Zb) \big( y_i \  G_{\uz^0_i}(\Zb)\big)^3\big]\big\| \leq 2\|\Imm(\Zb)\|^4 \left(\|x_i\|^3+\|y_i\|^3\right).
\end{equation*}
Therefore, we obtain 
\begin{multline}\label{eqn:Bernoulli-diff-variance}
\|E[G_{B_0}(\Zb)]-E[G_{B_1}(\Zb)]\|  \\
\leq \|\Zb^{-1}\| \|\Imm(\Zb)^{-1}\|^2  \|\eta_0-\eta_1\|+\frac{2}{\sqrt{m}}\|\Imm(\Zb)^{-1}\|^4(\max_{1\leq i\leq m}\|x_i\|^3+\max_{1\leq i\leq m}\|y_i\|^3).
\end{multline}
Finally, noting that \eqref{eqn:Bernoulli-diff-variance} holds for any $m$, we let $m\to \infty$ to obtain the desired bounded for $k=1$. The proof of the assertion for general $k \in \mathbb{N}$ can be easily proved by noting that Boolean independence is preserved under matrix amplification and then applying \eqref{eqn:Bernoulli-diff-variance} to $1_k\otimes B_1$ and $1_k\otimes B_0$ in the operator-valued $C^*$-probability space $(M_k(\A),id_k\otimes E,M_k(\B))$. Hence, we directly obtain for all $\Zb \in \bH^+(M_k(\B))$,
\[
\| \G^{M_k(\B)}_{\1_k \otimes B_0}(\Zb) - \G^{M_k(\B)}_{\1_k \otimes B_1}(\Zb)\| \leq  \|\Imm(\Zb)^{-1}\|^3 \|\id_k \otimes \eta_0 - \id_k \otimes \eta_1\|.
\]
Using the fact that $\| \id_k \otimes \eta_0 - \id_k \otimes \eta_1 \| \leq k \| \eta_0 - \eta_1 \|$, which follows from \cite[Exercise 3.10]{Paulsen}, we arrive at the desired bound for general $k$.
Finally, to prove \eqref{eq:OV-Bernoulli_integral} and \eqref{eq:OVB_comparison_Levy}, we note that for each $i \in [m]$ and $z\in \mathbb{C}^+$,
\begin{eqnarray*}
\big|\varphi\big[G_{\uz^0_i}(z)\big](\eta_0-\eta_1)(z)\varphi\big[G_{\uz^0_i}(z)\big]\big|\leq \frac{1}{|z|}\|\eta_0-\eta_1\|  \big |\varphi\big[G_{\uz^0_i}(z)\big]\big|^2 \leq \frac{\|\eta_0-\eta_1\|}{\Imm(z)} \big\|G_{\uz^0_i}(z)\big\|_{L^2}^2.
\end{eqnarray*}
Then, the remaining of the proof follows the analogous argument of proof in \cite[Theorem 3.5]{BannaMaiBerry}. 
\end{proof}

\subsection{Monotone CLT}\label{section:monotoneCLT} This section is devoted to studying the operator-valued monotone CLT.  The scalar-valued case was first proved by \cite{muraki2001monotonic} before getting extended to the operator-valued setting, see \cite{popa2013non} and \cite{hasebe2014operator}. Moreover, quantitative versions can be also found in \cite{arizmendi2021berry} in the scalar case and in \cite{Je-Li-19} and \cite{Jekel-notes} in the operator-valued case.

It is similar to Boolean case in the previous subsection that our main improvement is providing quantitative estimates in terms of moments instead of the operator norm, and it allows us to get the quantitative bounds on L\'evy distance. 

For any $1\leq j \leq n$, we denote by $\nu_{\eta_j}$ the $\B$-valued arcsine distribution with variance given by the completely positive map
\[
\eta_j : \B \rightarrow \B, \qquad b \mapsto \eta_j(b) = E[x_j b x_j].
\]
We denote by $A_n$ the $\B$-valued generalized arcsine element of $\A$ whose distribution is given by $\mathrm{dil}_{n^{-1/2}} (\nu_{\eta_1}\rhd \dots \rhd \nu_{\eta_n})$. For more details on generalized arcsine distributions, see \cite[Chapter 7]{Jekel-notes}. 
\begin{theorem}\label{tho:M-CLT}
Let $x:=\{x_1,\dots , x_n\}$ be a family of self-adjoint elements in $\A$ that are monotone independent with amalgamation over $\B$, i.e. $\A_{x_1} \prec \A_{x_2} \prec \dots \prec \A_{x_N}$ over $\B$, and that are such that $E[x_j]=0$. Then, for any $\Zb \in \bH^+(\B)$,
\begin{align*}
\big\| E[G_{X_n}(\Zb)] - E[G_{A_n}(\Zb)] \big\|
\leq \frac{1}{\sqrt{n}}  \|\Imm(\Zb)^{-1}\|^4 \sqrt{\alpha_2(x)}\left(\sqrt{\alpha_4(x)}+\sqrt{\frac{3}{2}\alpha_2(x)^2 }\right) .
\end{align*}
Furthermore, in the scalar case where $\B=\mathbb{C}$, there is a universal positive constant $c
$ such that
\[
L(\mu_{X_n}, \mu_{A_n}) \leq c \Big( \alpha_2(x)\big(\widetilde{\alpha}_4(x)+\frac{3}{2}\alpha_2(x)^2\big)\Big)^{1/14} n^{-1/14}.
\]
\end{theorem}
\begin{proof}
The proof follows by a direct application of Theorem \ref{theo:monotone-monotone} where we choose  $y=\{y_1 , \ldots , y_n\}$ to be a family of $\B$-valued arcsine elements that are such that $\A_{x_1} \prec \A_{x_2} \prec \dots \prec \A_{x_N} \prec \A_{y_1} \prec \A_{y_2} \prec \dots \prec \A_{y_N}$  over $\B$, $E[y_j] = 0$ and $ E[y_j b y_j]=E[x_j b x_j]$ for any $j \in [n]$ and $b \in \B$. Now note that since $y_j$ is a $\B$-valued arcsine element, then
\[
E[y_jb^*y_j^2by_j]
= E[y_jb^*y_j]E[y_jby_j] +\frac{1}{2} E\big[y_jb^* E[y_j^2] b y_j\big] = E[x_jb^*x_j]E[x_jbx_j] +\frac{1}{2} E\big[x_jb^* E[x_j^2] b x_j\big]
\]
where we have used the fact that the moments of the second order of $y_j$ and $x_j$ match. Using the fact that $E$ is positive and that $\sup\|E[x_jbx_j]\|=\|E[x_j^2]\|$ where the supremum is taken over $b\in \B$ with $\|b\|=1$, we conclude that 
$\alpha_4(y)
\leq \frac{3}{2} \alpha_2(x)^2 .$

Setting $A_n=\frac{1}{\sqrt{n}}\sum_{j=1}^n y_j$, it remains to notice that $A_n$ is a $\B$-valued generalized arcsine element whose distribution is given by $\mathrm{dil}_{n^{-1/2}}(\nu_{\eta_1}\rhd \dots \rhd \nu_{\eta_n})$ where $\nu_{\eta_j}$ is the arcsine distribution of variance $\eta_j$. 

To prove the bound on the L\'evy distance, we apply similarly \eqref{eq2:theo-monotone} which yields for any $\epsilon>0$
\[
\int_\R \big|\Imm \big(\varphi[G_{X_n}(t+i\epsilon)]\big)- \Imm \big(\varphi[G_{A_n}(t+i\epsilon)]\big)\big| \text{d}t \leq \frac{1}{\sqrt{n}}\frac{\pi}{\epsilon^3} \sqrt{\alpha_2(x)} \left(\sqrt{ \widetilde{\alpha}_4(x)} + \sqrt{\frac{3}{2}\alpha_2(x)^2}\right).
\]
Then using the bound in \eqref{eq:Levy_bound}, we get for any $\epsilon>0$,
\[
L(\mu_{X_n}, \mu_{A_n}) \leq 2\sqrt{\frac{\epsilon}{\pi}} +  \frac{1}{\sqrt{n}}\frac{1}{\epsilon^3} \sqrt{\alpha_2(x)} \left(\sqrt{ \widetilde{\alpha}_4(x)} +\sqrt{\frac{3}{2} \alpha_2(x)^2}\right).
\]
Finally, optimizing over $\epsilon>0$, we get the desired bounds on the L\'evy distance. 
\end{proof}

\begin{theorem}
Let $(\A,E,\B)$ be an operator-valued $C^\ast$-probability space and consider two $\B$-valued arcsine elements $A_0,A_1$ with respective covariance maps $\eta_0,\eta_1: \B \to \B$. Then, for every $k\in\bN$ and each $\Zb\in \bH^+(M_k(\B))$, we have 
\begin{equation}\label{eqn: OV arcsine-ineq}
\|\G^{M_k(\B)}_{\1_k \otimes A_1}(\Zb) - \G^{M_k(\B)}_{\1_k \otimes A_0}(\Zb)\| \leq k \|\Imm(\Zb)^{-1}\|^3 \|\eta_1-\eta_0\|.
\end{equation}
Moreover, if $(\A,\varphi)$ is a $W^\ast$-probability space, then the scalar-valued Cauchy transforms of $A_1$ and $A_0$ satisfy
\begin{equation}\label{eq:opval_arcsine_integral}
\frac{1}{\pi} \int_\R |\G_{A_1}(t+i\epsilon) - \G_{A_0}(t+i\epsilon)|\, \mathrm{d} t \leq \frac{1}{\epsilon^2} \|\eta_1 - \eta_0\|
\end{equation}
for each $\epsilon>0$ and, with the universal positive constant
$c = 5(\frac{1}{4\pi})^{2/5} <  1.817$, we have that
\begin{equation}\label{eq:opval_arscine_Levy}
L(\mu_{A_1},\mu_{A_0}) \leq c \|\eta_1-\eta_0\|^{1/5}.
\end{equation}
\end{theorem}
\begin{proof}
Fix $m\in \mathbb{N}$ and let $x=\{x_1,\dots,x_m\}$ and $y=\{y_1,\dots,y_m\}$ be two families whose elements are arcsine elements with variances  $\frac{1}{m}\eta_0$ and $\frac{1}{m}\eta_1$ respectively and that are such that 
$\A_{x_1}\prec \cdots \prec \A_{x_m}\prec \A_{y_1} \prec \cdots \prec \A_{y_m}$. The estimate in \eqref{eqn: OV arcsine-ineq} follows by similar arguments as for proving \eqref{eq:OVB_comparison_Cauchy}. We shall still illustrate in the following how to control the second-order term where the main difference lies.  Following the lines of the proof of \eqref{eq1:theo-monotone} and using $(i)$ of Lemma \ref{Lemma:Monotone-Cauchy}, we have
\begin{eqnarray*}
E\big[G_{\uz^0_i}(\Zb) \big( x_i \  G_{\uz^0_i}(\Zb)\big)^2\big] &=&  E\big[G_{\uu_i+\uv_i}(\Zb)x_i\Zb^{-1}x_iG_{\uu_i+\uv_i}(\Zb)\big] \\
&=& E\big[G_{\uu_i}(E[G_{\uv_i}(\Zb)]^{-1})\big]E[x_i\Zb^{-1}x_i] E\big[G_{\uu_i}(E[G_{\uv_i}(\Zb)]^{-1})\big] 
\end{eqnarray*}
where 
$\uu_i= \sum_{j=1}^{i-1} x_j$ and $\uv_i = ~\sum_{j=i+1}^{m} y_j$. 
Similarly, 
$$
E\big[G_{\uz^0_i}(\Zb) \big( y_i \  G_{\uz^0_i}(\Zb)\big)^2\big]= E\big[G_{\uu_i}(E[G_{\uv_i}(\Zb)]^{-1})\big]E[y_ib^{-1}y_i] E\big[G_{\uu_i}(E[G_{\uv_i}(\Zb)]^{-1})\big], 
$$
Hence, 
\begin{align}\label{montone-amplification}
\big\|E\big[G_{\uz^0_i}(\Zb) \big( x_i \  G_{\uz^0_i}(\Zb)\big)^2\big] -& E\big[G_{\uz^0_i}(\Zb) \big( y_i \  G_{\uz^0_i}(\Zb)\big)^2\big]\big\| \nonumber \\
&= \big\|E\big[G_{\uu_i}(E[G_{\uv_i}(\Zb)]^{-1})\big] (\eta_0-\eta_1)(\Zb^{-1}) E[G_{\uu_i}(E\big[G_{\uv_i}(\Zb)]^{-1})\big] \big\| \nonumber \\
&\leq  
\big\|\Imm(E[G_{\uv_i}(\Zb)]^{-1})^{-1}\big\|^2\cdot\| \eta_0-\eta_1\| \|\Zb^{-1}\| \nonumber \\
&\leq \|\Zb^{-1}\|\|\Imm(\Zb)^{-1}\|^2  \|\eta_0-\eta_1\|,
\end{align}
where we have used again the fact that $\Imm\big(E[G_{y}(\Zb)]^{-1}\big) \geq \Imm (\Zb)$ for any $\Zb \in \bH^+(\B)$.  This proves the assertion in \eqref{eqn: OV arcsine-ineq} for the case $k=1$. Noting that, by Proposition \ref{prop:lift-monotoneIND-to-matrices}, monotone independence is preserved under amplification with the identity matrix, we prove the assertion for general $k \in \mathbb{N}$, by applying \eqref{montone-amplification} for $1_k\otimes A_1$ and $1_k\otimes A_0$ in the framework of the operator-valued $C^*$-probability space $(M_k(\A),id_k\otimes E,M_k(\B))$. Hence, we get for all $\Zb \in \bH^+(M_k(\B))$,
\[
\| \G^{M_k(\B)}_{\1_k \otimes A_0}(\Zb) - \G^{M_k(\B)}_{\1_k \otimes A_1}(\Zb)\| \leq  \|\Imm(\Zb)^{-1}\|^3 \|\id_k \otimes \eta_0 - \id_k \otimes \eta_1\|.
\]
Using the fact that $\| \id_k \otimes \eta_0 - \id_k \otimes \eta_1 \| \leq k \| \eta_0 - \eta_1 \|$, which follows from \cite[Exercise 3.10]{Paulsen}, we arrive at the desired bound for general $k$.
Now to prove \eqref{eq:opval_arcsine_integral} and \eqref{eq:opval_arscine_Levy}, we observe that for $z\in \mathbb{C}^+,$
\begin{eqnarray*}
\big|\varphi\big[G_{\uz^0_i}(z)\big](\eta_0-\eta_1)(z)\varphi\big[G_{\uz^0_i}(z)\big]\big| 
&\leq& \frac{1}{\Imm(z)} \|\eta_0-\eta_1\| \big| \varphi\big[G_{\uu_i}(\varphi[G_{\uv_i}(z)]^{-1})\big] \big|^2  \\
&\leq&  \frac{1}{\Imm(z)} \|\eta_0-\eta_1\|  \|G_{\uz^0_i}(z)\|_{L^2(\A,\varphi)}^2.
\end{eqnarray*}
Note that the last inequality holds due to the Cauchy Schwarz inequality and following \eqref{Ineq: monotone-L^2}: 
$\|G_{\uu_i}(\varphi[G_{\uv_i}(z)]^{-1})\|_{L^2(\A,\varphi)}\leq\|G_{\uu_i+\uv_i}(z)\|_{L^2(\A,\varphi)}$.
Then, the remaining argument follows the analogous proof in \cite[Theorem 3.5]{BannaMaiBerry}. 
\end{proof}

\subsection{Fourth moment theorem for monotone infinitely divisible measures}\label{section:infinitelydivisiblemeasures}

Fourth moment theorems refer to a simplification of the moment method to prove the weak convergence of a sequence of random variables $y_n$ to a given random variable $y$. Such  theorems state that if the fourth moment of $y_n$, $\varphi[y_n^4]$, approaches the fourth moment of a random variable $y$,  $\varphi[y^4]$, then $y_n\to y$ weakly. In certain cases, one can even quantify such convergence in terms of the difference $|\varphi[y_n^4]-\varphi[y^4]|$. 

Here we are concerned with the class of infinitely divisible measures with respect to the monotone convolution, which we denote by $ID(\rhd)$. This completes the main theorems from \cite{arizmendijaramillo2014} where the authors give a quantitative version of the results in \cite{arizmendi2013convergence}, in the cases of  free and tensor independence. A fourth moment theorem was given in \cite{ArSa-Boolean} for the  Boolean case keeping the monotone case open. Our results complete the picture by including the monotone case and proving quantitative bounds for the fourth moment theorem in this setting. 

The main observation is that a  Berry-Esseen estimate may be used to prove such a theorem. Indeed, using the fact that $\varphi[x^4]\geq \varphi[x^2]^2$, we deduce by  Theorem \ref{tho:M-CLT} that if the $x_i$'s are monotonically independent and identically distributed variables with mean $0$, variance $1$ and finite fourth moment $\varphi[x_i^4]$, then 
\begin{equation}
    \label{eq:ineq4}
L(\mu_{X_n}, \mu_{A_n}) \leq K \Big(\frac{\varphi[x_i^4]}{n}\Big)^{1/14},
\end{equation}
for some constant $K>0$. In order to connect the above bound to a fourth moment theorem we will use the monotone cumulants. It was shown by Hasebe and Saigo \cite{hasebesaigo2011} that the sequence of monotone cumulants $\{h_n\}_{n\geq1}
$ satisfies  the property that for any $m\in \mathbb{N}$:  $h_n(\sum_{i=1}^mx_i,\dots,\sum_{i=1}^mx_i)=mh_n(x_1, \dots, x_1),$ whenever $\{x_i\}^m_{i=1}$ is a collection of identically distributed monotone independent variables. The first four cumulants are written in terms of the first four moments as follows:
\begin{equation}\label{monotone1}
\begin{split}
h_1(x)&=\varphi[x],\\
h_2(x,x)&=\varphi[x^2]-\varphi[x]^2,\\
h_3(x,x,x)&=\varphi[x^3]-\frac{5}{2}\varphi[x^2]\varphi[x]+\frac{3}{2}\varphi[x]^3,\\
h_4(x,x,x,x)&=\varphi[x^4]-\frac{3}{2}\varphi[x^2]^2-3\varphi[x^3]\varphi[x]+\frac{37}{6}\varphi[x^2]\varphi[x]^2-\frac{8}{3}\varphi[x]^4. 
\end{split}
\end{equation}

Note that if we normalize so that $\varphi[x]=0$ and $\varphi[x^2]=1$, then  $h_2(x,x)=\varphi[x^2]=1$ and $h_4(x,x,x,x)=\varphi[x^4]- \frac{3}{2}\varphi[x^2]^2=\varphi[x^4]-1.5$. Having this in hand, we now state our fourth moment theorem for the scalar-valued monotone case, which is completely new.  The key point is that now we have a bound in terms of the second and fourth moment, which was not the case in previous works, see \cite[Theorem 1.1] {arizmendi2021berry}.

\begin{theorem}
Let $Y$ be a monotone infinitely divisible random variable with mean $0$ and variance $1$ and $A$ be a standard arcsine element. Then $L(\mu_{Y},\mu_A)\leq K  (\varphi[Y^4]-1.5)^{1/14}$, 
for some $K>0$. 
\end{theorem}

\begin{proof}
Since $Y$ is infinitely divisible, then for each $n\in \mathbb{N}$, we may write $Y$ as a sum of $n$ identically monotone independent random variables $$Y=y_1+\dots+ y_n=\frac{1}{\sqrt{n}}(\sqrt{n}y_1+\dots +\sqrt{n}y_n).$$ Noting that  $\varphi[(\sqrt{n}y_i)^2]=h_2(\sqrt{n}y_i,\sqrt{n}y_i)=n h_2(y_i,y_i)=h_2(Y,Y)=1$ and  $h_4(\sqrt{n}y_i)=n^2h_4(y_i,y_i,y_i,y_i)=nh_4(Y,Y,Y,Y), $ then $$\varphi[(\sqrt{n}y_i)^4]=h_4(\sqrt{n}y_i,\sqrt{n}y_i,\sqrt{n}y_i,\sqrt{n}y_i)+1.5=nh_4(Y,Y,Y,Y)+1.5\, .$$

For a standard arcsine element $A$, we may assume, without loss of generality,  that $A=A_n=\frac{1}{\sqrt{n}}(x_1+\dots+x_n)$ where the $x_i$'s are also standard arcsine variables. With this observation, Theorem \ref{tho:M-CLT} holds and hence the bound \eqref{eq:ineq4} yields that
\begin{eqnarray*} L(\mu_Y,\mu_A)&=&L(\mu_Y,\mu_{A_{n}})
\leq K\left(\frac{\varphi[(\sqrt{n}y_i)^4]}{n}\right)^{1/14}\\
&=& K \left(\frac{h_4(\sqrt{n}y_i,\sqrt{n}y_i,\sqrt{n}y_i,\sqrt{n}y_i)+1.5}{n}\right)^{1/14}
=K \left(h_4(Y,Y,Y,Y)+\frac{1.5}{n}\right)^{1/14}.
\end{eqnarray*}
Since $n$ is arbitrary, we let it go to infinity to obtain as desired
$$L(\mu_A,\mu_Y)\leq K(h_4(Y,Y,Y,Y))^{1/14}=K(\varphi[Y^4]-1.5)^{1/14}.$$ 
\end{proof}

In the case of operator-valued monotone independent variables, we can give a bound for the difference of the Cauchy transforms in terms of the fourth $\B$-valued cumulants. 

For this aim, we need analogue operator-valued formulas to those in \eqref{monotone1}. As before, when we restrict to the centered case, i.e. $E[x]=0$, the formulas simplify and we get that
$h_1^{\B}(x)=E(x)=0$, $h_2^{\B}(xb,x)=E[xbx]=\eta(b)$, $h_3^{\B}(xb_1,xb_2,x)=E[xb_1xb_2x]$ and $$h_4^{\B}(xb_1,xb_2,xb_3,x)=E[xb_1xb_2xb_3x]-E[xb_1x]b_2E[xb_3x]- \frac{1}{2} E\left[xb_1E[xb_2x]b_3x\right].$$

\begin{theorem}\label{theo:OV-Monotone-4th-moment-theorem}
Let $(\A,E,\B)$ be a $C^\ast$-probability space and $Y$ be monotone $n$-divisible element over $\B$ with $E[Y]=0$ and $E[YbY]=\eta(b)$.  Then, for  any any $k \in \mathbb{N}$ and $\Zb \in \bH^+(M_k(\B))$,
 $$
\big\| E[G^{M_k(\B)}_{1_k \otimes Y}(\Zb)] - E[G^{M_k(\B)}_{1_k \otimes A}(\Zb)] \big\|
\leq k \|\Imm(\Zb)^{-1}\|^4\sqrt{\|\eta(1)\|}\sqrt{\sup \|h_4^{\B}(Yb,Y,Yb',Y)\|}, $$
where $A$ is the $\B$-valued arcsine element with $E[A]=0$ and $E[ A b  A ]=\eta(b)$ and the supremum is taken over all $b,b' \in \B$ with $ \|b\|
=\|b'\|=1$. 
\end{theorem}
As the above bound holds on the level of the fully matricial extensions of Cauchy transforms, it is sufficient to capture convergence in distribution over $\B$ which yields the following operator-valued monotone fourth moment theorem:
\begin{corollary}
Let $(\A,E,\B)$ be a $C^\ast$-probability space and $(Y_N)_{N \in \mathbb{N}}$ a family of monotone $n$-divisible elements over $\B$ such that $E[Y_N]=0$, $E[Y_NbY_N]=\eta(b)$ and $\sup_{N\in\bN} \|Y_N\| < \infty$. Provided that $\sup \|h_4^{\B}(Y_Nb,Y_N,Y_Nb',Y_N)\|\rightarrow 0 $ as $N\rightarrow \infty$ where the supremum is taken over all $b,b' \in \B$ with $ \|b\|
=\|b'\|=1$ , then $Y_N$ converges in distribution over $\B$ to a $\B$-valued centered arcsine element $A$ with variance $\eta$. 

If, in addition, $(\A,\varphi)$ is a $C^\ast$-probability space and $\varphi=\varphi \circ E$, then the distribution of $Y_N$ with respect to $\varphi$, converges to  that of $A$.
\end{corollary}

\begin{proof}[Proof of Theorem \ref{theo:OV-Monotone-4th-moment-theorem}]
Let $Y$ be monotone $n$-divisible over $\B$. We write $Y$ as a sum of $n$ identically distributed monotone independent random variables $$Y=y_1+\dots + y_n=\frac{1}{\sqrt{n}}(\sqrt{n}y_1+\dots +\sqrt{n}y_n).$$ 
For a $\B$-valued arcsine element $A$, we may assume that $A=A_n=\frac{1}{\sqrt{n}}(x_1+\cdots+x_n)$ where the $x_i's$ are themselves $\B$-valued arcsine variables with mean zero and variance map $\eta$.
With this observation, Theorem \ref{tho:M-CLT} holds and yields following bound: for any $n \in \mathbb{N}$ and any $\Zb \in \bH^+(\B)$,
\begin{align*}
\big\| E[G_{Y}(\Zb)] - E[G_{A}(\Zb)] \big\|
\leq \frac{1}{\sqrt{n}}  \|\Imm(\Zb)^{-1}\|^4 B_n(y) ,
\end{align*}
where $$B_n(y)=\sqrt{\alpha_2(\sqrt{n}y_1)}\left(\sqrt{\alpha_4(\sqrt{n}y_1)}+\sqrt{\frac{3}{2}\alpha_2(\sqrt{n}y_1)^2 }\right). $$
To control the above estimate, we note that 
$$
\alpha_2(\sqrt{n}y_1)=\|E[\sqrt{n}y_11\sqrt{n}y_1]\| = \|h_2^{\B}(\sqrt{n}y_11,\sqrt{n}y_1)\|=\|h_2^{\B}(Y1,Y)\|=\|\eta(1)\|.
$$
Moreover, we observe that 
\begin{align*}
E[\sqrt{n}y_1b^* (\sqrt{n}y_1)^2b\sqrt{n}y_1] &= h_4^{\B}(\sqrt{n}y_1b^*,\sqrt{n}y_1,\sqrt{n}y_1b,\sqrt{n}y_1)+\eta(b^*)\eta(b)+\frac{1}{2}\eta(b^*\eta(1)b) \\
&= nh_4^{\B}(Yb^*,Y,Yb,Y)+\eta(b^*)\eta(b)+\frac{1}{2}\eta(b^*\eta(1)b),
\end{align*}
which yields that
\begin{eqnarray*}
\alpha_4(\sqrt{n}y_1)&=&\sup\limits_{b\in\B,\ \|b\|=1}\|E[\sqrt{n}y_1b^*(\sqrt{n}y_1)^2b\sqrt{n}y_1]\| \\
&\leq& \sup\limits_{b\in\B,\ \|b\|=1} n \|h_4^{\B}(Yb^*,Y,Yb,Y)\|+\|\eta(1)\|^2 + \frac{1}{2}\|\eta(1)\|. 
\end{eqnarray*}
Hence, we  for any $n \in \mathbb{N}$
\begin{align*}
\frac{1}{\sqrt{n}}B_n(y)\! \leq \! \sqrt{\|\eta(1)\|} \!\left(\! \sqrt{\sup\limits_{b\in\B,\ \|b\|=1}\! \|h_4^{\B}(Yb^*,Y,Yb,Y)\|+\frac{\|\eta(1)\|^2}{n} + \frac{\|\eta(1)\|}{2n}}+ \!\sqrt{\frac{3\|\eta(1)\|^2}{2n}}\right).  
\end{align*}
 Finally letting $n\to \infty$ and putting the above bounds together, we  conclude that
\begin{equation}\label{eq:OVmonotone4thmomenttheorem}
    \big\| E[G_{Y}(\Zb)] - E[G_{A}(\Zb)] \big\|
\leq  \|\Imm(\Zb)^{-1}\|^4\sqrt{\|\eta(1)\|}\sqrt{\sup\limits_{b\in\B,\ \|b\|=1} \|h_4^{\B}(Yb^*,Y,Yb,Y)\|}.
\end{equation}

This proves the assertion for $k=1$. To prove it for general $k \in \mathbb{N}$, we note first that, by Proposition \ref{prop:lift-monotoneIND-to-matrices}, monotone independence is preserved under amplification with the identity matrix, we prove the assertion for general $k \in \mathbb{N}$, by applying \eqref{eq:OVmonotone4thmomenttheorem} for $1_k\otimes Y$ and $1_k\otimes A$ in the framework of the operator-valued $C^*$-probability space $(M_k(\A),id_k\otimes E,M_k(\B))$. we first note that by \cite[Lemma 5.4]{BannaMaiBerry}, we have that $\|(id_k \otimes \eta)(1)\| \leq \|\eta(1)\| $. 
Moreover, for any $B \in M_k(\B)$, one can easily check that
\begin{align*}
h_4^{M_k(\B)}\big((1_k \otimes Y)B^*,(1_k \otimes Y),(1_k \otimes Y)B,(1_k \otimes Y)\big) = \sum_{i,j,\ell=1}^k E_{ij} \otimes h_4^{\B}(YB_{i\ell}^*,Y,YB_{\ell j},Y),  
\end{align*}
and hence, by \cite[Exercise 3.10 (i)]{Paulsen}, we get 
\begin{align*}
\Big\|& h_4^{M_k(\B)}\big((1_k \otimes Y)B^*,(1_k \otimes Y),(1_k \otimes Y)B,(1_k \otimes Y)\big) \Big\| \\ & \leq \Big( \sum_{i,j=1}^k \big\| \sum_{\ell=1}^k h_4^{\B}(YB_{i\ell}^*,Y,YB_{\ell j},Y) \big\|^2\Big)^{1/2}
 \leq k^2 \|B\|^2 \sup\limits_{\substack{b\in\B,\ \|b\|=1 \\ b'\in\B,\ \|b'\|=1}} \big\|  h_4^{\B}(Yb,Y,Yb',Y) \big\|.
\end{align*}
From this, we infer that 
\begin{align*}
\sup_{B\in M_k(\B),\ \|B\|=1}&\Big\| h_4^{M_k(\B)}\big((1_k \otimes Y)B^*,(1_k \otimes Y),(1_k \otimes Y)B,(1_k \otimes Y)\big) \Big\| \\ & \leq \Big( \sum_{i,j=1}^k \big\| \sum_{\ell=1}^k h_4^{\B}(YB_{i\ell}^*,Y,YB_{\ell j},Y) \big\|^2\Big)^{1/2}
 \leq k^2 \sup\limits_{\substack{b\in\B,\ \|b\|=1 \\ b'\in\B,\ \|b'\|=1}} \big\|  h_4^{\B}(Yb,Y,Yb',Y) \big\|.
\end{align*}
Putting the above terms together, we prove the assertion for any $k \in \mathbb{N}$.

\end{proof}

\subsection{Infinitesimal CLTs}\label{section:Inf-CLT}

This section is devoted to studying operator-valued infinitesimal central limit theorems, that were introduced in \cite{perales2021operator}, see Theorem \ref{OVI-CLT}. We provide in this paper the first quantitative results in this setting that are on the level of operator-valued infinitesimal Cauchy transforms. 

Suppose that $(\A,E,E',\B)$ is an OV $C^*$-infinitesimal probability space. Let $x:=\{x_1,\dots,x_n\}$ be a family of self-adjoint elements in $\A$ such that $E[x_j]=E'[x_j]=0$ for all $j$ and set 
$$
X_n:=\frac{1}{\sqrt{n}}\sum_{j=1}^N x_j. 
$$
Given the setting we consider, we will denote by $S_n$ a $\B$-valued infinitesimal semicircle, Bernoulli or generalized arcsine element (see \cite[Section 3.3]{tseng2021thesis}) with the infinitesimal variance $(\eta_n,\eta_n')$ given by
$$
\eta_n(b) = \frac{1}{n} \sum_{j=1}^n E[x_j b x_j]\qquad  \text{ and } \qquad \eta'_n(b) = \frac{1}{n} \sum_{j=1}^n E'[x_j b x_j].
$$

We assume that our OVI probability space is rich enough to contain a family $y:=\{y_1,\dots,y_n\}$ of independent $\B$-valued infinitesimal semicircular/Bernoulli/arcsine elements, that is itself  independent of $x$ and is such that $E[y_iby_i]=E[x_ibx_i] $ and $E'[y_iby_i]=E'[x_ibx_i] $ for any $i=1,\dots,n$. In the monotone case, the independence of $y$ from $x$ is assumed to hold in the sense that $\A_{x_1} \pprec \dots \pprec \A_{x_n}\pprec \A_{y_1} \pprec \dots \pprec \A_{y_n}$ over $\B$.

We state now the first quantitative estimates in the infinitesimal setting that follow directly from Theorem \ref{eq:theo-Inf} and that provide analogue estimates to those in Section 4. 

\begin{theorem}
Let $x:=\{x_1,\dots , x_n\}$ be a family of self-adjoint elements in $\A$ that are infinitesimally free/Boolean/monotone independent with amalgamation over $\B$ and that are such that $E[x_j]=E'[x_j]=0$. Then, for any $\Zb \in \bH^+(\B)$,
$$
\big\|E'\big[G_{X_n}  (\Zb)\big]  - E' \big[G_{S_n}(\Zb)\big]\big\|\leq \frac{2}{\sqrt{n}} \|\Imm(\Zb)^{-1}\|^4  \left( \max_{1\leq i\leq N}\|x_i\|^3  + 8\max_{1\leq i\leq N}\|E[x_i^2]\|^{\frac{3}{2}} \right),
$$
where $S_n$ is a $\B$-valued infinitesimal semicircular/Bernoulli/generalized arcsine element with the infinitesimal variance $(\eta_n,\eta_n')$. The bounds can be slightly improved in the Boolean case by removing the factor $8$ on the right-hand side. 
\end{theorem}
\begin{proof}
To obtain the above estimates, we choose the family $y=\{y_1,\dots,y_n\}$ as above and then apply Theorem \ref{eq:theo-Inf}. We give a complete proof for the free case whereas we only indicate what choices of the family $y$ we do in the Boolean and monotone cases. 

For the free case, we choose $y=\{y_1,\dots,y_n\}$ to be a family of $\B$-valued infinitesimal semicircular elements that are infinitesimally free and that are such that  $E[y_j]=E'[y_j]=0$, $E[x_jbx_j]=E[y_jby_j]$, and $E'[x_jbx_j]=E'[y_jby_j]$ for any $b\in \B$ and  $j\in [n]$. Finally, we set $S_n=\frac{1}{\sqrt{n}}\sum_{j=1}^n y_j$, and note that by Lemma \ref{Lemma.Is.elt}, $S_n$ is a $\B$-valued infinitesimal semicircular element with infinitesimal variance $(\eta_n,\eta'_n)$. Hence, Theorem \ref{eq:theo-Inf} yields
$$
\big\|E'\big[G_{X_n}  (\Zb)\big]  - E' \big[G_{S_n}(\Zb)\big]\big\|\leq \frac{2}{\sqrt{n}} \|\Imm(\Zb)^{-1}\|^4  \left( \max_{1\leq i\leq N}\|x_i\|^3  + \max_{1\leq i\leq N}\|y_i\|^3 \right).
$$
Finally, we note that $\|y_i\|\leq 2\|\eta_{y_i}(1)\|^{1/2}=2\|E[x_i^2]\|^{1/2}$ for each $i=1,\dots,n$ (see \cite[Proposition 6.2.1]{Jekel-notes}) which completes the proof of the free case. 

For the Boolean case, we choose $y=\{y_1,\dots,y_n\}$ to be a family of $\B$-valued infinitesimal Bernoulli elements that are infinitesimally Boolean independent over $\B$ with matching moments as described above. Finally, we note that for any $j \in [n]$, $\|y_j\|=\|\eta_{y_j}(1)\|^{1/2}$ which is the reason why the factor  $8$ can be improved to $1$, see \cite[Proposition 6.2.4]{Jekel-notes}).

Finally, for the monotone case, we choose $y=\{y_1,\dots,y_n\}$ to be a family of $\B$-valued infinitesimal arcsine elements that are infinitesimally monotone independent in the sense that $\A_{x_1}\pprec  \cdots \pprec A_{x_n} \pprec\A_{y_1}\pprec \cdots \pprec A_{y_n}$ over $\B$ with matching moments as described above. Finally, we note that for any $j \in [n]$, $\|y_j\|\leq 2 \|\eta_{y_j}(1)\|^{1/2}$, see \cite[Proposition 6.2.6]{Jekel-notes}. 
\end{proof}
\noindent
 \textbf{Conjecture:} \emph{Let $(\A_1,E_1,E_1'), \ldots , (\A_k,E_k,E_k')$ be OVI $C^\ast$-probability spaces over $\B$ and let $(\A,E)$ be their OV $C^\ast$-free/Boolean/monotone product. Then there exists a linear $\B$-bimodule self-adjoint completely bounded map $E':\A \rightarrow \B$ with $E'(1)=0$ such that $\left.E'\right|_{\A_j}=E_j'$ for $1\leq j \leq k$ and such that $\A_1,\dots,\A_k$ are infinitesimally free/Boolean/monotone independent over $\B$. }
\begin{remark}
We give an algebraic construction of the OVI free/Boolean/monotone products in Appendix \ref{Appendix:OVI}.  The mapping $E'$ that we construct, satisfies all the properties in the above conjecture except for complete boundedness which is still open. 
\end{remark}
\begin{theorem}
Let $(\A,E,E',\B)$ be an OV $C^*$-infinitesimal probability space, and let $c_0$ and $c_1$ be infinitesimal semicircular/Bernoulli/arcsine elements with respective infinitesimal variances $(\eta_0,\eta_0')$ and $(\eta_1,\eta_1')$. Provided that the above conjecture holds, then  for any $\Zb \in \bH^+(\B)$,
\begin{equation*}
\|E'[G_{c_0}(\Zb)]-E'[G_{c_1}(\Zb)]\| 
\leq 9 \|\Imm(\Zb)^{-1}\|^3(\|\eta_0-\eta_1\|+\|\eta'_0-\eta'_1\|).
\end{equation*}
\end{theorem}

\begin{proof}
We shall prove only the free case and comment at the end on the Boolean and monotone cases. Let $s_0$ and $s_1$ be infinitesimal semicircular elements with respective infinitesimal variances $(\eta_0,\eta_0')$ and $(\eta_1,\eta_1')$. Given $m\in \mathbb{N}$, we let $x=\{x_1,\dots,x_m\}$ and $y=\{y_1,\dots,y_m\}$ be two infinitesimally free families whose elements are themselves infinitesimally free copies of $\frac{1}{\sqrt{m}}s_0$ and $\frac{1}{\sqrt{m}}s_1$ respectively. We consider the upper triangular probability space $(\widetilde{A},\widetilde{E},\widetilde{B})$ that is induced by $(\A,E,E',\B)$,  set
$$S_0=\begin{bmatrix}s_0 & 0 \\ 0 & s_0\end{bmatrix} \text{ and }S_1=\begin{bmatrix}s_1 & 0 \\ 0 & s_1\end{bmatrix}, 
$$
and finally  define $\widetilde{\eta}_0(B):=\widetilde{E}[S_0BS_0]$ and $\widetilde{\eta}_1(B):=\widetilde{E}[S_1BS_1]$ for any $B\in\widetilde{B}$.

The proof relies as before on the operator-valued Lindeberg method with the only difference that it is now applied with respect to $\widetilde{E}$ instead of $E$. With this aim we set, for each $i\in [m]$,
$$
\widetilde{x}_i=\begin{bmatrix}x_i & 0 \\
0 & x_i\end{bmatrix} \text{ and }\widetilde{y}_i=\begin{bmatrix}y_i & 0 \\ 0 & y_i\end{bmatrix},
$$
and note that the $\widetilde{x}_i$'s and $\widetilde{y}_i$'s are centered with respect to  $\widetilde{E}$ but do not have matching moments of second order. Thus, we have for any $B=\begin{bmatrix}
\Zb & b' \\ 0 & \Zb 
\end{bmatrix} \in \widetilde{\B}$ with $\Zb\in \bH^+(\B)$,
\begin{eqnarray*}
\widetilde{E}[G_{S_0}(B)]-\widetilde{E}[G_{S_1}(B)] 
&=& \frac{1}{m}\sum\limits_{i=1}^m \left( \widetilde{E}[G_{\widetilde{\uz}^0_i}(B)(\widetilde{x}_iG_{\widetilde{\uz}^0_i}(B))^2]-\widetilde{E}[G_{\widetilde{\uz}^0_i}(B)(\widetilde{y}_iG_{\widetilde{\uz}^0_i}(B))^2]\right)  \\
&+& \frac{1}{m\sqrt{m}}\sum\limits_{i=1}^m \left( \widetilde{E}[G_{\widetilde{\uz}_i}(B)(\widetilde{x}_iG_{\widetilde{\uz}^0_i}(B))^3] - \widetilde{E}[G_{\widetilde{\uz}_{i-1}}(B)(\widetilde{y}_iG_{\widetilde{\uz}^0_i}(B))^3] \right)\nonumber
\end{eqnarray*} 
where
$$
\widetilde{\uz}_i=\sum\limits_{j=1}^i \widetilde{x}_j+\sum\limits_{j=i+1}^m \widetilde{y}_j
\qquad \text{and} \qquad 
\widetilde{\uz}^0_i=\sum\limits_{j=1}^{i-1}\widetilde{x}_j+ \sum\limits_{j=i+1}^m \widetilde{y}_j.
$$ 
Hence, 
\begin{multline}
 \|\widetilde{E}[G_{S_0}(B)]-\widetilde{E}[G_{S_1}(B)] \| \\
\leq \frac{1}{m}\sum\limits_{i=1}^m \left\| \widetilde{E}[G_{\widetilde{\uz}^0_i}(B)(\widetilde{x}_iG_{\widetilde{\uz}^0_i}(B))^2]-\widetilde{E}[G_{\widetilde{\uz}^0_i}(B)(\widetilde{y}_iG_{\widetilde{\uz}^0_i}(B))^2]\right\| \\
+ \frac{1}{m\sqrt{m}}\sum\limits_{i=1}^m \left(\left\| \widetilde{E}[G_{\widetilde{\uz}_{i-1}}(B)(\widetilde{x}_iG_{\widetilde{\uz}^0_i}(B))^3] \right\|+\left\| \widetilde{E}[G_{\widetilde{\uz}_{i-1}}(B)(\widetilde{y}_iG_{\widetilde{\uz}^0_i}(B))^3] \right\| \right).\nonumber   
\end{multline}
To control the above terms, we observe first that
$$\|G_{\widetilde{\uz}^0_i}(B)\|=\left\|\begin{bmatrix}(b-\uz^0_i)^{-1} & -(b-\uz^0_i)^{-1}b'(b-\uz^0_i)^{-1} \\ 0 & (b-\uz^0_i)^{-1}\end{bmatrix}\right\|
\leq \|\Imm(b)^{-1}\|(1+ \|b'\|\|\Imm(b)^{-1}\|),
$$
and that if we let $\widetilde{\eta}=\widetilde{\eta}_0-\widetilde{\eta}_1$, $\eta=\eta_0-\eta_1$, and $\eta'=\eta'_0-\eta'_1$, then $\|\widetilde{\eta}\|\leq \|\eta\|+\|\eta'\|$. Indeed, we have
\begin{eqnarray*}
\left\|\widetilde{\eta}(B)\right\| &\leq& \|\eta(b)\|+\|\eta(b')\|+\|\eta'(b)\| \\
                                    &\leq& \|\eta\|\|b\|+\|\eta\|\|b'\|+\|\eta'\|\|b\| \\
                                    &\leq& (\|\eta\|+\|\eta'\|)(\|b\|+\|b'\|)=(\|\eta\|+\|\eta'\|)\|B\|,
\end{eqnarray*}
By Remark \ref{Inf-Rem}, the third order terms are controlled as follows: for each $i \in [m]$,  $$
\left\| \widetilde{E}\big[G_{\widetilde{\uz}_i}(B)(\widetilde{x}_iG_{\widetilde{\uz}^0_i}(B))^3\big] \right\|\leq 3 \|G_{\widetilde{\uz}^0_i}(B)\|^4\|\widetilde{x}_i\|^3 \leq 3 \|\Imm(b)^{-1}\|^4(1+ \|b'\|\|\Imm(b)^{-1}\|)^4 \|x_i\|^3,
$$
and
$$
\left\| \widetilde{E}\big[G_{\widetilde{\uz}_{i-1}}(B)(\widetilde{y}_iG_{\widetilde{\uz}^0_i}(B))^3\big] \right\|\leq 3 \|\Imm(b)^{-1}\|^4(1+ \|b'\|\|\Imm(b)^{-1}\|)^4 \|y_i\|^3.
$$
Following the lines of proof \cite[Theorem 3.1]{BannaMaiBerry} for free case and Remark \ref{Inf-Rem},  we then obtain for each $i \in [m]$,
\begin{align*}
\left\| \widetilde{E}[G_{\widetilde{\uz}_i}(B)(\widetilde{x}_iG_{\widetilde{\uz}^0_i}(B))^2]-\widetilde{E}[G_{\widetilde{\uz}_i}(B)(\widetilde{y}_iG_{\widetilde{\uz}^0_i}(B))^2]\right\| 
&= \left\| \widetilde{E}\left[G_{\widetilde{\uz}^0_i}(B)\left(\widetilde{\eta}_0-\widetilde{\eta}_1\right)(\widetilde{E}[G_{\widetilde{\uz}^0_i}(B)])G_{\widetilde{\uz}^0_i}(B)\right]\right\|\\
&\leq 3\|G_{\widetilde{\uz}^0_i}(B)\|^2 \|\widetilde{\eta}_0-\widetilde{\eta}_1\| \|\widetilde{E}[G_{\widetilde{\uz}^0_i}(B)]\| \\
&\leq 9 \|G_{\widetilde{\uz}^0_i}(B)\|^3\|\widetilde{\eta}_0-\widetilde{\eta}_1\|. 
\end{align*}
Now, choosing $B=\begin{bmatrix}\Zb & 0 \\ 0 & \Zb\end{bmatrix}$ (i.e, $b'=0$), we get 
\begin{eqnarray*}
\|E'[G_{s_0}(\Zb)]-E'[G_{s_1}(\Zb)]\|&\leq& 
\left\| \begin{bmatrix} E[G_{s_0}(\Zb)]-E[G_{s_0}(\Zb)] & E'[G_{s_0}(\Zb)]-E'[G_{s_1}(\Zb)] \\
0 & E[G_{s_0}(\Zb)]-E[G_{s_0}(\Zb)]
\end{bmatrix}\right\| \\
&=& \|\widetilde{E}[G_{S_0}(B)]-\widetilde{E}[G_{S_1}(B)] \|.
\end{eqnarray*}
Therefore, 
\begin{eqnarray}\label{E'CLT-ineq}
&&\|E'[G_{s_0}(\Zb)]-E'[G_{s_1}(\Zb)]\| \nonumber\\ 
&\leq& 9 \|\Imm(\Zb)^{-1}\|^3\|\widetilde{\eta}_0-\widetilde{\eta}_1\| +\frac{3}{\sqrt{m}} \|\Imm(\Zb)^{-1}\|^4 
\big(\max_{1\leq i\leq m}\|x_i\|^3+\max_{1\leq i\leq m}\|y_i\|^3\big).
\end{eqnarray}
Finally, note that \eqref{E'CLT-ineq} holds for any $m\in\mathbb{N}$, then we let $m\to\infty$ to complete the proof. 

For Boolean and monotone cases, we apply the proofs of (\ref{eq1:theo-BM} and \ref{eq1:theo-monotone} respectively,  
\begin{eqnarray*}
\left\| \widetilde{E}\big[G_{\widetilde{\uz}_i}(B)(\widetilde{x}_iG_{\widetilde{\uz}^0_i}(B))^2\big]-\widetilde{E}\big[G_{\widetilde{\uz}_i}(B)(\widetilde{y}_iG_{\widetilde{\uz}^0_i}(B))^2\big]\right\| 
\leq 9 \|G_{\widetilde{\uz}^0_i}(B)\|^2 \|B^{-1}\|\|\widetilde{\eta}_0-\widetilde{\eta}_1\|
\end{eqnarray*}
and note that
$$
\|B^{-1}\|\leq \|\Zb^{-1}\|(1+\|b'\|\|\Zb^{-1}\|) \text{ and }\|G_{\widetilde{\uz}^0_i}(B)\|
\leq \|\Imm(\Zb)^{-1}\|(1+ \|b'\|\|\Imm(\Zb)^{-1}\|).
$$
The remaining arguments follow the analogous ones of the free case. 
\end{proof}

\section{Matrices with Boolean Entries}\label{section:matrices}

Let $(\A, \varphi)$ be a $W^\ast$-probability space and denote by $M_n(\A)$ the algebra of $n \times n$ matrices whose entries are elements of $\A$, i.e. $M_n(\A)=M_n(\C) \otimes \A$.  The aim of this section is to study the distribution of self-adjoint matrices whose entries are Boolean independent with a variance profile. In Section \ref{section:matrices-Bernoulli-entries}, we will first illustrate some properties of $\B$-valued Bernoulli elements and then study distributions of matrices whose entries are in particular Boolean independent \emph{$\eta$-circular elements}. Finally, in Section \ref{section:Wigner-Booelean}, we extend the study to the general Boolean Wigner case. 

\subsection{Matrix Valued Bernoulli}\label{section:matrices-Bernoulli-entries}

 Let $(\A, E, \B)$ be an OV $C^*$-probability space, and $(\A, \varphi)$ be a $C^*$-probability space such that $\varphi=\varphi\circ E$. Let $B$ be a $\B$-valued Bernoulli element with variance map $\eta(b)=E[BbB]$. We illustrate in the following lemma what the scalar-valued distribution of $B$ is. For a measure $\mu$, let us denote by $\mu^{(2)}$ the push forward of the measure $\mu$ along the map $t\to t^2.$

\begin{lemma}\label{Lem: Bernoulli}
Let $B$ be a $\B$-valued Bernoulli element with $\eta(b)=E[BbB]$ as variance map, then for any $n \in \mathbb{N}$
\begin{equation} \label{eq:OVbernoulli} 
    \varphi[B^{2n}]=\varphi[\eta(1)^n] \quad \text{ and } \quad \varphi[B^{2n+1}]=0.
\end{equation}
In particular, if $\mu_B$ and $\mu_{\eta(1)}$ are the distributions of  $B$ and $\eta(1)$ with respect to $\varphi$, respectively, then $\mu_B$ is the unique symmetric measure such that $\mu^{(2)}_B=\mu_{\eta(1)}  $.
\end{lemma}

\begin{proof}
To compute the moments of a $\B$-valued Bernoulli variable $B$, we apply the Boolean moment-cumulant formula and use the fact that the only non-vanishing cumulant of $B$ is  $\beta_2(Bb,B)=\eta(b)$. Since there is only one pair interval partition on $[2n]$, namely $\pi=\{\{1,2\},\dots,\{2n-1,2n\}\}$, we get
$$E[ B^{2n}]=\eta(1)^n \text{ and }  E[B^{2n+1}]=0.$$
Now applying $\varphi$ on both sides, we obtain \eqref{eq:OVbernoulli} and end the proof.
\end{proof}
Having Lemma \ref{Lem: Bernoulli} in hand, we obtain immediately the following result. We give a proof for the convenience of the reader. 

\begin{corollary}
\label{cor OVBernoulli} Let $B$ and $\{B_n\}_n\geq 0$ be $\B$-valued Bernoulli elements with variance maps $\eta$ and $\{\eta_n\}_n\geq 0$, respectively. Then, $\mu_{B_n}\to \mu_{B}$ weakly, if and only if  $\mu_{\eta_n(1)}\to \mu_{\eta(1)}$, weakly.
\end{corollary}

\begin{proof}
To relate  $\mu_{\eta_n(1)}$ and $\mu_{\eta(1)}$, we first note that $\mu_{B_n}([-t,t])=\mu_{\eta_n(1)}([0,t^2])$,  and similarly, $\mu_{B}([-t,t])=\mu_{\eta(1)}([0,t^2])$. If $\mu_{B_n}\to\mu_B$,
then
$\mu_{B_n}([-t,t])\to\mu_{B}([-t,t])$
and then $\mu_{\eta_n(1)}([0,t^2])$ $\to\mu_{\eta(1)}([0,t^2])$, which implies that $\eta_n(1)\to \eta(1)$, since by definition $\eta_n(1)\to \eta(1)$ are measure supported on $\mathbb{R^+}$. The converse follows the same lines by the observations that symmetric measures are determined by the measure on the  intervals of the form $[-t,t]$.

\end{proof}

In order to study matrices with Boolean entries, we consider the quadruple $(M_n(\A), \tr_n\otimes \varphi,$ $ id_n \otimes \varphi, M_n(\mathbb{C}))$ where $(\A,\varphi)$ is a given $W^*$-probability space. For this aim, we call an element $x\in \A$  a \emph{$\eta$-circular element} if $x$ is $\eta$-diagonal (see \cite[Definition 2.6]{bercovici2018eta}) and its Boolean cumulants satisfy in addition the property that $\beta_n(x,x^*,\cdots,x,x^*)=\beta_n(x^*,x,\cdots,x^*,x)=0$ unless $n=2.$

\begin{remark}[Norm of $\eta$-circular operator]\label{Remark-eta-circular}
We note that one can construct $\eta$-circular elements such that $\beta_{2}(A,A^*)=\sqrt{\alpha}$,   $\beta_{2}(A^*,A)=\tilde \alpha$ and $||A||= max(\sqrt{\alpha},\sqrt{\tilde \alpha)}$. Indeed, in \cite[Remark 5.3 ]{bercovici2018eta}, a construction is given for $\eta$-diagonal elements. Their construction is of the form $A=V(X\otimes Y)$, where $V$ is the flip operator, $X= T_1\oplus T_2 $ and $Y$ is a partial isometry. Since the norm of $\|V(X\otimes Y)\|=\|X\otimes Y\|$ and $\|X\otimes Y\|=\|X\|~ \| Y\|$, then $\| V(X\otimes Y)\|=\|X\|$, because $Y$ is a partial isometry. In the specific case of the  $\eta$-circular elements with $\beta_2(A,A^*)=\alpha$ and  $\beta_2(A^*,A)=\tilde \alpha$, we choose $T_1$ and $T_2$ to be the Bernoulli variables given by  $T_1=\begin{bmatrix} 0&  \sqrt{\alpha} \\  \sqrt{\alpha} & 0\end{bmatrix}$ and  $T_2=\begin{bmatrix} 0&  \sqrt{\tilde{\alpha}} \\ \sqrt{\tilde{\alpha}} & 0 \end{bmatrix}$. Hence, 
$\|X\|=\|T_1\otimes T_2\|=\max \{\|T_1\|,\|T_2\|\}=max(\sqrt{\alpha},\sqrt{\tilde{\alpha}})$. 
\end{remark}

Let $b=\{b_{ij} \mid 1\leq j \leq i \leq n\}$ be a family of \emph{Boolean independent elements} such that
\begin{itemize}
    \item for any $1  \leq i \leq n$, $b_{ii}$ is a Bernoulli element with $\varphi(b_{ii}) =0$ and $\varphi(b_{ii}^2):= \sigma_i$.
    \item for any $1\leq j < i \leq n$,  $b_{ij}$ is a $\eta$-circular element with $ \varphi(b_{ij})=0$ and  $\varphi(b_{ij}b_{ij}^*):=\alpha_{ij}$ and $\varphi(b_{ij}^\ast b_{ij}) :=\tilde{\alpha}_{ij}$.
\end{itemize}
Let us set $e_{ii} = \frac{1}{2 \sqrt{n}} E_{ii}$ and $e_{ij}= \frac{1}{\sqrt{n}} E_{ij}$ for  $j<i$, with $(E_{ij})_{1\leq i,j \leq n}$ are the standard matrix units in $M_n(\C)$. Our main interest in this section is the matrix  \begin{equation}\label{Matrix Bernoulli}
 B_{n} = \sum_{1\leq 
 j \leq i \leq n} \big(e_{ij} \otimes b_{ij}+ e_{ij}^* \otimes b_{ij}^* \big),
 \end{equation}
which we will prove  to be a $M_n(\C)$-valued Bernoulli element in the space of matrices over $\mathcal{A}$, see Proposition \ref{prop:matrixBernoulli} for the precise statement. To do this,  we need the following lemma whose proof follows similarly to the free case, see  \cite[Section 9.3]{mingospeicherbook}. 

\begin{lemma} \label{lema:booleanmatrix cumulants}
Let $(\mathcal{A},\varphi)$ be a non-commutative probability space. We denote by $\beta_n$ the scalar-valued Boolean cumulants with respect to $\varphi$  and by $\beta^{M_n(\C)}_n$ the operator-valued cumulants with respect to $id \otimes \varphi$. Consider the matrices $A_n^{(k)}=(a^{(k)}_{ij})_{i,j=1}^n\in M_n(\mathcal{A})$. Then the operator-valued cumulants of the family $\{A^{(k)}\}_k$ are given in terms of the cumulants of their entries as follows:

$$\big[\beta^{M_n(\C)}_m\big(A_n^{(r_1)},A_n^{(r_2)},\dots,A_n^{(r_m)}\big)\big]_{ij}=\sum^{n}_{i_2,\ldots,i_m=1} \beta_m\big(a_{ii_2}^{(r_1)},a_{i_2i_3}^{(r_2)},\ldots,a_{i_mj}^{(r_m)}\big). $$
\end{lemma}
With Lemma \ref{lema:booleanmatrix cumulants} in hand, we are ready to describe the OV-distribution of $B_n$.

\begin{proposition} \label{prop:matrixBernoulli}
In the quadruple $(M_n(\A), \tr_n\otimes \varphi, id_n \otimes \varphi, M_n(\mathbb{C}))$, the matrix $B_n$  is an operator valued Bernoulli element over $\D_n$, the subalgebra of diagonal matrices with variance map \[
\eta_n: \D_n \rightarrow \D_n, \qquad D\mapsto \eta_n(D)\]
where for any $D=(d_{ij})_{i,j=1}^n \in \D_n$, 
\[
\big(\eta_n (D) \big)_{i,j} = \delta_{ij} \Big( \sum^{i-1}_{k=1}d_{kk}\alpha_{ik}+d_{ii}\sigma_i+\sum^{n}_{k=i+1}d_{kk}\tilde{\alpha}_{ik}\Big).
\]
In particular, the distribution of $B_n$ with respect to $tr_n\otimes\varphi$ is given by 
\begin{equation}\label{formula:law-B_n}
\mu_{B_n}=\frac{1}{2n} \sum^{n}_{i=1}(\delta_{\sqrt{\lambda_i}}+\delta_{-\sqrt{\lambda_i}}). 
\end{equation}
\end{proposition}
\begin{proof}
The fact that $B_n$ is an operator-valued Bernoulli element follows directly from Lemma \ref{lema:booleanmatrix cumulants}. Indeed, as the diagonal entries are Bernoulli elements and the off-diagonal entries are $\eta$-circular elements that are all Boolean independent up to symmetry, it follows that for any $m \in \mathbb{N}$ and $C, C_1, \dots , C_{m-1} \in M_n(\C)$, we have that $\beta^{M_n(\C)}_m(B_nC_1,\dots,B_{n-1}C_{n-1},B_n)\big]=0$ whenever $m\neq 2$ and $ \beta^{M_n(\C)}_2(B_nC,B_n):=\eta_n(C)$ is given for any $i,j \in [n]$ by
\[\begin{aligned}
\big[\beta^{M_n(\C)}_2(B_nC,B_n)\big]_{ij}&=\sum^{n}_{k,\ell=1} \beta_2(b_{ik}c_{k\ell},b_{\ell j})=\delta_{ij}\sum^{n}_{k=1} c_{kk} \beta_2(b_{ik},b_{ki})
\\ & = \delta_{i,j} \big( \sum^{i-1}_{k=1}c_{kk}\alpha_{ik}+c_{ii}\sigma_i+\sum^{n}_{k=i+1}c_{kk}\tilde{\alpha}_{ik}\big).
\end{aligned}\]
As any $M_n(\C)$-valued joint Boolean cumulants of order larger than $2$ is equal to $0$, we deduce that $B_n$ is an operator-valued Bernoulli element over $M_n(\C)$. It suffices to remark that the covariance map $\eta_n$ sends the subalgebra of diagonal matrices $\D_n \subset M_n(\C)$ into itself, to see that $B_n$ is in fact an operator-valued Bernoulli element over $\D_n$. We omit the proof of this statement in the Boolean case, which can be proven in the same way as its free analogue in \cite[Theorem 3.1]{Nica-shly-speicher-OV-distributions}.  

To prove  the last part of the proposition, we note that by Lemma \ref{Lem: Bernoulli}, we have
\[\tr_n\otimes \varphi[B_n^{2k}]=\tr_n\otimes \varphi[\eta(1)^k]=\frac{1}{n}\sum\limits_{i=1}^n \lambda_i^k \quad \text{and} \quad \tr_n\otimes \varphi[B_n^{2k+1}]=0.\] On the other hand, denoting by $\delta_\lambda$ the Dirac measure at $\lambda$, we can easily check that  the $m$-th moment of the probability measure $\mu=\frac{1}{2n}\sum\limits_{i=1}^n(\delta_{\sqrt{\lambda_i}}+\delta_{-\sqrt{\lambda_i}})$ is equal to $\tr_n\otimes \varphi[B_n^{m}]$, which proves \eqref{formula:law-B_n}. 

\end{proof}
\begin{example} \label{example: OVmatrixBernoullis}
We show in the following examples how we can explicitly compute the spectral distribution of $B_n$ using Proposition \ref{prop:matrixBernoulli} and find their limiting distribution, as $n$ tends to infinity. We consider  the cases when the entries are identically distributed and when they have a variance profile. 
\begin{enumerate}
    \item  Identically distributed entries: For any $1\leq j\leq i \leq n$,  let $\sigma_i=\sigma$, $\alpha_{ij}=\alpha/n$ and $\tilde{\alpha}_{ij}=\tilde{\alpha}/n$, and assume without loss of generality that  $0<\tilde{\alpha}\leq\alpha$.  
In this case, $$\lambda_i=\sigma+ \frac{i-1}{n}\alpha+ \frac{n-i}{n}\tilde{\alpha}=\sigma+ \tilde{\alpha}-\frac{\alpha}{n}+\frac{i}{n}(\alpha-\tilde{\alpha}). $$  Now as $\lambda_{i+1}-\lambda_{i}=(\alpha-\tilde{\alpha})/{n}$, we notice that  the $\lambda_i$'s are equally spaced with $\lambda_1=\sigma+\frac{n-1}{n}\tilde{\alpha}$ and $\lambda_{n}=\sigma+\frac{n-1}{n}\alpha$. Thus, in the limit as $n\rightarrow \infty$, 
$$\mu_{X_n}:=\frac{1}{n}\sum_i\delta_{\lambda_i}\to \text{Uniform}\big(\sigma+\tilde{\alpha},\ \sigma+\alpha\big). $$ 
One can also see this by computing the cumulative distribution function of $X_n$ for any $\omega \in \R$,
\[
P(X_n\leq w) = \frac{w-(\sigma+\tilde{\alpha})+\alpha/n}{\alpha-\tilde{\alpha}}  \longrightarrow \frac{w-(\sigma+\tilde{\alpha})}{\alpha-\tilde{\alpha}} \text{ as }n\to \infty.
\]
From this, one can easily show that  $\mu_{B_n}:=\frac{1}{2n} \sum^{n}_{i=1}(\delta_{\sqrt{\lambda_i}}+\delta_{-\sqrt{\lambda_i}}) \xrightarrow[n\rightarrow \infty]{} \mu$ with density
$$\mu(dt)=\frac{|t|}{(\alpha-\tilde{\alpha})}dt,  \quad |t|\in(\sqrt{\sigma+\tilde{\alpha}},\sqrt{\sigma+\alpha}).$$
Moreover, in the case where $\tilde{\alpha}=\alpha$, $B_n$ is a scalar Bernoulli element with distribution $$\mu_{B_n}=\frac{1}{2}\delta_{\sqrt{\sigma+{\frac{n-1}{n}}\alpha}}+\frac{1}{2}\delta_{-\sqrt{\sigma+{\frac{n-1}{n}}\alpha}},$$ which converges to  $\frac{1}{2}\delta_{\sqrt{\sigma+\alpha}}+\frac{1}{2}\delta_{-\sqrt{\sigma+\alpha}}.$
\item Variance profile: Assume for any $1\leq j\leq i \leq n$,  $\alpha_{ij}=\tilde{\alpha}_{ij}=|i-j|/n^2$ and $\sigma_i=0$. Then in this case, it is easy to see that $$\lambda_i=\frac{1}{n^2}\sum^{n}_{k=1}|k-i|=\frac{(n-i+1)(n-i)}{2n^2}+\frac{i(i-1)}{2n^2}.$$
This means that for $t\in[0,1]$, $\lambda_{[tn]}\to \frac{1}{2} -t(1-t)$. Thus if $\rho$ is the limiting distribution of $\frac{1}{n}\sum_i\delta_{\lambda_i}$, it satisfies that $\rho((1/4,(t^2+1)/4))=t$, for $t\in(0,1)$. From where the the density of $\rho$ is given by $\rho(dt)=\frac{4}{\sqrt{4t-1}}dt$ for $t\in (1/4,1/2)$ and thus 
$$\mu(dt)=\frac{2|t|}{\sqrt{4 t^2 - 1}} \quad |t|\in(1/2,\sqrt{1/2}).$$
\end{enumerate}

\end{example}

\subsection{Wigner Matrices with Boolean entries and variance profile}\label{section:Wigner-Booelean}

Consider the $n \times n$ operator-valued matrix $A_n$ defined by 
\begin{equation}
\label{eq-BooleanWigner}    
A_n = \sum_{1\leq j \leq i \leq n} \big( e_{ij} \otimes a_{ij} + e_{ij}^* \otimes a_{ij}^* \big)
\end{equation}
where we recall that $e_{ii} = \frac{1}{2 \sqrt{n}} E_{ii}$ and $e_{ij}= \frac{1}{\sqrt{n}} E_{ij}$ for  $j<i$ with $(E_{ij})_{1\leq i,j \leq n}$ denoting the standard matrix units in $M_n(\C)$. 

Let $a= \{a_{ij} \mid 1\leq j \leq i \leq n\} $ be a family of elements of $\A$ that are Boolean independent, centered with respect to $\varphi$ and are such that $a_{ii}= a_{ii}^*$, $\varphi[a_{ii}^2]:= \sigma_i$ for all $i$ and  $\varphi[a_{ij}^2]=0$,  $\varphi[a_{ij}a_{ij}^*]:=\alpha_{ij}$ and $\varphi[a_{ij}^\ast a_{ij}] :=\tilde{\alpha}_{ij}$ for all $j < i $. Note that the entries of $A_n$ are assumed to be, up to symmetry, Boolean independent but do not need to be identically distributed. We will give quantitative estimates on the distribution of such matrices in terms of Cauchy transforms. Our results generalize the ones in Popa and Hao \cite{popahao2019} in the case of C*-algebras, since we do not impose a condition on the convergence of moments of order larger than $4$ for the entries and more important we allow having a variance profile in the matrix.

With this aim, we apply Theorem \ref{theo:boolean-monotone} in the quadruple $(M_n(\A),  $ $ \tr_n\otimes \varphi, id_n \otimes \varphi, M_n(\mathbb{C}))$ associated with $(\A, \varphi)$.

\begin{theorem}\label{theo:op-matrices}
Let $A_n$ be an operator-valued Wigner matrix described above and $B_n$ an operator-valued Bernoulli $B_n$ as in \eqref{Matrix Bernoulli}. Then for any $z\in \mathbb{C}^+$, 
\begin{equation} \label{eq:distance wigner matrices}
 \big|(\tr_n \otimes \varphi ) [G_{A_n}(z)] - (\tr_n \otimes \varphi )  [G_{B_n}(z)]\big| \leq   \frac{16}{\Imm(z)^4} \Big(\max_{ j \leq i} \|a_{ij}\|^3\Big)  \frac{1}{\sqrt{n}}.
\end{equation}
\end{theorem}

\begin{proof}
Let $b=\{b_{ij} \mid 1\leq j \leq i \leq n\}$ be a family of \emph{Boolean independent elements} such that
\begin{itemize}
    \item for any $1  \leq i \leq n$, $b_{ii}$ is a Bernoulli element with $\varphi[b_{ii}] =0$ and $\varphi[b_{ii}^2]= \varphi[a_{ii}^2]$,
    \item for any $1\leq j < i \leq N$,  $b_{ij}$ is a $\eta$-diagonal circular element with $ \varphi[b_{ij}]=0$ and  $\varphi[b_{ij}b_{ij}^*] =\varphi[a_{ij} a_{ij}^*]:=\alpha_{ij}$ and $\varphi[b_{ij}^\ast b_{ij}] =\varphi[a_{ij}^\ast a_{ij}]:=\tilde{\alpha}_{ij}$.
\end{itemize}
Without loss of generality, $b$ can be assumed to be Boolean independent from $a$. Now set $S_N$ to be the $N\times N$ operator-valued Wigner matrix whose entries are given by the family $b$, i.e.
 \[
 B_n = \sum_{1\leq j \leq i \leq n} \big(e_{ij} \otimes b_{ij}+ e_{ij}^* \otimes b_{ij}^* \big).
 \]
 Finally, set $x_{ij} = e_{ij} \otimes a_{ij} + e_{ij}^* \otimes a_{ij}^* $ and $y_{ij} = e_{ij} \otimes b_{ij} + e_{ij}^* \otimes b_{ij}^*$ for all $1\leq j\leq i\leq n$. By an application of Lemma \ref{lema:booleanmatrix cumulants}, one sees that, when lifted to matrices, Boolean independence is preserved, in the sense that independence of the entries $\{a_{ij} , b_{ij} \mid 1\leq j \leq i \leq n\}$ with respect to $\varphi$ implies independence  with amalgamation over $M_n(\C)$ of $\{x_{ij} , y_{ij} \mid 1\leq j \leq i \leq n\}$ with respect to $\id_n \otimes \varphi$. Moreover, the variables are centered and the moments of second order match: for any $b \in M_n(\C)$, 
\[
(\id_n \otimes \varphi)[x_{ij}] = (\id_n \otimes \varphi)[y_{ij}]=0
\quad \text{and} \quad 
(\id_n \otimes \varphi)\big[ x_{ij} b x_{ij}\big] = (\id_n \otimes \varphi) \big[ y_{ij} \,  b \,  y_{ij} \big]. 
\]
Having this in hand and setting $N=n(n+1)/2$, we follow the same steps as Theorem \ref{theo:boolean-monotone}. By adopting its notation, we get for any $z \in \C^+$,
\[
\big|(\tr_n \otimes \varphi ) [G_{A_n}(z)] - (\tr_n \otimes \varphi )  [G_{B_n}(z)]\big| 
\leq \sum_{1\leq j \leq i \leq n} \Big(\big|(tr_n\otimes\varphi)[W_{ij}(z)]\big| + \big|(tr_n\otimes\varphi)[\widetilde{W}_{ij}(z)]\big|\Big),
\]
where 
\[\begin{aligned}
W_{ij}(z)&= G_{\uz_{ij}}(z) x_{ij}   G_{\uz^0_{ij}}(z) x_{ij}   G_{\uz^0_{ij}}(z)x_{ij} G_{\uz^0_{ij}}(z),
\\
\widetilde{W}_{ij}(z)&= G_{\uz_{i-1,j}(z)} y_{ij} G_{\uz^0_{ij}}(z) y_{ij} G_{\uz^0_{ij}}(z)y_{ij} G_{\uz^0_{ij}}(z).
\end{aligned}\]
We start by controlling the term:
\begin{align*}
\big|(tr_n \otimes &\varphi)\big[  G_{\uz_{ij}}(z) (e_{ij}\otimes a_{ij})   G_{\uz^0_{ij}}(z) (e_{ij}\otimes a_{ij})   G_{\uz^0_{ij}}(z)(e_{ij}\otimes a_{ij}) G_{\uz^0_{ij}}(z) \big] \big| 
\\ 
&\leq  \|G_{\uz_{ij}}(z)\| \|G_{\uz^0_{ij}}(z)\| (tr_n\otimes\varphi)
 \big[ \big| (e_{ij}\otimes a_{ij})   G_{\uz^0_{ij}}(z) (e_{ij}\otimes a_{ij})   G_{\uz^0_{ij}}(z)(e_{ij}\otimes a_{ij})\big| \big]
 \\& \leq \frac{1}{n\sqrt{n}} \|G_{\uz_{ij}}(z)\| \|G_{\uz^0_{ij}}(z)\| (tr_n\otimes\varphi)
 \big[ \big|  E_{ij} \otimes a_{ij}[G_{\uz_{ij}^0}(z)]_{ji}a_{ij}[G_{\uz_{ij}^0}(z)]_{ji}a_{ij} \big|  \big]
 \\& \leq \frac{1}{n^2\sqrt{n}} \|G_{\uz_{ij}}(z)\| \|G_{\uz^0_{ij}}(z)\|^3 \|a_{ij}\|^3
 \\&\leq \frac{1}{n^2\sqrt{n}}\frac{1}{\Imm(z)^4} \|a_{ij}\|^3.
\end{align*}
As the other terms are treated in the same way, we get that 
\[
\begin{aligned}
\big|(tr_n\otimes\varphi)[W_{ij}(z)]\big| &\leq 8 \frac{1}{n^2\sqrt{n}} \frac{1}{\Imm(z)^4} \|a_{ij}\|^3,
\\
\big|(tr_n\otimes\varphi)[\widetilde{W}_{ij}(z)]\big| &\leq 8 \frac{1}{n^2\sqrt{n}} \frac{1}{\Imm(z)^4} \|b_{ij}\|^3.
\end{aligned} 
\]
Summing over $1\leq j \leq i \leq n$, we finally get for any $z \in \C^+$,
\[
\big|(\tr_n \otimes \varphi ) [G_{A_n}(z)] - (\tr_n \otimes \varphi )  [G_{B_n}(z)]\big| \leq   \frac{8}{\Imm(z)^4} \max_{ j \leq i }\Big(\|a_{ij}\|^3+ \|b_{ij}\|^3\Big)  \frac{1}{\sqrt{n}}.
\]
As the $b_{ij}$'s are $\eta$-diagonal elements with the same variance as the $a_{ij}$'s, then by Remark \ref{Remark-eta-circular} we have that $\|b_{ij}\|=\max \{\sqrt{\alpha_{ij}}, \sqrt{\widetilde{\alpha}_{ij}}\}$.
To obtain \ref{eq:distance wigner matrices} it remains to notice that $||a_{ij}||^2=||a_{ij}a_{ij}^*||\geq \varphi[a_{ij}a_{ij}^*]=\alpha_{ij}
$ and $||a_{ij}||^2=||a_{ij}^*a_{ij}||\geq \varphi[a_{ij}^*a_{ij}]=\widetilde{\alpha}_{ij}$.

\end{proof}

Combining Theorem \ref{theo:op-matrices} with Theorem \ref{theo:OV_Bernoulli_comparison} we can get limit theorem for Wigner Matrices with Boolean entries.

\begin{proposition}
Let $A_n$ be as in \eqref{eq-BooleanWigner}  and suppose that $||a_{ij}||\leq C$, for some constant $C$ independent of $N$.
Furthermore let $\eta_n(b)= (id_n \otimes \varphi)[A_nbA_n]$ and suppose that $\eta_n(1)$ has distribution $\rho_n$ with respect $tr\otimes \varphi$. If, $\rho_n$ to converges to $\rho$ in distribution, then, the distribution of $A_n$ with respect to $tr\otimes \varphi$ converges to the unique symmetric measure $\mu$ such that $\mu^{(2)}=\rho$.
\end{proposition}

\begin{proof}

Recall that convergence in distribution is equivalent to uniform convergence of the Cauchy transform on compact sets. Let $\mu_{A_n}$ be the distribution of $A_n$ with respect to $tr\otimes \varphi$. We will prove that for any $\epsilon>0$, and any $K$ a compact sets of $\mathbb{C}^+$ there is an $N$  such that, for any $n>N$,
\[
\big|\mathcal{G}_{\mu_{A_n}}(z)- \mathcal{G}_\mu(z)\big| \leq \epsilon, \quad \text{for all }z\in K.
\]
To do this, let $B_n$ be an operator-valued Bernoulli as in Theorem \ref{theo:op-matrices} with distribution $\rho_n$ and let $B$ be an operator-valued  Bernoulli random variable with variance map $\eta$ such that $\eta(1)$ has distribution $\rho$. Since $\rho_n$ converges weakly to $\rho$, then by Corollary \ref{cor OVBernoulli}, $\mu_{B_n}$ converges to $\mu_B$, weakly. Hence, there exists an integer $N_1>1$ such that, for all $n>N_2$,
$$\big|\mathcal{G}_{\mu_{B_n}}(z)-\mathcal{G}_{\mu_{B}}(z)\big|\leq \epsilon/2, \quad \text{for all }z\in K.$$
On the other hand, equation \eqref{eq:distance wigner matrices} yields that there exists $N_2>1$ such that for all  $n>N_2$, \begin{equation}
\label{triangle2} |\mathcal{G}_{\mu_{A_n}}(z)-\mathcal{G}_{\mu_{B_n}}(z)\big|\leq \epsilon/2, \quad \text{for all }z\in K.
\end{equation}
The result follows by taking $N=max(N_1,N_2)$.
\end{proof}

Let us finish by considering the case when $\{a_{ij}\}_{n\geq i>j\geq0}$ are identically distributed, which corresponds to Theorem 5.1 of Popa and Hao\cite{popahao2019}.

\begin{example} 
Let $\{a_{ij}\}_{n\geq i>j\geq0}$ be a family of Boolean independent identically distributed random variables in some $C^*$-probability space $(A,\varphi)$ and let $A_n$ be the $n \times n$ matrix defined by 
\begin{equation}
\label{eq-BooleanWigner}    
A_n = \sum_{1\leq j \leq i \leq n} \big( e_{ij} \otimes a_{ij} + e_{ij}^* \otimes a_{ij}^* \big).
\end{equation}
If we denote by  $\varphi(a_{ij}a_{ij}^\ast)=\alpha$ and $\varphi(a_{ij}^*a_{ij})=\tilde{\alpha}$, then $\rho_n:=\eta_n(b)= (id_n \otimes \varphi)[A_nbA_n]=(id_n \otimes \varphi)[B_nbB_n]$ where $B_n$ is the matrix-valued Bernoulli as in part (i) from  Example \ref{example: OVmatrixBernoullis} with $\sigma=0$. Thus $\rho_n$ converges to the uniform distribution $ \text{Uniform}\big(\tilde{\alpha},\alpha\big)$ and hence we conclude by Theorem \ref{theo:op-matrices} the limiting distribution of $A_n$ with respect to $\tr_n\otimes\varphi$ is 
$$\mu(dt)=\frac{|t|}{(\alpha-\tilde{\alpha})}dt,  \quad |t|\in(\sqrt{\tilde{\alpha}},\sqrt{\alpha}).$$
In the case $\alpha=\tilde\alpha$,  $\rho_n$ converges to $\delta_\alpha$ and $\mu_{A_n}\to\frac{1}{2}\delta_{\sqrt{\alpha}}+\frac{1}{2}\delta_{\sqrt{\alpha}}$.
\end{example}

\section*{Acknowlegments}

OA was supported by CONACYT Grant CB-2017-2018-A1-S-9764 and by the SFB-TRR 195 “Symbolic Tools in Mathematics and their Application” of the German Research Foundation (DFG).  This project started when OA and MB were visiting Saarland University. They would like to thank Roland Speicher for his hospitality. 

\appendix
\section{Operator-valued Infinitesimal Products}\label{Appendix:OVI}
This appendix is dedicated to the study of operator-valued infinitesimal products of operator-valued infinitesimal probability spaces. Note that the OV free product was introduced in \cite{speicher1998combinatorial}, while the OV Boolean and monotone cases have been studied in \cite{popa2013non} and \cite{popa2008combinatorial}, respectively. We also refer the reader to  \cite[Chapter 7]{Jekel-notes} for a nice summary on the subject. However, this is less studied in the infinitesimal setting where only the (scalar-valued) infinitesimal free product was constructed in \cite{fevrier2010infinitesimal}. In this appendix, we give an algebraic construction of the operator-valued infinitesimal product algebraically in the free, Boolean, and monotone settings.  

First of all, let us recall the following result (see \cite[Theorem 4.3.1]{Jekel-notes}):  
suppose that $H_1,\dots,H_n$ are  Hilbert-$\B$-bimodules and $\xi_j$ is a unit vector in $H_j$ for each $j=1,\dots,n$. In addition, for each $j$, we define the map $E_j[b]=\langle \xi_j,b\xi_j\rangle$ for all $b\in\B(H_j)$. Then there exists a Hilbert $\B$-bimodule $H$ and a unit vector $\xi$ in $H$, and an injective $*$-homomorphism $\rho_j:\B(H_j)\to \B(H)$ such that 
$$
E[\rho_j(b)]=E_j[b] \text{ for each }b\in \B(H_j).
$$
Moreover, $\rho_1(\B(H_1)),\rho_2(\B(H_2)),\dots,\rho_n(\B(H_n))$ are freely/Boolean/monotone independent with respect to $E$. 

Suppose that for each $j$, there exists a self-adjoint $\B$-bimodule linear map $E'_j:\B(H_j)\to \B$ that is completely bounded with $E_j'[1]=0$.  
Here we let $\A$ be be the unital $C^*$-algebra generated by $\rho_1(\B(H_1)),\dots,\rho_n(\B(H_n))$. 
For each notion of independence, our aim is to construct a self-adjoint $\B$-bimodule linear map $E'$ on $\A$ that vanishes on $\B$ and for which $\rho_1(\B(H_1)),\rho_2(\B(H_2)),$ $\dots,\rho_n(\B(H_n))$ are infinitesimally independent in the OVI probability space $(\A,E,E',\B)$. Note that for each $j$, we shall define $E'$ on $\rho_j(\B(H_j))$ as follows: 
$$
E'[\rho_j(b)]:=E'_j[b] \text{ for all }b\in \B(H_j).
$$
To describe the construction of $E'$ on $\A$, it is sufficient to show how $E'$ is defined on terms of the following form 
$$
\rho_{j_1}(\B(H_{j_1}))\rho_{j_2}(\B(H_{j_2}))\dots \rho_{j_k}(\B(H_{j_k}))
$$
where $j_1,\dots,j_k\in [n]$ and $j_1\neq j_2\neq \dots \neq j_k$. \\
\textbf{The Infinitesimally Free Case:}
$$
E'[\rho_{j_1}(b_1)\dots \rho_{j_k}(b_k)]:= \sum\limits_{s=1}^k E[\rho_{j_1}(b_1)\dots \rho_{j_{s-1}}(b_{s-1})E'[\rho_{j_s}(b_s)]\rho_{j_{s+1}}(b_{s+1})\dots \rho_{j_k}(b_k)]
$$
where $b_1\in H_{j_1},\dots,b_k\in H_{j_k}$ for some $j_1\neq j_2\neq \dots \neq j_k$ that $E_{j_1}(b_1)=\dots=E_{j_k}(b_k)=0.$
\\
\textbf{The Infinitesimally Boolean Case:} 
$$
E'[\rho_{j_1}(b_1)\dots \rho_{j_k}(b_k)] := \sum\limits_{s=1}^k E[\rho_{j_1}(b_1)]\dots E[\rho_{j_{s-1}}(b_{s-1})]E'[\rho_{j_s}(b_s)]E[\rho_{j_{s+1}}(b_{s+1})]\dots E[\rho_{j_k}(b_k)]
$$
where $b_1\in H_{j_1},\dots,b_k\in H_{j_k}$ for some $j_1\neq j_2\neq \dots \neq j_k$.\\
\textbf{The Infinitesimally Monotone Case:}
\begin{eqnarray*}
E'[\rho_{j_1}(b_1)\dots \rho_{j_k}(b_k)]&:=&E'[\rho_{j_1}(b_1)\dots \rho_{j_{s-1}}(b_{s-1})E[\rho_{j_s}(b_s)]\rho_{j_{s+1}}(b_{s+1})\dots \rho_{j_k}(b_k)] \\
&&+
E[\rho_{j_1}(b_1)\dots \rho_{j_{s-1}}(b_{s-1})E'[\rho_{j_s}(b_s)]\rho_{j_{s+1}}(b_{s+1})\dots \rho_{j_k}(b_k)] 
\end{eqnarray*}
whenever $b_1\in H_{j_1},\dots,b_k\in H_{j_k}$ and $j_{s-1}<j_s>j_{s+1}$. 

Note that the fact that $E'$ is well-defined on $\A$ in the Boolean case is obvious. For the free case, we observe that $\A$ can be written as follows: 
$$\A=\B 1\oplus \bigoplus\limits_{n\geq 0} \bigoplus\limits_{i_1\neq \cdots\neq i_n}
\rho_{i_1}(\B(H_{i_1}))^\circ\rho_{i_2}(\B(H_{i_2}))^\circ\dots \rho_{i_n}(\B(H_{i_n}))^\circ
$$
where $\rho_{i_j}(\B(H_{i_j}))^\circ$ is the set of centered elements of $\rho_{i_j}(\B(H_{i_j}))^\circ$. Due to the direct sum decomposition, we only need to define it on $\B 1$ and each of $\rho_{i_1}(\B(H_{i_1}))^\circ\rho_{i_2}(\B(H_{i_2}))^\circ\dots \rho_{i_n}(\B(H_{i_n}))^\circ$. 
Therefore, for the free case, $E'$ is well-defined on $\A$. In addition, $E'$ is also well-defined on $\A$ for the monotone case by recursively computing the mixed moments. 

Since $E_j$ and $E_j'$ are both self-adjoint for each $j$, we can easily see that $E'$ is self-adjoint for each case. Moreover, following the definition of $E'$, it is clear that $\rho_1(\B(H_1)),\rho_2(\B(H_2)),\dots,$ $\rho_n(\B(H_n))$ are infinitesimally freely/ Boolean/monotone independent with respect to $(E,E')$. However, it is not clear whether we have any continuity property for $E'$. Hence, we only have the operator-valued infinitesimal products in the algebraic sense. 

\nocite{*}

\bibliographystyle{abbrv}
\bibliography{ref}
\end{document}